%% file: main.tex
\documentclass[11pt]{article}

\usepackage[a4paper,margin=1in]{geometry}
\usepackage{amsmath,amssymb,amsthm,mathtools,bm,mathrsfs}
\numberwithin{equation}{section}
\usepackage{enumitem}
\usepackage{graphicx}
\usepackage{booktabs}
\usepackage{placeins}
\usepackage{flafter}

\usepackage[ruled]{algorithm}
\usepackage{algpseudocode}
\usepackage{microtype}
\usepackage{xcolor}
\usepackage[
  backend=biber,
  style=numeric,
  sorting=nyt,
  giveninits=true,
  maxbibnames=99,
  doi=true,
  isbn=false,
  url=false,
  eprint=true
]{biblatex}

\newcommand{\revisionarxividentifier}[1]{\nolinkurl{#1}}
\DeclareFieldFormat{eprint:arxiv}{%
  arXiv\addcolon\space
  \ifhyperref
    {\href{https://arxiv.org/abs/#1}{%
       \revisionarxividentifier{#1}%
       \iffieldundef{eprintclass}{}
         {\addspace\texttt{\mkbibbrackets{\thefield{eprintclass}}}}}}
    {\revisionarxividentifier{#1}%
     \iffieldundef{eprintclass}{}
       {\addspace\texttt{\mkbibbrackets{\thefield{eprintclass}}}}}}

\DeclareFieldAlias{eprint:arXiv}{eprint:arxiv}

\DeclareNameAlias{default}{given-family}
\DeclareNameAlias{sortname}{given-family}

\DeclareFieldFormat[article,inproceedings]{title}{#1}

\renewbibmacro{in:}{%
  \ifentrytype{article}
    {}
    {\printtext{\bibstring{in}\intitlepunct}}}
    
\usepackage{hyperref}
\usepackage[nameinlink,capitalize,noabbrev]{cleveref}

\hypersetup{
  hypertexnames=false,
  hidelinks,
  pdftitle={Polynomial Acceleration of Alternating Projections: Chebyshev Acceleration, Convergence Rates, and Retractions},
  pdfauthor={Dengyu Zheng and Shixiang Chen},
  pdfkeywords={alternating projections, Chebyshev acceleration, stationary polynomial transform, retraction, clean intersection}
}

\newtheorem{theorem}{Theorem}[section]
\newtheorem{lemma}[theorem]{Lemma}
\newtheorem{proposition}[theorem]{Proposition}
\newtheorem{corollary}[theorem]{Corollary}
\theoremstyle{definition}
\newtheorem{assumption}[theorem]{Assumption}
\newtheorem{definition}[theorem]{Definition}
\theoremstyle{remark}
\newtheorem{remark}[theorem]{Remark}

\Crefname{assumption}{assumption}{assumptions}
\Crefname{assumption}{Assumption}{Assumptions}
\Crefname{definition}{definition}{definitions}
\Crefname{definition}{Definition}{Definitions}
\Crefname{lemma}{lemma}{lemmas}
\Crefname{lemma}{Lemma}{Lemmas}
\Crefname{proposition}{proposition}{propositions}
\Crefname{proposition}{Proposition}{Propositions}
\Crefname{corollary}{corollary}{corollaries}
\Crefname{corollary}{Corollary}{Corollaries}
\Crefname{remark}{remark}{remarks}
\Crefname{remark}{Remark}{Remarks}
\Crefname{appendix}{appendix}{appendices}
\Crefname{appendix}{Appendix}{Appendices}

\newcommand{\E}{\mathcal E}
\newcommand{\R}{\mathbb R}
\newcommand{\M}{\mathcal M}
\newcommand{\Mone}{\mathcal M_1}
\newcommand{\Mtwo}{\mathcal M_2}
\newcommand{\Id}{\mathrm{Id}}
\newcommand{\dist}{\operatorname{dist}}
\newcommand{\spec}{\operatorname{spec}}
\newcommand{\spr}{\operatorname{spr}}
\newcommand{\Retr}{\operatorname{Retr}}
\newcommand{\II}{\mathrm{II}}
\newcommand{\holderexp}{\alpha_{\mathrm H}}

\title{Polynomial Acceleration of Alternating Projections:
Chebyshev Acceleration, Convergence Rates, and Retractions}
\author{Dengyu Zheng\thanks{School of Mathematical Sciences, Key Laboratory of the Ministry of Education for Mathematical Foundations and Applications of Digital Technology, University of Science and Technology of China, Hefei, Anhui, China
(\href{mailto:zq_renius919@mail.ustc.edu.cn}{\texttt{zq\_renius919@mail.ustc.edu.cn}},
\href{mailto:shxchen@ustc.edu.cn}{\texttt{shxchen@ustc.edu.cn}}).}
\and Shixiang Chen\footnotemark[1]}
\date{\small \today \vspace{-3em}}

\begin{document}
\maketitle

\begin{abstract}
Alternating projections (AP) near a clean intersection of smooth manifolds can
converge slowly when the Friedrichs angle is small.
We develop the \emph{Fixed-Point Manifold Acceleration (FPMA) framework} for
{local convergence and polynomial acceleration of} smooth maps
that fix a manifold pointwise.
{Under suitable regularity and uniform normal spectral stability,
FPMA establishes local R-linear convergence and a smooth limiting map inducing
a local retraction.
At fixed points, consistent polynomial acceleration with $p(1)=1$ preserves
the selected limit's first-order response to normal initial perturbations
whenever the transformed normal blocks are uniformly spectrally stable.
Under sufficient smoothness, we give a curvature criterion for second-order
induced retractions and show that stable consistent polynomial acceleration
preserves both the quadratic retraction expansion and this criterion.
{As a special case, we establish new local R-linear
convergence guarantees for fixed-degree Chebyshev-accelerated exact AP
and second-order induced retractions.
For uniform Friedrichs angle $\theta_F\in(0,\pi/2)$, acceleration strictly
improves the local convergence rate per AP evaluation.
The large-degree limiting spectral factor at fixed $\theta_F$ predicts
$O(\theta_F^{-1})$ evaluations for fixed relative error reduction as
$\theta_F\downarrow0$, versus $O(\theta_F^{-2})$ for ordinary AP.}}
{The framework also covers} inexact, relaxed, and generalized AP
{under their respective regularity and stability conditions.}
\end{abstract}



\section{Introduction}\label{sec:introduction}

The method of alternating projections (AP) for two embedded submanifolds
$\Mone$ and $\Mtwo$ of a finite-dimensional Euclidean space $\E$ generates the
sequence
\[
 y_{k+1}=S_{\rm AP}(y_k),\qquad S_{\rm AP}=P_{\Mone}P_{\Mtwo}.
\]
{Here} $P_{\M_i}$ is the locally single-valued metric projection onto
$\M_i$, and $\M:=\Mone\cap\Mtwo$ is the intersection.
The method originates with von Neumann's subspace theorem
\cite[Theorem~13.7]{vonNeumann1950} and belongs to the classical projection
framework \cite{Bregman1967,BauschkeBorwein1996,Deutsch2001}.
{For nonconvex sets, local convergence has been established
under geometric regularity conditions}
\cite{LewisMalick2008,DrusvyatskiyIoffeLewis2015}.
{Write $T_x\mathcal N$ for the tangent space of a manifold
$\mathcal N$ at $x$. Suppose that $\M$ is locally an embedded submanifold
and that the intersection is clean at $x\in\M$, meaning that}
\[
T_x\mathcal M=T_x\mathcal M_1\cap T_x\mathcal M_2.
\]
{Let $c_F(x):=c_F(T_x\M_1,T_x\M_2)$ denote the Friedrichs
cosine. The normal block of the AP derivative at $x$ has spectral radius
$c_F(x)^2$, which approaches one as the Friedrichs angle tends to zero.
AP can therefore converge slowly at small intersection angles.}

{This slow convergence motivates polynomial acceleration.
For a linear iteration with operator $L$, polynomial acceleration replaces
its $m$-step operator $L^m$ by $p_m(L)$, where the degree $\deg p_m$ is at
most $m$. We call $p_m$ consistent if $p_m(1)=1$, so that $p_m(L)$
preserves vectors fixed by $L$. On a real spectral interval excluding
$1$, the Chebyshev minimax polynomial minimizes the maximum magnitude
among consistent polynomials of degree at most $m$.}
Chebyshev polynomial acceleration is classical in linear iterative methods
\cite{GolubVarga1961,Varga2000,Saad2003,Saad2025}, and periodic Chebyshev
relaxation has also been studied for {locally contractive}
fixed-point iterations
\cite{WadayamaTakabe2021}.
For {AP on} two subspaces, Knyazev et al.
\cite{KnyazevJujunashviliArgentati2010} developed the self-adjoint compression
and optimal-polynomial viewpoint. Bello-Cruz \cite{BelloCruz2026} gave an
explicit full-interval Chebyshev semi-iteration.
{A related benchmark is the optimally tuned generalized
alternating projections (GAP) method, which attains the linear convergence
factor $(1-\sin\theta_F)/(1+\sin\theta_F)$, where $\theta_F$ is the
Friedrichs angle. Its corresponding local rate extends to smooth manifolds}
\cite{FaltGiselsson2017,FaltGiselsson2023}.
Other acceleration mechanisms include line search, extrapolation,
supporting hyperplanes, and vector extrapolation
\cite{GearhartKoshy1989,BauschkeDeutschHundalPark2003,
BauschkeCombettesKruk2006,Pang2015,BarreTaylorAspremont2022}.
{To place the polynomial-operator viewpoint in a common
local nonlinear setting, we develop the \emph{Fixed-Point Manifold
Acceleration (FPMA) framework} for smooth maps that fix a manifold
pointwise and their consistent polynomial transformations.
For exact AP near clean intersections of smooth manifolds, we establish
{new} local {R-linear} convergence
{results for fixed-degree} Chebyshev-accelerated maps and prove strict improvement of the normal
spectral bound per AP evaluation under the assumptions stated below.
Iterations on subspaces with a fixed polynomial are included as a linear
special case, while inexact AP, relaxed AP, and full GAP are also covered
under their respective regularity and stability conditions.}

{A smooth AP limiting map induces a local retraction at
$x\in\M$ by mapping each sufficiently small tangent vector
$\eta\in T_x\M$ to the iteration limit from $x+\eta$.
This supports manifold optimization when the intersection projection
$P_\M$ is difficult to compute but the component projections
$P_{\M_i}$ are tractable.}
{Absil and Malick
\cite[Proposition~5 and Section~4]{AbsilMalick2012} show that
$(x,\eta)\mapsto P_\M(x+\eta)$, $\eta\in T_x\M$, is a local retraction
and generalize this construction to projection-like maps.
Chen et al. \cite{ChenHeHuang2026} use tangent--normal decompositions of
iteration errors and their derivatives to establish smooth limiting
retractions for AP-type maps and their second-order property under the
stated assumptions. Their NewtonSLRA analysis also gives fourth-order
agreement with metric projection for tangent perturbations in structured
low-rank approximation (SLRA).
For Regularized Newton--SLRA with fixed regularization, Zheng and Chen
\cite[Appendix B]{ZhengChen2026} establish a $C^1$ local limiting map and,
under additional smoothness, a second-order induced retraction
{with fourth-order agreement with metric projection for
tangent perturbations}.}
{FPMA recovers the corresponding local limiting-map and
retraction results for smooth stationary maps under its regularity and
normal stability assumptions, while allowing more general tangent--normal
coupling and extending the analysis to polynomial acceleration.
On a nonisolated solution set, this extension must address both the rate
of convergence and the dependence of the selected limit on the initial
point, since the latter can affect the accuracy of the induced retraction.}

{To analyze these two aspects together, FPMA uses the
derivative of a smooth map $S$ fixing $\M$ pointwise.
At $x\in\M$, decompose an ambient direction as $\eta=\eta_T+\eta_N$,
with $\eta_T\in T_x\M$ and $\eta_N\in N_x\M$, where $N_x\M$ is the
normal space. Let $DS(x)$ denote the Fr\'echet derivative of $S$.
Its blocks $A_x:N_x\M\to N_x\M$ and $B_x:N_x\M\to T_x\M$ satisfy}
\[
    {DS(x)[\eta_T+\eta_N]
      =\eta_T+B_x\eta_N+A_x\eta_N,\qquad}
    DS(x)=\begin{pmatrix}I&B_x\\0&A_x\end{pmatrix}.
\]
{Since $S$ fixes $\M$ pointwise, $DS(x)$ has tangent block $I$
and zero tangent-to-normal block; $A_x$ controls normal convergence,
while $B_x$ maps normal perturbations to tangent displacements.
Uniform normal spectral stability means
$\rho_*:=\sup_{x\in K}\operatorname{spr}(A_x)<1$ on a compact set $K\subset\M$,
where $\operatorname{spr}$ denotes the spectral radius.
Under the regularity assumptions of \Cref{thm:limiting-map}, it yields
local R-linear convergence and a smooth limiting map}
\[
    {\Psi_S(y):=\lim_{k\to\infty}S^k(y)\in\M.}
\]
{The local R-linear bound admits every factor
$\widehat q\in(\rho_*,1)$, with a prefactor and a neighborhood of $K$
that may depend on $\widehat q$.
At fixed points near $K$, the limiting map's derivative is}
\[
    {D\Psi_S(x)[\eta_T+\eta_N]=\eta_T+C_x\eta_N,\qquad
    C_x=B_x(I-A_x)^{-1}:N_x\M\to T_x\M.}
\]
{Thus $C_x$ gives the accumulated first-order effect of a
normal initial perturbation on the selected limit.
It also determines the tangent spaces of stable fibers, the sets of
initial points converging to the same limit, connecting FPMA with
classical invariant-manifold theory \cite{HirschPughShub1977}
(\Cref{rem:phase-stable-fibers}).
}

{Regarding retraction properties, let $\eta\in T_x\M$ be
sufficiently small.
Under the assumptions of \Cref{prop:base-tangent-jet}, the induced
retraction has the expansion}
\[
{
\Retr_x^S(\eta)=\Psi_S(x+\eta)
=x+\eta+\frac12\bigl(\II_x(\eta,\eta)-C_x\II_x(\eta,\eta)\bigr)
+O(\|\eta\|^3).
}
\]
{Here $\II_x$ is the second fundamental form:
$\tfrac12\II_x(\eta,\eta)$ is the normal curvature term, while
$-\tfrac12C_x\II_x(\eta,\eta)$ is the tangential correction from limit
selection.
The retraction is therefore second-order at $x$ precisely when
$C_x\II_x(\eta,\eta)=0$ for every $\eta\in T_x\M$
(\Cref{thm:second-order}); this criterion is more general than $C_x=0$
because it only requires $C_x$ to annihilate curvature-generated normal
directions.}

{These expressions link limit selection and retraction
accuracy through $C_x$, whose behavior under polynomial acceleration we
examine next.
For a consistent polynomial $p$, consider a smooth map $H_p$ fixing
$\M$ pointwise with $DH_p(x)=p(DS(x))$.
The normal block becomes $p(A_x)$.
{We show that uniform spectral stability of these blocks
leaves $C_x$ unchanged for every consistent polynomial and every
first-order realization. Thus polynomial acceleration preserves the
selected limit's first-order response to normal initial perturbations
and, under the required smoothness, the retraction's quadratic expansion
and second-order criterion.}
Our contributions are as follows:}

\begin{enumerate}[label=\textup{(\roman*)}]
\item \emph{{FPMA and polynomial acceleration.}}
{For smooth maps fixing $\M$ pointwise, we establish local
R-linear convergence and a smooth limiting map under uniform normal
spectral stability, allowing $B_x\ne0$ and normal blocks with
$\|A_x\|\ge1$ (\Cref{thm:limiting-map}).
For consistent polynomial transformations, the same conclusions follow
from uniform spectral stability of $p(A_x)$, and the limiting-map
derivative is independent of the polynomial and its realization
(\Cref{thm:main-polynomial}).}

\item \emph{{Chebyshev acceleration of AP.}}
{For exact AP near a clean intersection, take a compact
set $K\subset\M$ with $0<\bar c_F:=\sup_{x\in K}c_F(x)<1$.
We construct affine-combination and factorized Chebyshev maps using $m$
AP evaluations per cycle and compute their exact uniform normal spectral
radius $q_m$.
The bound $q_m^{1/m}<\bar c_F^2$ for every $m\ge1$ gives strict spectral
improvement per AP evaluation (\Cref{thm:chebyshev-exact-ap}).
Taking $m\to\infty$ at fixed $\theta_F:=\arccos\bar c_F$ yields a
limiting spectral factor that predicts $O(\theta_F^{-1})$ evaluations
for a fixed relative error reduction as $\theta_F\downarrow0$, compared
with {the $O(\theta_F^{-2})$ prediction for ordinary AP}.
Under the respective assumptions, strict spectral improvement per
evaluation of the chosen base map also holds for inexact AP and
two-parameter relaxed AP with component relaxation parameters in $(0,1]$
(\Cref{cor:chebyshev-iap,thm:chebyshev-gap}).}

\item \emph{{Induced retractions and geometric accuracy.}}
{Under the stated smoothness assumptions, we characterize
second-order induced retractions by the curvature criterion above and
show that stable consistent polynomial acceleration preserves their
quadratic expansion (\Cref{thm:second-order,thm:main-polynomial}).
For exact AP, inexact AP with component maps matching metric projections
in value and derivative at the intersection, and the relaxed and GAP
maps studied below, the coupling block $B_x$ vanishes.
Under their respective stability conditions, this gives $C_x=0$.
Their limiting maps, including those of stable consistent polynomial
transformations, therefore induce second-order retractions
(\Cref{cor:chebyshev-exact-ap-geometry,cor:chebyshev-iap,cor:gap-retraction,cor:outer-gap-retraction,cor:polynomial-gap}).}
\end{enumerate}

{Sections 2--4 develop FPMA and polynomial acceleration;
Sections 5 and 6 apply the framework to exact AP and its inexact, relaxed,
and generalized variants, respectively, and Section 7 reports experiments.
Complete proofs of the three supporting lemmas in Section 3 appear in
\Cref{app:limiting-map-proofs}; the other appendices provide GAP and
zero-coupling analyses and reproducibility details.}

\section{Geometric and spectral preliminaries}
\label{sec:fixed-manifold-framework}

Throughout, $\E$ is a finite-dimensional Euclidean space with inner product
$\langle\cdot,\cdot\rangle$ and induced norm $\|\cdot\|$, and
$r\in\{1,2\}$. {All coordinate spaces carry the Euclidean
norm; for linear and multilinear maps, $\|\cdot\|$ denotes the corresponding
induced operator norm. The symbols $A^*$ and $A^T$ denote the Euclidean
adjoint and matrix transpose, respectively.}
For a set $\mathcal N\subset\E$, $P_{\mathcal N}$ denotes
the metric (nearest-point) projection wherever it is locally single-valued;
when $\mathcal N$ is a linear subspace, $P_{\mathcal N}$ is the orthogonal
projector. {The distance to $\mathcal N$ is
$\dist(y,\mathcal N):=\inf_{z\in\mathcal N}\|y-z\|$.}
The symbol $D$ denotes the Fr\'echet derivative.
{For an embedded submanifold $\mathcal N$,}
$T_x\mathcal N$ and $N_x\mathcal N$ denote the tangent and normal spaces at
$x\in\mathcal N$. {We write $P_{T_x\mathcal N}$ and
$P_{N_x\mathcal N}$ for the corresponding orthogonal projectors.}
For a linear operator $A$, let $\spec(A)$ be its spectrum
and {$\spr(A):=\max\{|\lambda|:\lambda\in\spec(A)\}$}
its spectral radius, with $\spr(A):=0$ on a zero-dimensional space. The
symbols $I$ and $\Id$ denote, respectively, a linear identity operator and
an identity map on the indicated space.

Fix $0<\holderexp\leq1$. The class $C^{k,\holderexp}_{\rm loc}$ consists of
maps whose derivatives through order $k$ are continuous and whose $k$th
derivative is locally H\"older continuous with exponent $\holderexp$.
{At every point, there are a neighborhood $U$ and a constant
$L>0$ such that $\|D^k f(x)-D^k f(y)\|\leq L\|x-y\|^{\holderexp}$ for
$x,y\in U$.}
When $\holderexp=1$, this means that $D^k f$ is locally Lipschitz.
{We use $C^{k,\holderexp}$ for this local class throughout
and interpret manifold regularity in local charts.} {The tubular-coordinate
representation used below is stated in \Cref{lem:tubular-coordinates}.}

Following \cite{ChenHeHuang2026}, we call an iteration
\emph{alternating-projection-type} if it is assembled from exact, inexact, or
relaxed projection steps near an intersection manifold.

\subsection{Clean intersection, transversality, and Friedrichs angle}
\label{sec:geometric-intersections}

Let $\Mone,\Mtwo\subset\E$ be embedded submanifolds and suppose that
$\M:=\Mone\cap\Mtwo$ is an embedded submanifold near the points under
consideration. The intersection is \emph{transversal} at $x\in\M$ if
{$T_x\Mone+T_x\Mtwo=\E$.}
The intersection is \emph{clean} if
\begin{equation}
\label{eq:clean-intersection}
T_x\M=T_x\Mone\cap T_x\Mtwo,
\qquad x\in\M.
\end{equation}
Every transversal intersection is clean, but a clean intersection need not be
transversal.

\begin{definition}[Friedrichs angle and principal angles]
\label[definition]{def:friedrichs-principal}
Let $U,V\subset\E$ be linear subspaces, set $W=U\cap V$, and define
{$U_0:=U\cap W^\perp$ and $V_0:=V\cap W^\perp$.}
The Friedrichs cosine {is}
\begin{equation}
\label{eq:friedrichs-cosine}
c_F(U,V):=
\sup\{\langle u,v\rangle:
u\in U_0,\ v\in V_0,\ \|u\|=\|v\|=1\},
\end{equation}
with $c_F(U,V)=0$ if either reduced space is zero.
{The Friedrichs angle is
$\theta_F(U,V):=\arccos c_F(U,V)$.
The principal angles between $U_0$ and $V_0$ are}
$0\leq\theta_1\leq\cdots\leq\theta_s\leq\pi/2$, where
$s=\min\{\dim U_0,\dim V_0\}$. They are defined recursively by maximizing
$\langle u,v\rangle$ over unit vectors {$u\in U_0$ and
$v\in V_0$} orthogonal to the previously selected principal vectors.
{If $s>0$, then $\theta_F(U,V)=\theta_1$ and
$c_F(U,V)=\cos\theta_1$.} If $s=0$, the list is empty and
$c_F(U,V)=0$, $\theta_F(U,V)=\pi/2$.
{These reduced-space angles correspond to the principal angles between
$U$ and $V$ after removing the zero angles associated with $W$; see}
\cite[Definitions~3.1--3.2 and Proposition~3.3]
{BauschkeBelloCruzNghiaPhanWang2016}.
\end{definition}

{For a compact set $K$ in a clean manifold intersection,
the identity \eqref{eq:clean-intersection} and continuity of the tangent projectors
make $x\mapsto c_F(T_x\Mone,T_x\Mtwo)$ continuous. Its pointwise values
are strictly less than one, so}
\begin{equation}
\label{eq:friedrichs-uniform-bound}
{\sup_{x\in K}c_F(T_x\Mone,T_x\Mtwo)<1.}
\end{equation}

\subsection{{Tubular coordinates and local graphs}}
\label{sec:tubular-coordinates}

{To analyze points near $\M$, we separate their position
along the manifold from their normal displacement using tubular
coordinates. For expansions at a fixed base point, a local graph
describes how $\M$ departs from its tangent plane and makes the
contribution of curvature explicit.}

\paragraph{{Tubular coordinates.}}
{Using the normal bundle
$N\M:=\{(x,n):x\in\M,\ n\in N_x\M\}$, we represent a nearby ambient
point as $y=x+n$, with $(x,n)\in N\M$.
The lemma below gives a unique such pair near the zero section over $K$
for each $y$ in a sufficiently small neighborhood of $K$ in $\E$;
both the base point $x$ and the normal vector $n$ are determined by $y$.}

\begin{lemma}[{Tubular neighborhood}]
\label{lem:tubular-coordinates}
{Let $\M\subset\E$ be a $C^{r+1,\holderexp}$ embedded
submanifold, let $K\subset\M$ be compact, and let $U\subset\E$ be an
open neighborhood of $K$. There are an open neighborhood
$\mathcal D$ of $\{(x,0):x\in K\}$ in $N\M$ and an open neighborhood
$V\subset U$ of $K$ such that}
\begin{equation}
\label{eq:tubular-map}
{
\Theta:\mathcal D\to V,\qquad \Theta(x,n)=x+n,
}
\end{equation}
{is a $C^{r,\holderexp}$ diffeomorphism. The nearest-point
projection on $V$ is single-valued and $C^{r,\holderexp}$, with}
\begin{equation}
\label{eq:tubular-projection}
{
P_\M(x+n)=x,\qquad \dist(x+n,\M)=\|n\|,
\qquad (x,n)\in\mathcal D.
}
\end{equation}
\end{lemma}

\begin{proof}
{For the classical smooth tubular-neighborhood theorem,
see \cite[Theorem~6.24]{Lee2013}.
By \cite[Lemma~4]{AbsilMalick2012}, $P_\M$ is
single-valued and $C^r$ near each $x\in K$.
Local $C^{r,\holderexp}$ normal frames make $\Theta$ a
$C^{r,\holderexp}$ map, with
$D\Theta(x,0)[a,b]=a+b$ for $a\in T_x\M$ and $b\in N_x\M$.
The inverse-function theorem \cite{KrantzParks2013} gives local
$C^{r,\holderexp}$ inverses; after shrinking the neighborhoods,
continuity of $P_\M$ and nearest-point orthogonality identify them as
$y\mapsto(P_\M(y),y-P_\M(y))$.
Uniqueness of $P_\M$ makes these inverses agree on overlaps, so a finite
cover of $K$, with the ambient neighborhoods restricted to $U$, yields
$\mathcal D$ and $V$.
The inverse formula gives the stated projection regularity and distance
identity.}
\end{proof}

{{To write the tubular representation in Euclidean
coordinates, near each point of $K$ choose a local}
$C^{r+1,\holderexp}$ parametrization
$\chi:\mathcal O\to\chi(\mathcal O)\subset\M$, where
$\mathcal O\subset\R^{\dim\M}$ is open and $\chi(\mathcal O)$ is open
in $\M$. Let $\{\nu_i(u)\}_{i=1}^c$ be a $C^{r,\holderexp}$ orthonormal
frame of $N_{\chi(u)}\M$, where $c:=\operatorname{codim}\M$.
For a pair $(x,n)\in\mathcal D$ with $x\in\chi(\mathcal O)$, its local
coordinates $(u,v)$ are defined by $x=\chi(u)$ and
$n=\sum_{i=1}^c v_i\nu_i(u)$.
{Thus $u\in\mathcal O\subset\R^{\dim\M}$ parametrizes
the base point $x$, while $v\in\R^c$ gives the coefficients of the
normal vector $n$ in this frame. Together, $(u,v)$ are local coordinates
for the ambient point $y=x+n$, with total dimension
$\dim\M+c=\dim\E$.}
After shrinking $\mathcal O$, the map $(x,n)\mapsto x+n$ in
\eqref{eq:tubular-map} gives a Euclidean tubular chart, also denoted by
$\Theta$, on an open neighborhood of $\mathcal O\times\{0\}$ in
$\mathcal O\times\R^c$, with}
\begin{equation}
\label{eq:tubular-chart}
{
\Theta(u,v)=x+n=\chi(u)+\sum_{i=1}^c v_i\nu_i(u).
}
\end{equation}
{The zero section $v=0$ represents the local patch $\chi(\mathcal O)$.
Orthonormality together with
\eqref{eq:tubular-projection} gives
$\|v\|=\|n\|=\dist(\Theta(u,v),\M)$.
By compactness, finitely many such patches $\chi_\ell(\mathcal O_\ell)$
cover $K$; their associated tubular charts are denoted by $\Theta_\ell$.
For subsets $A,B$ of the same Euclidean space or manifold,
$\overline A$ denotes the closure of $A$ in that space, namely the set
of all limits of sequences in $A$.
We write $A\Subset B$ when $\overline A$ is compact and contained in $B$.
Once compact inner domains are fixed, the tubular charts
and their inverses have uniform bounds for derivatives through order $r$
and for the H\"older constants of the $r$th derivatives.
The coordinate domains and normal radii used for the iteration estimates
are chosen in the proof of \Cref{lem:contracting-iterate} in
\Cref{app:proof-contracting-iterate}.}

\paragraph{{Graph representation at a fixed base point.}}
{For local expansions, fix $x\in\M$. The orthogonal
splitting $\E=T_x\M\oplus N_x\M$ represents a nearby ambient point
$y\in\E$ as}
\begin{equation}
\label{eq:base-point-splitting}
{
y=x+\eta_T+\eta_N,\qquad
\eta_T=P_{T_x\M}(y-x),\quad \eta_N=P_{N_x\M}(y-x).
}
\end{equation}
{Throughout, vectors and derivative blocks at each base
point are written in this tangent--normal order.
For points on $\M$ near the fixed $x$, the tangential projection
$z\mapsto P_{T_x\M}(z-x)$ restricted to $\M$ has derivative
$I_{T_x\M}$ at $x$. The inverse-function theorem therefore gives
a neighborhood of $x$ in $\M$ of the form
\cite[Chapter~5, pp.~98--107]{Lee2013}}
\begin{equation}
\label{eq:local-manifold-graph}
{
\{x+\xi+g(\xi):\xi\in U_x\subset T_x\M\},\qquad
 g:U_x\to N_x\M.
}
\end{equation}
{Here $U_x$ is a neighborhood of $0$, and
$g(0)=0$, $Dg(0)=0$. The graph map has the same $C^k$ or
$C^{k,\holderexp}$ regularity as $\M$, for $k\ge1$; in particular,
$\|g(\xi)\|=O(\|\xi\|^2)$ when $\M$ is $C^2$.}

\subsection{Retractions}
\label{sec:retraction-preliminaries}

In manifold optimization, retractions replace the often expensive exponential
map with a tractable map having the same first-order behavior at the zero section
\cite{AbsilMahonySepulchre2008,ChenHeHuang2026}.

The tangent bundle and its zero section are
{$T\M:=\{(x,\eta):x\in\M,\ \eta\in T_x\M\}$ and
$0_{T\M}:=\{(x,0):x\in\M\}$.}
Under the $C^{r+1,\holderexp}$ assumptions used below, $T\M$ has the regularity
needed for the local $C^r$ maps used in the analysis.

\begin{definition}[Local and second-order retractions]
\label[definition]{def:retraction}
{Let $\Omega\subset\M$ be open, and let
$\mathcal D\subset T\M|_\Omega$ be an open neighborhood of
$\{(x,0):x\in\Omega\}$. A \emph{local retraction} on $\M$ over $\Omega$
is a $C^1$ map $\Retr:\mathcal D\to\M$}
such that, writing $\Retr_x(\eta):=\Retr(x,\eta)$,
\begin{equation}
\label{eq:retraction-conditions}
\Retr_x(0)=x,
\qquad
D\Retr_x(0)={I_{T_x\M}}
\end{equation}
{for every $x\in\Omega$. Equivalently,
$\Retr_x(\eta)=x+\eta+o(\|\eta\|)$ as $\eta\to0$ in $T_x\M$.}
If $\Retr$ is $C^2$, it is a \emph{second-order retraction} at
{$x\in\Omega$} if, for every $\eta\in T_x\M$, the curve
{$c(t):=\Retr_x(t\eta)$}
has zero covariant acceleration at $t=0$.
\end{definition}

{Let $\nabla$ denote the Levi-Civita connection of the
metric induced on $\M$ by $\E$.}
We use the convention that the second fundamental form
$\II_x:T_x\M\times T_x\M\to N_x\M$ is the normal component of ambient
acceleration. The Gauss decomposition for a curve $c(t)\subset\M$ with
$c(0)=x$ and $\dot c(0)=\eta$ is
\begin{equation}
\label{eq:gauss-acceleration}
  \ddot c(0)
  =
  \nabla_{\dot c}\dot c(0)+\II_x(\eta,\eta).
\end{equation}
{Here $\ddot c(0)$ denotes the second derivative of $c$
viewed as a curve in the ambient Euclidean space $\E$.}
{By \eqref{eq:gauss-acceleration}, zero covariant acceleration is equivalent to}
$P_{T_x\M}\ddot c(0)=0$, or, equivalently,
\begin{equation}
\label{eq:second-order-retraction-jet}
  \Retr_x(t\eta)
  =
  x+t\eta+\frac12t^2\II_x(\eta,\eta)+o(t^2)
  \qquad(t\to0).
\end{equation}
Consequently, \eqref{eq:second-order-retraction-jet} is equivalent to
agreement with the exponential map through second order; see
\cite{AbsilMahonySepulchre2008,ChenHeHuang2026}.

{Our retractions are constructed from limiting maps.
If a map $F$ fixes $\M$ pointwise and its iterates converge locally to
$\M$, we write}
\[
  \Psi_F(y):=\lim_{k\to\infty}F^k(y){\in\M.}
\]
{Restricting this map to tangent perturbations gives the
candidate retraction}
\begin{equation}
\label{eq:limiting-map-retraction-preview}
  \Retr_x^F(\eta):=\Psi_F(x+\eta).
\end{equation}
{We establish conditions under which this construction
defines a local retraction, determine its first- and second-order behavior,
and examine whether acceleration preserves these properties.}
{For AP-type maps, \cite{ChenHeHuang2026} establishes
this construction, including joint smoothness near the zero section of the tangent bundle and
the second-order property under the stated assumptions.} {The limiting map of Regularized
Newton--SLRA with fixed positive regularization is treated in}
{\cite[Appendix~B]{ZhengChen2026}}.

\subsection{Tangent--normal model for a fixed-point manifold}

{To analyze convergence toward $\M$ and the tangential
effect of normal perturbations, we decompose the derivative at fixed points.}
\begin{assumption}[Fixed-point-manifold structure]\label[assumption]{ass:geometry}
Let $\M\subset\E$ be a $C^{r+1,\holderexp}$ embedded submanifold, where
$r\in\{1,2\}$, and let $K\subset\M$ be compact. Let $S:U\to\E$ be a
$C^{r,\holderexp}$ mapping on an open neighborhood $U$ of $K$.
{Assume that $S(x)=x$ for every $x\in\M\cap U$.}
\end{assumption}

{Under \Cref{ass:geometry}, define the derivative blocks
at $x\in\M\cap U$ by}
\begin{equation}
\label{eq:normal-tangent-blocks}
{A_x:=P_{N_x\M}DS(x)|_{N_x\M},\qquad
B_x:=P_{T_x\M}DS(x)|_{N_x\M}.}
\end{equation}
{These are continuous bundle maps by the regularity of $S$ and the tangent
and normal projectors. For an ambient direction
$\eta=\eta_T+\eta_N\in\E=T_x\M\oplus N_x\M$, the fixed-point property gives}
\begin{equation}
\label{eq:tangent-normal-block}
{DS(x)[\eta]=\eta_T+B_x\eta_N+A_x\eta_N,\qquad}
DS(x)=
\begin{pmatrix}
I&B_x\\
0&A_x
\end{pmatrix}.
\end{equation}
{The block $A_x$ describes the first-order normal response
to a normal perturbation, while $B_x$ describes its tangential effect.
Since $S$ fixes $\M$ pointwise, $DS(x)[\eta_T]=\eta_T$; hence the tangent
block is $I$ and the tangent-to-normal block is $0$.
In the AP-type setting of \cite[Lemma~4.6]{ChenHeHuang2026}, the
tangent--normal analysis gives $B_x=0$ and uniform contraction of the
normal blocks in the Euclidean operator norm.
Building on this analysis, FPMA allows $B_x\ne0$, so normal perturbations
can produce first-order tangential drift.
We also establish convergence and limiting-map regularity under uniform
normal spectral stability, defined in the next subsection, which allows
stable normal blocks with $\|A_x\|\ge1$.}

\subsection{Uniform normal stability}

\begin{definition}[Uniform normal spectral stability]\label[definition]{def:stability}
{A normal block $A_x:N_x\M\to N_x\M$ is
\emph{spectrally stable} if $\spr(A_x)<1$, equivalently, if $A_x^k\to0$
as $k\to\infty$. A continuous family of these blocks is}
\emph{uniformly spectrally stable} on $K$ if
\begin{equation}
\label{eq:normal-spectral-stability}
  \rho_*:=\sup_{x\in K}\spr(A_x)<1.
\end{equation}
\end{definition}

\begin{remark}[Uniform power stability]\label[remark]{rem:bundle-powers}
Let $K\subset\M$ be compact, and let
{$A_x:N_x\M\to N_x\M$}
be a continuous bundle endomorphism over $K$. If
{\eqref{eq:normal-spectral-stability} holds},
then, for every $q\in(\rho_*,1)$, there exists $C_q>0$ such that
\begin{equation}
\label{eq:normal-power-bound}
\|A_x^k\|\leq C_q q^k,
\qquad x\in K,\qquad k\geq0.
\end{equation}
To see this, complexify the bundle and cover $K$ by finitely many bundle
trivializations. Compactness gives a uniform bound on $(zI-A_x)^{-1}$ for
$|z|=q_0\in(\rho_*,q)$, and the Cauchy integral formula yields
{\eqref{eq:normal-power-bound}}
\cite{HornJohnson2013,DahlquistHuangLeVeque1983}.
\end{remark}

\section{Limiting maps and induced retractions}
\label{sec:limiting-map}

{We now use uniform normal spectral stability to establish
local R-linear convergence and a $C^r$ limiting map for a map $F$ fixing
$\M$ pointwise. The limiting map's derivative at fixed points describes
how normal perturbations shift the selected limit. Restricting the limiting map to tangent
perturbations then gives local retractions. We characterize their
second-order property and give a sufficient condition for fourth-order
agreement with metric projection.}

\begin{theorem}[{$C^r$ limiting map under normal spectral stability}]
\label{thm:limiting-map}
\label{thm:phase-map}
{Suppose \Cref{ass:geometry} holds with $S=F$, and let
$A_x,B_x$ be the blocks in \eqref{eq:normal-tangent-blocks}. Assume uniform normal
spectral stability on $K$ in the sense of \Cref{def:stability}, with
$\rho_*$ defined in \eqref{eq:normal-spectral-stability}.
Then there is an open neighborhood $V$ of $K$ with the following properties.}
\begin{enumerate}[label={\textup{(\roman*)}},leftmargin=*]
\item {For every $y\in V$, the sequence
$\{F^k(y)\}_{k\geq0}$ is well defined and converges to}
\begin{equation}
\label{eq:limiting-map-definition}
\Psi_F(y):=\lim_{k\to\infty}F^k(y){\in\M.}
\end{equation}

\item {The map $\Psi_F:V\to\M$ is $C^r$ and fixes $\M\cap V$
pointwise. At every $x\in\M\cap V$, its derivative is}
\begin{equation}
\label{eq:limiting-map-derivative}
D\Psi_F(x)=
\begin{pmatrix}
I&C_x\\
0&0
\end{pmatrix},
\quad
C_x=B_x(I-A_x)^{-1}.
\end{equation}

\item {For every $\widehat q\in(\rho_*,1)$, after shrinking $V$
depending on $\widehat q$ if necessary, there is a constant
$C_{\widehat q}>0$ such that}
\begin{equation}
\label{eq:limiting-map-rate}
\|F^k(y)-\Psi_F(y)\|
\leq C_{\widehat q}\widehat q^{\,k}\dist(y,\M)
{,\qquad y\in V,\quad k\geq0.}
\end{equation}
Thus every $\widehat q\in(\rho_*,1)$ is an admissible local R-linear
convergence factor.
\end{enumerate}
\end{theorem}

\input{limiting_map_proof}

\begin{remark}[Phase and stable fibers]
\label[remark]{rem:phase-stable-fibers}
{In this pointwise-fixed setting, $\Psi_F$ plays the role of an
\emph{asymptotic-phase map}, linking \Cref{thm:limiting-map} to the
classical stable-fiber theory of normally attracting invariant manifolds
\cite{HirschPughShub1977}.
For $x\in\M\cap V$, its stable fiber
$\mathcal W_x^s:=\{y\in V:\Psi_F(y)=x\}$ consists of the initial points
whose iterates converge to $x$.}
{By \eqref{eq:limiting-map-derivative} and compactness of $K$, after
shrinking $V$ around $K$, $\Psi_F:V\to\M$ is a $C^r$ submersion.
Its fibers are therefore local $C^r$ embedded submanifolds of
codimension $\dim\M$, with}
\[
{
T_x\mathcal W_x^s=\ker D\Psi_F(x)
=\{n-C_xn:n\in N_x\M\}.
}
\]
{For $n\in N_x\M$, \eqref{eq:limiting-map-derivative} gives
$\Psi_F(x+tn)=x+tC_xn+o(t)$ as $t\to0$.
Thus $C_x$ describes the first-order tangential response of the
selected limit to normal initial perturbations{; in the
tangent--normal order of \eqref{eq:base-point-splitting}, the fiber tangent
space is the graph of $-C_x:N_x\M\to T_x\M$}.}
{Whenever $y,F(y)\in V$, the identity
$\Psi_F(F(y))=\Psi_F(y)$ shows that $F$ preserves the fiber.}

{For exact AP at a clean intersection, both tangent-space
projectors preserve $N_x\M$, so $B_x=C_x=0$ and
$D\Psi_{\rm AP}(x)=P_{T_x\M}$ (\Cref{lem:exact-linearization}).
The restrictions of $\Psi_{\rm AP}$ to $\M_1$ and $\M_2$
are also submersions near $x$, so}
\[
{
T_x\mathcal W_x^s=N_x\M,
\qquad
T_x(\mathcal W_x^s\cap\M_i)=T_x\M_i\cap N_x\M,
\quad i=1,2.
}
\]
{These are first-order identities: nonlinear fibers may
be curved, and $\Psi_{\rm AP}$ need not coincide with $P_\M$.}
\end{remark}

{The map $\Retr^F$ in
\eqref{eq:limiting-map-retraction-preview} is jointly $C^r$ on a
neighborhood of the zero section over $\M\cap V$ by composition with
$\Psi_F$, and satisfies \eqref{eq:retraction-conditions} by
\eqref{eq:limiting-map-derivative} and the fixed-point property of $\Psi_F$.}
{To obtain the retraction's second-order expansion, we
first estimate the normal displacement of $x+\eta$ caused by curvature.}

\begin{lemma}[Quadratic normal displacement]\label[lemma]{lem:tangent-normal}
{If $\M$ is $C^2$ and $x\in\M$, then, for
$\eta\in T_x\M$ sufficiently small, the nearest point
$z=P_\M(x+\eta)$ and displacement $n=x+\eta-z$ satisfy}
\begin{equation}
\label{eq:tangent-normal-displacement}
n\in N_z\M,\qquad \|n\|=O(\|\eta\|^2).
\end{equation}
{The estimate is uniform for $x$ in compact subsets of $\M$.}
\end{lemma}

\begin{proof}
{Using the $C^2$ graph representation
\eqref{eq:local-manifold-graph} at $x$, compare $x+\eta$ with
$x+\eta+g(\eta)$ to obtain}
\[
  \dist(x+\eta,\M)\leq \|g(\eta)\|=O(\|\eta\|^2).
\]
For sufficiently small $\eta$, the nearest point $z$ is unique. The
first-order optimality condition gives $n=x+\eta-z\in N_z\M$, and
$\|n\|=\dist(x+\eta,\M)$, which proves the claim.
{On compact subsets, finitely many graph neighborhoods
with bounded second derivatives give a uniform radius and constant.}
\end{proof}

\begin{proposition}[Expansion of the induced retraction]
\label[proposition]{prop:base-tangent-jet}
{Under the hypotheses of \Cref{thm:limiting-map} with
$r=2$, the induced retraction has the following expansion.
Uniformly for $x$ in compact subsets of $\M\cap V$ and
$\eta\in T_x\M$ sufficiently small,}
\begin{equation}
\label{eq:base-tangent-jet}
  {\Retr_x^F(\eta)=}\Psi_F(x+\eta)
  =
  x+\eta
  +\frac12\Bigl(
      \II_x(\eta,\eta)
      -C_x\II_x(\eta,\eta)
    \Bigr)
  +O(\|\eta\|^3).
\end{equation}
{The term $\tfrac12\II_x(\eta,\eta)$ is the normal
displacement required by curvature, while
$-\tfrac12C_x\II_x(\eta,\eta)$ is the additional tangential displacement
produced by the choice of the limiting point.}
\end{proposition}

\begin{proof}
{Use the $C^{3,\holderexp}$ graph representation
\eqref{eq:local-manifold-graph} at $x$.}
{Here $D^2g(0)=\II_x$ by \eqref{eq:gauss-acceleration}.} Write the nearest point as
\[
  z=x+{\xi_*}+g({\xi_*}).
\]
Comparing the distance at the nearest point with the distance to the graph
point $x$ gives $\|{\xi_*}-\eta\|\leq\|\eta\|$ and hence
${\xi_*}=O(\|\eta\|)$. The stationarity equation for the squared distance to the
graph is
\begin{equation}
\label{eq:tangent-projection-stationarity}
  {\xi_*}-\eta+Dg({\xi_*})^*g({\xi_*})=0.
\end{equation}
Because $Dg({\xi_*})=O(\|{\xi_*}\|)$ and $g({\xi_*})=O(\|{\xi_*}\|^2)$,
{\eqref{eq:tangent-projection-stationarity} gives}
${\xi_*}=\eta+O(\|\eta\|^3)$. Taylor expansion of $g$ then yields
\[
  g({\xi_*})=\frac12\II_x(\eta,\eta)+O(\|\eta\|^3).
\]
Consequently, for $n:=x+\eta-z$,
\begin{equation}
\label{eq:tangent-projection-expansion}
  z
  =
  x+\eta+\frac12\II_x(\eta,\eta)+O(\|\eta\|^3),
  \qquad
  n
  =
  -\frac12\II_x(\eta,\eta)+O(\|\eta\|^3),
\end{equation}
with $n\in N_z\M$ and $\|n\|=O(\|\eta\|^2)$. Since $\Psi_F$ is
$C^2$, its first derivative is locally Lipschitz. Taylor expansion at the
manifold point $z$ therefore yields
\begin{equation}
\label{eq:normal-phase-expansion}
  \Psi_F(z+n)
  =
  z+D\Psi_F(z)n+O(\|n\|^2)
  =
  z+C_zn+O(\|n\|^2),
\end{equation}
{where $C_z$ is defined in \eqref{eq:limiting-map-derivative}.}
Under the $r=2$ hypotheses, the bundle map $z\mapsto C_z$ is $C^1$ and
hence locally Lipschitz along $\M$ when represented in a fixed $C^1$
local trivialization. Therefore
\begin{equation}
\label{eq:phase-tangential-term}
  C_zn
  =
  -\frac12C_x\II_x(\eta,\eta)+O(\|\eta\|^3).
\end{equation}
Because $\|n\|^2=O(\|\eta\|^4)$,
{substituting \eqref{eq:tangent-projection-expansion},
\eqref{eq:normal-phase-expansion}, and \eqref{eq:phase-tangential-term}}
proves \eqref{eq:base-tangent-jet}. The estimates are uniform on compact subsets.
\end{proof}

\begin{theorem}[{Second-order retractions and fourth-order agreement}]
\label{thm:second-order}
\label{thm:fourth-order-agreement}
{Assume the hypotheses of \Cref{thm:limiting-map} with
$r=2$. The induced retraction has the following properties.}
\begin{enumerate}[label={\textup{(\roman*)}},leftmargin=*]
\item {At $x\in\M\cap V$, it is second-order if and only if}
\begin{equation}
\label{eq:second-order-criterion}
C_x\II_x(\eta,\eta)=0,
\qquad {\text{for every }\eta\in T_x\M.}
\end{equation}
{In particular, $B_x=0$ is a sufficient condition.}

\item {Suppose additionally that $B_z=0$ for every
$z\in\M\cap V$. For every compact $K_0\subset\M\cap V$, there exist
$C>0$ and $\varepsilon>0$ such that}
\begin{equation}
\label{eq:fourth-order-agreement}
\|\Psi_F(x+\eta)-P_\M(x+\eta)\|
\leq C\|\eta\|^4
\end{equation}
{whenever $x\in K_0$, $\eta\in T_x\M$, and
$\|\eta\|<\varepsilon$. In this case, $\Retr^F$ is second-order
throughout $\M\cap V$ and has fourth-order agreement with the projection
retraction.}
\end{enumerate}
\end{theorem}

\begin{proof}
{For \textup{(i)}, \eqref{eq:base-tangent-jet} and
\eqref{eq:gauss-acceleration} give}
\[
{
P_{T_x\M}\left.\frac{d^2}{dt^2}\Retr_x^F(t\eta)\right|_{t=0}
=-C_x\II_x(\eta,\eta).
}
\]
{Thus zero covariant acceleration in every tangent
direction is equivalent to \eqref{eq:second-order-criterion}.
The sufficient condition follows from $B_x=0\Rightarrow C_x=0$
in \eqref{eq:limiting-map-derivative}.}

{For \textup{(ii)}, set $z=P_\M(x+\eta)$ and
$n=x+\eta-z$. For sufficiently small $\eta$, $z\in\M\cap V$, and
\eqref{eq:tangent-normal-displacement} gives
$n\in N_z\M$ and $\|n\|=O(\|\eta\|^2)$.
The neighborhood assumption ensures $C_z=0$ at this projection point.
Consequently,}
\[
{
\begin{aligned}
\Psi_F(x+\eta)-P_\M(x+\eta)
&=\Psi_F(z+n)-z\\
&\overset{\eqref{eq:normal-phase-expansion}}{=}
C_zn+O(\|n\|^2)\\
&=O(\|n\|^2)
\overset{\eqref{eq:tangent-normal-displacement}}{=}
O(\|\eta\|^4).
\end{aligned}
}
\]
{For $x\in K_0$, choose a common perturbation radius
so that $z$ and the segment $z+\tau n$, $0\leq\tau\leq1$, lie in a
fixed compact subset of $V$.
The second derivatives of $\Psi_F$ and the constants in
\eqref{eq:tangent-normal-displacement} are then uniformly bounded.
This gives $C$ and $\varepsilon$ in \eqref{eq:fourth-order-agreement};
the second-order property of $\Retr^F$ follows from \textup{(i)}.
The comparison map $\eta\mapsto P_\M(x+\eta)$ is the standard
second-order projection retraction
\cite[Proposition~5 and Section~4]{AbsilMalick2012}.}
\end{proof}

\begin{remark}[{Geometric accuracy under acceleration}]
{The pointwise criterion
\eqref{eq:second-order-criterion} only requires $C_x$ to annihilate the
span of the curvature vectors $\II_x(\eta,\eta)$.
It is therefore more general than $C_x=0$: other normal directions may
still affect the selected limit.
{If \eqref{eq:second-order-criterion} fails at $x$, the
induced map still satisfies the retraction conditions
\eqref{eq:retraction-conditions}, but is not second-order at $x$.}
This dependence on $C_x$ motivates the polynomial invariance analysis in
\Cref{subsec:polynomial-convergence}.}
\end{remark}

\section{Polynomial acceleration in the FPMA framework}
\label{sec:stationary-polynomial}
\label{sec:chebyshev}

{The limiting-map theorem gives local convergence under uniform normal
spectral stability. We now study how a polynomial transformation changes the normal
blocks and the first-order dependence of the selected limit on the initial
point. We first construct nonlinear maps whose derivatives at fixed points
are $p(DS(x))$, then apply \Cref{thm:limiting-map} under transformed
spectral stability. This leads to Chebyshev minimax choices and comparisons
per evaluation of the base map $S$.}

\subsection{Consistent polynomials and first-order realizations}
\label{subsec:consistent-polynomial}

\begin{definition}[Consistent polynomial and first-order realization]
\label[definition]{def:polynomial-transform}
{Let $p(\lambda)=\sum_{j=0}^m a_j\lambda^j$
be a real polynomial of degree $m\geq0$.}
We call $p$ \emph{consistent} if $p(1)=1$.
{Let $K\subset\M$ be compact and let $S:U\to\E$ satisfy
\Cref{ass:geometry}.} A
$C^{r,\holderexp}$ map $H_p$ defined near $K$ is a
\emph{first-order realization of $p$ for $S$ along $\M$} if
\begin{equation}
\label{eq:first-order-polynomial-realization}
 H_p(x)=x,
 \qquad
 DH_p(x)=p(DS(x)),
 \qquad x\in\M
\end{equation}
on a relative neighborhood of $K$.
\end{definition}

{For a consistent polynomial $p$, the condition $p(1)=1$ preserves
the identity on tangent directions. The following two constructions satisfy
\eqref{eq:first-order-polynomial-realization}.}
\begin{itemize}[leftmargin=*]
\item \emph{{Affine-combination realization}}
{\cite[Section~3.2]{Saad2025}}.
{Define}
\begin{equation}
\label{eq:orbit-affine-transform}
 G_p^{{\rm comb}}(y)
 :=\sum_{j=0}^m a_jS^j(y).
\end{equation}

\item \emph{{Factorized realization}}
{\cite[Section~3.2]{Saad2025}}.
{Alternatively, given real numbers $\omega_1,\ldots,\omega_m$ such that}
\begin{equation}
\label{eq:consistent-factorization}
 p(\lambda)=\prod_{j=1}^m
 \bigl(1-\omega_j+\omega_j\lambda\bigr),
\end{equation}
{define $R_{\omega_j}(y):=(1-\omega_j)y+\omega_jS(y)$,
$j=1,\ldots,m$. Each map forms an affine combination of the current
point and its image under $S$; this update is called a relaxation step for $S$.
Composing these steps gives}
\begin{equation}
\label{eq:factored-cycle}
 G_p^{\rm cyc}{(y)}
 :={(}R_{\omega_m}\circ\cdots\circ R_{\omega_1}{)(y)}.
\end{equation}
\end{itemize}

{When $m=0$, the empty product is $1$ and the empty composition is
$\Id$. Both constructions use $m$ evaluations of $S$ per application and
require only the composite map $S$. They agree when $S$ is affine;
for nonlinear $S$, they have the same derivative at fixed points but may
differ beyond first order.}

\begin{assumption}[Local well-definedness]\label[assumption]{ass:well-defined}
The realization $H_p$ is $C^{r,\holderexp}$ on an open neighborhood
$U_p$ of $K$.
\end{assumption}

{For \eqref{eq:orbit-affine-transform}, this assumption follows by
shrinking $U_p$ so that $S^j(U_p)\subset U$ for $j=0,\dots,m$.
For \eqref{eq:factored-cycle}, restrict the finitely many intermediate
relaxed maps to a common local domain. Such restrictions are possible
because each intermediate map fixes $\M$.}

\begin{lemma}[{Polynomial transformation of the tangent--normal blocks}]
\label[lemma]{lem:poly-linearization}
Let $p$ be a consistent polynomial.
{Under \Cref{ass:geometry}, the maps
\eqref{eq:orbit-affine-transform} and \eqref{eq:factored-cycle}, whenever
locally well defined, satisfy \eqref{eq:first-order-polynomial-realization}.
For any first-order realization of a consistent polynomial,}
\begin{equation}
\label{eq:polynomial-linearization}
DH_p(x)=
\begin{pmatrix}
I&B_{p,x}\\
0&p(A_x)
\end{pmatrix}_{T_x\M\oplus N_x\M},
\end{equation}
{at fixed points $x\in\M$ near $K$, with}
\begin{equation}
\label{eq:polynomial-coupling-block}
B_{p,x}=B_xq_p(A_x).
\end{equation}
{Here $q_p(t):=(1-p(t))/(1-t)$ is a polynomial because $p(1)=1$.}
\end{lemma}

\begin{proof}
{At a fixed point $x$, $S^j(x)=x$ and $D(S^j)(x)=DS(x)^j$.
Thus the affine-combination map satisfies}
\[
G_p^{{\rm comb}}(x)=\sum_{j=0}^m a_jx=p(1)x=x,
\qquad
DG_p^{{\rm comb}}(x)=\sum_{j=0}^m a_jDS(x)^j=p(DS(x)).
\]
{Each relaxed factor also fixes $x$, so the chain rule and
\eqref{eq:consistent-factorization} give}
\[
 DG_p^{\rm cyc}(x)
 =\prod_{j=1}^m\bigl((1-\omega_j)I+\omega_jDS(x)\bigr)
 =p(DS(x)).
\]
{These derivative factors commute, since each is a polynomial in
$DS(x)$; the nonlinear maps need not commute away from $\M$.
It remains to compute the blocks of $p(DS(x))$. From
\eqref{eq:tangent-normal-block}, induction gives}
\[
(DS(x))^j=
\begin{pmatrix}
I&B_x\sum_{\ell=0}^{j-1}A_x^\ell\\
0&A_x^j
\end{pmatrix},
{\qquad j\geq0,}
\]
{where the sum is empty for $j=0$. The upper-right block is therefore}
\[
B_x\sum_{j=0}^ma_j\sum_{\ell=0}^{j-1}A_x^\ell
=B_xq_p(A_x),
\]
{using $1-p(t)=(1-t)q_p(t)$. The diagonal blocks are
$p(1)I=I$ and $p(A_x)$, which proves \eqref{eq:polynomial-linearization}.}
\end{proof}

{If $1\notin\spec(A_x)$, the identity
$(1-t)q_p(t)=1-p(t)$ gives the equivalent expressions}
\begin{equation}
\label{eq:polynomial-resolvent}
q_p(A_x)=(I-p(A_x))(I-A_x)^{-1},
\qquad
B_{p,x}=B_x(I-p(A_x))(I-A_x)^{-1}.
\end{equation}

\subsection{Invariance of the limiting-map derivative}
\label{subsec:polynomial-convergence}

{By \eqref{eq:polynomial-linearization},
\Cref{thm:limiting-map} applies to the polynomial-accelerated map $H_p$
when its normal blocks $p(A_x)$ are uniformly spectrally stable.
The following theorem specializes its conclusions and shows that the
limiting-map derivative at fixed points is independent of the polynomial
and its realization.}

\begin{theorem}[{Stable polynomial acceleration}]
\label{thm:main-polynomial}
{Under \Cref{ass:geometry,ass:well-defined}, let $p$ be consistent
and $H_p$ a first-order realization whose normal blocks satisfy}
\begin{equation}
\label{eq:transformed-normal-stability}
\rho_p:=\sup_{x\in K}\spr(p(A_x))<1.
\end{equation}
{Then all conclusions of \Cref{thm:limiting-map} hold with
$F=H_p$ and $\rho_*=\rho_p$. The resulting $C^r$ limiting map
$\Psi_p:V\to\M$ satisfies, for $x\in\M\cap V$,}
\begin{equation}
\label{eq:polynomial-limiting-derivative}
\Psi_p|_{\M\cap V}=\Id,\qquad
D\Psi_p(x)=
\begin{pmatrix}
I&C_x\\
0&0
\end{pmatrix},
\quad
C_x=B_x(I-A_x)^{-1}.
\end{equation}
{In particular, this derivative is independent of the polynomial and
its realization, provided \eqref{eq:transformed-normal-stability} holds.}
{For every $\widehat q\in(\rho_p,1)$, after shrinking $V$ if
necessary, there is a constant $C_{\widehat q}>0$ such that}
\begin{equation}
\label{eq:polynomial-linear-rate}
\|H_p^k(y)-\Psi_p(y)\|
\leq C_{\widehat q}\widehat q^{\,k}\dist(y,\M),
\qquad y\in V{,\quad k\geq0.}
\end{equation}
\end{theorem}

\begin{proof}
{Apply \Cref{thm:limiting-map} using
\eqref{eq:polynomial-linearization} and
\eqref{eq:transformed-normal-stability}. The normal-to-tangent block of
$D\Psi_p(x)$ is $B_{p,x}(I-p(A_x))^{-1}$.
By continuity and compactness of $K$, shrink $V$ so that
$\spr(p(A_x))<1$ for $x\in\M\cap V$.
Since $p(1)=1$, spectral mapping also gives $1\notin\spec(A_x)$.
Thus \eqref{eq:polynomial-resolvent} yields}
\begin{equation}
\label{eq:polynomial-limiting-invariance}
\begin{aligned}
B_{p,x}(I-p(A_x))^{-1}
&={B_x(I-p(A_x))(I-A_x)^{-1}(I-p(A_x))^{-1}}\\
&=B_x(I-A_x)^{-1}.
\end{aligned}
\end{equation}
\end{proof}

\begin{remark}[{Retraction properties under polynomial acceleration}]
{The invariance in \eqref{eq:polynomial-limiting-invariance} concerns
derivatives at fixed points; limiting maps may differ away from $\M$.}
{Their stable fibers have the same tangent spaces at fixed
points by \Cref{rem:phase-stable-fibers}.
For $r=2$, \Cref{prop:base-tangent-jet} also gives the same quadratic
retraction expansion and second-order criterion
\eqref{eq:second-order-criterion}.
If, in addition, $B_z=0$ throughout a relative neighborhood of $K$ in
$\M$, \eqref{eq:polynomial-coupling-block} and
\Cref{thm:second-order}\textup{(ii)} give fourth-order agreement with
metric projection.}
\end{remark}

\begin{corollary}[{Fourth-order agreement} {of induced retractions}]
\label[corollary]{cor:realization-agreement}
{Suppose \Cref{thm:main-polynomial} applies with $r=2$
to first-order realizations $H_p$ and $\widetilde H_{\widetilde p}$ of
consistent polynomials $p$ and $\widetilde p$ for the same base map $S$.
Let $\Psi_p$ and $\widetilde\Psi_{\widetilde p}$ be their limiting maps on a
common open neighborhood $V$ of $K$.
Then, uniformly on compact subsets of $\M\cap V$,}
\begin{equation}
\label{eq:polynomial-realization-agreement}
 \|\Psi_p(x+\eta)-{\widetilde\Psi_{\widetilde p}}(x+\eta)\|
 =O(\|\eta\|^4),
 \qquad \eta\in T_x\M.
\end{equation}
\end{corollary}

\begin{proof}
{Let $z=P_\M(x+\eta)$ and $n=x+\eta-z$.
Equation \eqref{eq:tangent-normal-displacement} gives
$n\in N_z\M$ and $\|n\|=O(\|\eta\|^2)$.
Both limiting maps fix $z$, and
\eqref{eq:polynomial-limiting-derivative} gives
$D\Psi_p(z)=D{\widetilde\Psi_{\widetilde p}}(z)$.
Since both maps are $C^2$, Taylor expansion at $z$ yields}
\[
{
\begin{aligned}
\Psi_p(x+\eta)-{\widetilde\Psi_{\widetilde p}}(x+\eta)
&=\bigl(D\Psi_p(z)-D{\widetilde\Psi_{\widetilde p}}(z)\bigr)n+O(\|n\|^2)\\
&=O(\|n\|^2)=O(\|\eta\|^4).
\end{aligned}
}
\]
{For base points in a fixed compact subset of $\M\cap V$,
choose a common perturbation radius so that $z$ and the segment
$z+\tau n$, $0\leq\tau\leq1$, remain in a compact subset of $V$.
The second derivatives of both limiting maps are then uniformly bounded,
giving the uniform estimate in \eqref{eq:polynomial-realization-agreement}.}
\end{proof}

{The comparison in
\eqref{eq:polynomial-realization-agreement} requires no zero-coupling
condition.
If the base normal blocks are also uniformly spectrally stable on $K$,
take $\widetilde p(\lambda)=\lambda$ and
$\widetilde H_{\widetilde p}=S$.
The estimate then gives fourth-order agreement between the accelerated
and unaccelerated induced retractions.}

\subsection{Relaxation and Chebyshev minimax polynomials}

{Classical polynomial acceleration selects a polynomial to control
the spectrum of an iteration operator \cite{Saad2003,Varga2000}.
In FPMA, \Cref{thm:main-polynomial} identifies the corresponding
objective as the uniform spectral radius $\rho_p$ in
\eqref{eq:transformed-normal-stability}. For a fixed degree bound
$m\geq1$, we therefore seek a polynomial that reduces $\rho_p$ while
retaining the consistency condition $p(1)=1$. Spectral mapping gives}
\begin{equation}
\label{eq:spectral-mapping}
\spr(p(A_x))
=\max_{\lambda\in\spec(A_x)}|p(\lambda)|.
\end{equation}
{Suppose that $\spec(A_x)\subset[a,b]$ for every $x\in K$,
where $a<b$ are known real bounds. Then
$\rho_p\leq\max_{\lambda\in[a,b]}|p(\lambda)|$, so we choose $p$
to minimize this interval bound. For the given interval and degree bound,
a \emph{minimax polynomial on $[a,b]$} is a consistent real polynomial
attaining the minimum}
\[
{
\min_{\substack{p\in\mathbb R[\lambda]\\
\deg p\leq m,\;p(1)=1}}
\max_{\lambda\in[a,b]}|p(\lambda)|.
}
\]
{An interval containing $1$ cannot yield a bound below $1$,
since $p(1)=1$. We therefore assume $1\notin[a,b]$ in the
Chebyshev construction below.}

{The simplest case is $m=1$: every consistent polynomial
has the form $p_\alpha(\lambda)=1-\alpha+\alpha\lambda$. Its
realization $(1-\alpha)\Id+\alpha S$ relaxes the entire base map $S$
and is the single-factor case of \eqref{eq:factored-cycle}.
On $[0,\rho]$, with $0<\rho<1$, its largest absolute value occurs
at an endpoint. Balancing the opposite endpoint values gives}
\begin{equation}
\label{eq:optimal-degree-one-relaxation}
\alpha_\ast=\frac{2}{2-\rho},
\qquad
\min_\alpha\max_{\lambda\in[0,\rho]}|p_\alpha(\lambda)|
=\frac{\rho}{2-\rho}.
\end{equation}
{For higher degree bounds, the minimizer is obtained from Chebyshev
polynomials.}

\subsubsection{Chebyshev minimax polynomial}
{Let $T_m$ be the Chebyshev polynomial of the first kind, defined by
$T_m(\cos\theta)=\cos(m\theta)$ for $m\geq0$.}
{Its recurrence, explicit real-variable formulas, and power
expansion are collected in \Cref{app:chebyshev-formulas}.}
{The following classical result identifies the interval minimax
polynomial in the normalization $p(1)=1$}
{\cite[Theorem~6.25]{Saad2003}; see also \cite{Varga2000}.}

\begin{theorem}[Chebyshev minimax polynomial among consistent polynomials]
\label{thm:chebyshev-minimax}
Let $a<b$, $m\geq1$, and $1\notin[a,b]$. Among all consistent real
polynomials of degree at most $m$, the unique minimax polynomial is
\begin{equation}
\label{eq:chebyshev-minimax-polynomial}
 p_m^{[a,b]}(\lambda)=
 \frac{T_m\!\left(\frac{2\lambda-a-b}{b-a}\right)}
 {T_m\!\left(\frac{2-a-b}{b-a}\right)}.
\end{equation}
Its minimax value is
\begin{equation}
\label{eq:chebyshev-minimax-value}
 \min_{\deg p\leq m,\,p(1)=1}
 \max_{\lambda\in[a,b]}|p(\lambda)|
 =\frac{1}{\left|T_m\!\left(\frac{2-a-b}{b-a}\right)\right|}.
\end{equation}
\end{theorem}

\begin{proof}
{Map $[a,b]$ to $[-1,1]$ and set}
\[
{\tau}=\frac{2-a-b}{b-a}.
\]
{Since $|\tau|>1$, the candidate $q_*(t)=T_m(t)/T_m(\tau)$
is well defined and has uniform norm $M:=1/|T_m(\tau)|$.
Let $t_j=\cos((m-j)\pi/m)$, $j=0,\ldots,m$, and let $L_j$ be the
Lagrange basis polynomials for these nodes. For either $\tau>1$ or
$\tau<-1$, $L_j(\tau)$ has the same sign as $q_*(t_j)$.
Any polynomial $q$ of degree at most $m$ with $q(\tau)=1$ and
$\max_{t\in[-1,1]}|q(t)|\leq M$ therefore satisfies}
\[
{L_j(\tau)\bigl(q_*(t_j)-q(t_j)\bigr)\geq0,\qquad j=0,\ldots,m.}
\]
{Interpolation at the exterior point gives}
\[
{0=q_*(\tau)-q(\tau)
=\sum_{j=0}^m L_j(\tau)\bigl(q_*(t_j)-q(t_j)\bigr).}
\]
{Every term is nonnegative and $L_j(\tau)\neq0$, so
$q(t_j)=q_*(t_j)$ at all $m+1$ nodes. Thus $q=q_*$, proving
both optimality and uniqueness.}
\end{proof}

{This optimality is relative to the chosen enclosing interval $[a,b]$.
On a smaller or disconnected spectral set, the same polynomial retains
the bound \eqref{eq:chebyshev-minimax-value} but need not be minimax.
Accordingly, an ``exact factor'' below means equality with the spectral
radius for the specified polynomial and family, and an ``attained endpoint''
means that the upper endpoint belongs to that family's spectral set.}

{The minimax polynomial in \eqref{eq:chebyshev-minimax-polynomial}
can be used in the affine-combination map \eqref{eq:orbit-affine-transform}.
To obtain the factorized map \eqref{eq:factored-cycle} through
\eqref{eq:consistent-factorization}, note that its roots are}
\begin{equation}
\label{eq:chebyshev-roots}
 \lambda_j
 =\frac{a+b}{2}
  +\frac{b-a}{2}
    \cos\!\left(\frac{(2j-1)\pi}{2m}\right),
 \qquad j=1,\ldots,m.
\end{equation}
{Since $1\notin[a,b]$, these roots differ from $1$, giving}
\begin{equation}
\label{eq:chebyshev-root-factorization}
 p_m^{[a,b]}(\lambda)
 =\prod_{j=1}^m\frac{\lambda-\lambda_j}{1-\lambda_j}
 =\prod_{j=1}^m
   \bigl(1-\omega_j+\omega_j\lambda\bigr),
 \qquad
 \omega_j:=\frac{1}{1-\lambda_j}.
\end{equation}
{The factorized map is thus evaluated by the cycle}
\begin{equation}
\label{eq:chebyshev-cycle}
 z_0=y,
 \qquad
 z_j=(1-\omega_j)z_{j-1}+\omega_jS(z_{j-1}),
 \quad j=1,\ldots,m,
 \qquad
 G_m^{\rm cyc}(y)=z_m.
\end{equation}
{Repeating this fixed map gives a stationary restarted scheme.
It stores a single iterate, whereas \eqref{eq:orbit-affine-transform}
combines the base-map iterates. Both constructions belong to the classical
semi-iterative tradition}
\cite{GolubVarga1961,Varga2000,WadayamaTakabe2021}.

{For the interval $[0,\rho]$, with $0<\rho<1$, we use the notation}
\begin{equation}
\label{eq:chebyshev-interval-polynomial}
 p_m^\rho(\lambda)=
 \frac{T_m(2\lambda/\rho-1)}{T_m(2/\rho-1)},
 \qquad
 C_m(\rho):=\frac{1}{|T_m(2/\rho-1)|}.
\end{equation}
{The minimax value satisfies}
\begin{equation}
\label{eq:chebyshev-geometric-bound}
 C_m(\rho)\leq
 2\left(\frac{1-\sqrt{1-\rho}}{1+\sqrt{1-\rho}}\right)^m.
\end{equation}
{For $m=1$, $p_m^\rho$ is the optimal relaxation in
\eqref{eq:optimal-degree-one-relaxation}.}

{The interval bound controls operator norm as well as spectral
radius when $A_x$ is self-adjoint. If $\spec(A_x)\subset[a,b]$, then}
\begin{equation}
\label{eq:normal-operator-polynomial-bound}
\|p_m^{[a,b]}(A_x)\|
=\max_{\lambda\in\spec(A_x)}
|p_m^{[a,b]}(\lambda)|
\leq
\max_{\lambda\in[a,b]}|p_m^{[a,b]}(\lambda)|.
\end{equation}
{The identity \eqref{eq:normal-operator-polynomial-bound} also
holds for a normal operator with real spectrum.
For nonnormal blocks, \eqref{eq:spectral-mapping} still holds but does not
imply a norm bound with unit prefactor or control transient growth.
The nonlinear convergence estimate \eqref{eq:polynomial-linear-rate}
instead allows a prefactor and every rate $\widehat q\in(\rho_p,1)$.}

\subsection{{Spectral factors per base-map evaluation and stabilization}}

{The two realizations above use $m$ evaluations of $S$ per
application. To compare it with ordinary iteration, we account for this
cost in the spectral factor. This comparison describes the first-order
asymptotic rate; it is not a finite-step nonlinear complexity estimate.}

\begin{definition}[{Spectral factor per base-map evaluation}]\label[definition]{def:acc-stab}
{Let $\rho_S:=\sup_{x\in K}\spr(A_x)<1$, and let $p$ be a consistent
polynomial of degree $m\geq1$. For a first-order realization requiring
at most $m$ evaluations of $S$ per application, charge $m$ evaluations
and define its spectral factor per base-map evaluation as $\rho_p^{1/m}$,
where $\rho_p$ is given in \eqref{eq:transformed-normal-stability}.
It improves on the base factor $\rho_S$ precisely when}
\begin{equation}
\label{eq:per-evaluation-acceleration}
\rho_p^{1/m} < \rho_S,
\qquad\text{equivalently}\qquad
\rho_p < \rho_S^m.
\end{equation}
{We use $0^{1/m}=0$. Both \eqref{eq:orbit-affine-transform} and
\eqref{eq:factored-cycle} meet the evaluation bound.}
\end{definition}

{Unlike $\rho_p<\rho_S$ alone, \eqref{eq:per-evaluation-acceleration}
accounts for the cost of a polynomial step. The Chebyshev interval bound
gives the following sufficient comparison.}

\begin{proposition}[Finite-degree Chebyshev improvement per base-map evaluation]
\label[proposition]{prop:finite-chebyshev-factor}
Let $0<\rho<1$. For every $m\geq1$, the Chebyshev minimax value on
$[0,\rho]$ satisfies
\begin{equation}
\label{eq:finite-chebyshev-factor}
 C_m(\rho)^{1/m}
 \leq \frac{\rho}{2-\rho}
 <\rho.
\end{equation}
{Consequently, if $\spec(A_x)\subset[0,\rho]$ for $x\in K$ and
$\rho_S=\rho$, every positive Chebyshev degree yields a smaller spectral
factor per base-map evaluation.}
\end{proposition}

\begin{proof}
{The polynomial $(p_1^\rho)^m$ is consistent and has degree $m$.
By \eqref{eq:chebyshev-minimax-value} and
\eqref{eq:optimal-degree-one-relaxation},}
\[
 C_m(\rho)
 \leq
 \max_{\lambda\in[0,\rho]}|p_1^\rho(\lambda)|^m
 =\left(\frac{\rho}{2-\rho}\right)^m.
\]
{Taking $m$th roots and using $0<\rho<1$ proves
\eqref{eq:finite-chebyshev-factor}.}
\end{proof}

\begin{proposition}[Asymptotic Chebyshev factor per base-map evaluation]
\label[proposition]{prop:asymptotic-chebyshev-factor}
{For $0<\rho<1$, the minimax values $C_m(\rho)$ in
\eqref{eq:chebyshev-interval-polynomial} satisfy}
\begin{equation}
\label{eq:asymptotic-chebyshev-factor}
\lim_{m\to\infty}C_m(\rho)^{1/m}
=
\frac{1-\sqrt{1-\rho}}{1+\sqrt{1-\rho}}
=:r(\rho),
\end{equation}
{with $r(\rho)<\rho$.}
\end{proposition}

\begin{proof}
Set $t=2/\rho-1>1$ and $\xi=t+\sqrt{t^2-1}$.
{By \eqref{eq:chebyshev-real-formulas} and
\eqref{eq:chebyshev-interval-polynomial}} \cite{MasonHandscomb2003},
\[
{
C_m(\rho)^{1/m}
=\xi^{-1}\left(\frac{2}{1+\xi^{-2m}}\right)^{1/m}
\longrightarrow\xi^{-1}.
}
\]
With
$s=\sqrt{1-\rho}\in(0,1)$,
\[
\xi^{-1}=\frac{1-s}{1+s}=r(\rho)<1-s^2=\rho,
\]
where the strict inequality follows from $1<(1+s)^2$.
\end{proof}

{If the chosen interval endpoint exceeds the actual base factor
$\rho_S$, the bound $C_m(\rho)^{1/m}<\rho_S$ still guarantees improvement,
but is not necessary: the exact comparison uses $\rho_p$ in
\eqref{eq:per-evaluation-acceleration}. When $\rho=\rho_S$,
\eqref{eq:finite-chebyshev-factor} guarantees improvement for every
$m\geq1$.}

{Polynomial transformation can also produce convergence when the
base map is not stable. We call $p$ \emph{stabilizing on $K$} if
\eqref{eq:transformed-normal-stability} holds, without requiring
$\rho_S<1$. The consistency condition creates one obstruction:
by \eqref{eq:spectral-mapping}, an eigenvalue $1$ of $A_x$ remains
an eigenvalue $1$ of $p(A_x)$ and prevents strict stability. A real
spectral interval excluding $1$ admits the following construction.}

\begin{corollary}[Finite-degree Chebyshev stabilization]
\label[corollary]{cor:finite-cheb-stabilization}
Let $S$ satisfy \Cref{ass:geometry}, and suppose its normal blocks are
self-adjoint and satisfy
\[
\spec(A_x)\subset[a,b],
\qquad x\in K,
\]
where {$a<b$ and} $1\notin[a,b]$. Then there exists $m_0\geq1$ such that, for every
$m\geq m_0$,
\[
\sup_{x\in K}\|p_m^{[a,b]}(A_x)\|<1.
\]
{For each fixed such $m$, a realization satisfying
\Cref{ass:well-defined} therefore has the convergence and derivative
properties of \Cref{thm:main-polynomial}, and the retraction properties
of \Cref{thm:second-order} under its additional hypotheses.}
\end{corollary}

\begin{proof}
{Set $t_1=(2-a-b)/(b-a)$. Since $|t_1|>1$,
\eqref{eq:chebyshev-real-formulas} gives $|T_m(t_1)|\to\infty$. Equations
\eqref{eq:normal-operator-polynomial-bound} and
\eqref{eq:chebyshev-minimax-value} give}
\[
\sup_{x\in K}\|p_m^{[a,b]}(A_x)\|
\leq
\max_{\lambda\in[a,b]}|p_m^{[a,b]}(\lambda)|
=
|T_m(t_1)|^{-1}{\longrightarrow0.}
\]
{Choose $m_0$ so that the last quantity is below $1$ for every
$m\geq m_0$.}
\end{proof}

\begin{remark}[Scope of the stationary analysis]
{In the affine-combination map \eqref{eq:orbit-affine-transform},
the coefficients are prescribed by a fixed polynomial.
For linear problems, Krylov methods and Anderson acceleration also admit
a polynomial interpretation, but can select combinations using evolving
subspaces or stored residuals} \cite{Saad2025}.
{Such data-dependent coefficients generally change with the
iteration, so the fixed-map analysis above does not directly apply.
In \eqref{eq:chebyshev-cycle}, the relaxation parameters vary within a
cycle, while the same composite map is repeated.}
{For each fixed degree and root ordering, local well-definedness
follows by shrinking the neighborhood; neither this neighborhood nor the
nonlinear prefactor need be uniform in the degree.}

{Different root orderings give the same derivative at fixed points
and the same spectral factor, but can change intermediate growth,
roundoff behavior, and the region of convergence. To moderate intermediate
products, our implementation uses a greedy Leja ordering
\cite[Algorithm~3.1]{BaiHuReichel1994}: start with the largest root,
then select the remaining root that maximizes the product
of its distances from the selected roots. The theory does not compare
the nonlinear or numerical effects of these orderings.}
\end{remark}
\section{Chebyshev-accelerated alternating projections}
\label{sec:applications}

{For the iteration error $e_k:=F^k(y)-\Psi_F(y)$,
\Cref{thm:limiting-map}\textup{(iii)} gives the local complexity of
achieving $\|e_k\|\leq\varepsilon$ for a tolerance $\varepsilon>0$.
Fix an admissible factor
$\widehat q\in(\rho_*,1)$ and an initial point $y$ in its corresponding
neighborhood. By \eqref{eq:limiting-map-rate}, the error tolerance is
met after at most}
\[
{
\left\lceil
\frac{\log\!\bigl(1+C_{\widehat q}\dist(y,\M)/\varepsilon\bigr)}
{-\log\widehat q}
\right\rceil
=O\!\left(\frac{\log(1/\varepsilon)}{-\log\widehat q}\right)
\quad(\varepsilon\downarrow0)
}
\]
{applications of $F$. Choosing $\widehat q=(1+\rho_*)/2$
gives the spectral-gap bound $O((1-\rho_*)^{-1}\log(1/\varepsilon))$
on its corresponding neighborhood.
For a fixed degree-$m$ polynomial realization using at most $m$
base-map evaluations per cycle, \Cref{thm:main-polynomial} gives
the same cycle bound with $\widehat q\in(\rho_p,1)$.
Charging $m$ base-map evaluations per cycle as in \Cref{def:acc-stab}
gives}
\[
{
O\!\left(\frac{m\log(1/\varepsilon)}{-\log\widehat q}\right)
=O\!\left(\frac{\log(1/\varepsilon)}{-\log(\widehat q^{1/m})}\right)
}
\]
{base-map evaluations. These estimates hold for a fixed problem,
initial point, and polynomial degree as $\varepsilon\downarrow0$;
the local neighborhood and prefactor depend on the chosen factor.
We therefore compare ordinary AP and Chebyshev acceleration through their
normal spectral factors, both per cycle and per AP evaluation.
The tangent-space geometry supplies the real spectral interval for this
comparison; vanishing normal-to-tangent coupling also gives the
retraction properties.}

\subsection{Exact alternating projections}\label{sec:exact-ap}

{Throughout this section, let $\Mone$, $\Mtwo$, and
$\M=\Mone\cap\Mtwo$ be $C^{r+1,\holderexp}$ embedded submanifolds,
with $r\in\{1,2\}$ and compact $K\subset\M$. Assume clean intersection
near $K$ as in \eqref{eq:clean-intersection}. Restrict to an ambient
neighborhood where both metric projections, and their composition
$S_{\rm AP}=P_{\Mone}P_{\Mtwo}$, are single-valued and
$C^{r,\holderexp}$. At fixed points, the chain rule gives}
\begin{equation}
\label{eq:exact-ap-derivative}
 DS_{\rm AP}(x)=P_{T_x\Mone}P_{T_x\Mtwo}.
\end{equation}

\begin{lemma}[Vanishing normal-to-tangential coupling for exact AP]
\label[lemma]{lem:exact-linearization}
{At $x\in\M$ near $K$, the derivative
\eqref{eq:exact-ap-derivative} is the identity on $T_x\M$ and preserves
$N_x\M$. Its decomposition \eqref{eq:tangent-normal-block} therefore
has $B_x^{\rm AP}=0$ and normal block}
\begin{equation}
\label{eq:exact-ap-normal-block}
A_x^{\rm AP}
=P_{N_x\M}P_{T_x\Mone}P_{T_x\Mtwo}|_{N_x\M}.
\end{equation}
\end{lemma}

\begin{proof}
{Both projectors fix $T_x\M$. For $n\in N_x\M$ and
$v\in T_x\M$, their self-adjointness gives}
\[
 \langle P_{T_x\Mone}P_{T_x\Mtwo}n,v\rangle
 =\langle n,P_{T_x\Mtwo}P_{T_x\Mone}v\rangle
 =\langle n,v\rangle=0.
\]
{Thus the normal space is invariant. This additional property of
exact AP gives $B_x^{\rm AP}=0$; fixing $\M$ alone does not force
vanishing coupling.}
\end{proof}

For $i=1,2$, define the reduced tangent spaces
\[
 E_i(x)=T_x\M_i\cap(T_x\M)^\perp,
 \qquad
 c_F(x):=c_F(T_x\Mone,T_x\Mtwo)
 =\|P_{E_1(x)}P_{E_2(x)}\|.
\]
{Clean intersection gives $c_F(x)<1$. To identify the spectrum of
\eqref{eq:exact-ap-normal-block}, introduce the self-adjoint compression}
\begin{equation}
\label{eq:exact-ap-compression}
{\widetilde A_x:=}P_UP_VP_U|_{U\ominus W},
\qquad
U=T_x\Mone,\quad V=T_x\Mtwo,\quad W=U\cap V,
\quad U\ominus W:=U\cap W^\perp.
\end{equation}
{Its eigenvalues are the squared cosines of the reduced principal
angles, together with any unmatched zero modes} 
\cite{BauschkeBelloCruzNghiaPhanWang2016}.

\begin{lemma}[Spectrum of the exact-AP normal block]\label[lemma]{lem:ap-spectrum}
{For $x\in\M$ near $K$, the normal block satisfies}
\begin{equation}
\label{eq:exact-ap-spectrum}
 \spec(A_x^{\rm AP})\subset[0,c_F(x)^2],
 \qquad
 \spr(A_x^{\rm AP})=c_F(x)^2.
\end{equation}
\end{lemma}

\begin{proof}
{Use $U,V,W$ from \eqref{eq:exact-ap-compression}. By
\Cref{lem:exact-linearization}, $A_x^{\rm AP}=P_UP_V|_{W^\perp}$.
If}
\[
z\in W^\perp,\qquad
A_x^{\rm AP}z=P_UP_Vz=\lambda z,\qquad \lambda\neq0,
\]
then $z=\lambda^{-1}P_UP_Vz\in U\cap W^\perp=U\ominus W$.
Hence $P_Uz=z$ and
\[
P_UP_VP_Uz=\lambda z.
\]
{Conversely, a nonzero eigenvector of $\widetilde A_x$ belongs
to $U\ominus W\subset W^\perp$ and is an eigenvector of
$A_x^{\rm AP}$ with the same eigenvalue. The nonzero spectra therefore
coincide. The compression is positive semidefinite, and for $z\in U\ominus W$,}
\[
 \langle P_UP_VP_Uz,z\rangle=\|P_Vz\|^2
 \leq c_F(x)^2\|z\|^2.
\]
{Its spectral radius equals
$\|\widetilde A_x\|=c_F(x)^2$. Any remaining eigenvalues of the
normal block are zero. These conclusions also hold when a reduced
space is trivial, by the convention in \Cref{def:friedrichs-principal}.}
\end{proof}

\begin{theorem}[{Chebyshev acceleration of exact AP}]
\label{thm:chebyshev-exact-ap}
{Under the standing hypotheses of this section, set}
\begin{equation}
\label{eq:exact-ap-uniform-factor}
 \bar c_F:=\sup_{x\in K}c_F(x),\qquad
 \rho_{\rm AP}:=\sup_{x\in K}\spr(A_x^{\rm AP})=\bar c_F^2.
\end{equation}
{Assume $0<\bar c_F<1$. For $m\geq1$, let $H_m$ be a locally
well-defined affine-combination or factorized realization of
$p_m^{\bar c_F^2}$ for $S_{\rm AP}$, as in
\eqref{eq:orbit-affine-transform} or \eqref{eq:factored-cycle},
with a fixed root ordering in the factorized case. Then the following
properties hold.}
\begin{enumerate}[label={\textup{(\roman*)}},leftmargin=*]
\item {The exact uniform normal spectral factor per accelerated cycle is}
\begin{equation}
\label{eq:exact-ap-chebyshev-factor}
 q_m
 :=\sup_{x\in K}
   \spr\bigl(p_m^{\bar c_F^2}(A_x^{\rm AP})\bigr)
 =C_m(\bar c_F^2)
 =\frac{1}{\left|T_m(2/\bar c_F^2-1)\right|}
 <1.
\end{equation}
{The equality uses the upper spectral endpoint
$\bar c_F^2$. Minimax optimality remains relative to the enclosing
interval $[0,\bar c_F^2]$, as in \Cref{thm:chebyshev-minimax}.}
\item {For this fixed $m$, the iterates converge locally to a $C^r$
limiting map $\Psi_m$. Every $\widehat q\in(q_m,1)$ is an admissible
local R-linear convergence factor in \eqref{eq:polynomial-linear-rate},
with $H_p=H_m$ and one iteration corresponding to a full accelerated cycle.}
\item {The spectral factor per base AP evaluation satisfies}
\begin{equation}
\label{eq:exact-ap-finite-work-acceleration}
 q_m^{1/m}
 \leq\frac{\bar c_F^2}{2-\bar c_F^2}
 <\rho_{\rm AP},
\end{equation}
{Thus every positive degree has a smaller spectral factor per AP evaluation
than ordinary AP under the cost convention of \Cref{def:acc-stab}.}
\end{enumerate}
\end{theorem}

\begin{proof}
{Equation \eqref{eq:exact-ap-spectrum} gives
$\rho_{\rm AP}=\bar c_F^2$ and encloses all normal spectra in
$[0,\bar c_F^2]$. By \eqref{eq:spectral-mapping} and
\eqref{eq:chebyshev-minimax-value}, $q_m\leq C_m(\bar c_F^2)$.
Since $c_F(x)=\|P_{T_x\Mone}P_{T_x\Mtwo}-P_{T_x\M}\|$ is continuous,
compactness gives $x_*\in K$ with $c_F(x_*)=\bar c_F$. The endpoint
$\bar c_F^2$ is therefore an eigenvalue of $A_{x_*}^{\rm AP}$, and}
\[
 q_m\geq
 \left|p_m^{\bar c_F^2}(\bar c_F^2)\right|
 =C_m(\bar c_F^2).
\]
{This proves \eqref{eq:exact-ap-chebyshev-factor}, while
\eqref{eq:finite-chebyshev-factor} gives
\eqref{eq:exact-ap-finite-work-acceleration}. Both realizations satisfy
\eqref{eq:first-order-polynomial-realization} by
\Cref{lem:poly-linearization}. Applying \Cref{thm:main-polynomial}
then gives convergence and regularity. The finite projection or relaxed
stages are locally well defined after restricting their domains as in
\Cref{ass:well-defined}.}

{The case $\bar c_F=0$ is excluded only because the interval formula
degenerates. Its normal spectrum is $\{0\}$, which already gives
local spectral stability without a Chebyshev construction.}
\end{proof}

\begin{remark}[Compression norm versus the ambient normal block]
\label[remark]{rem:exact-ap-compression-norm}
{For $\widetilde A_x$ in \eqref{eq:exact-ap-compression}, functional
calculus gives}
\begin{equation}
\label{eq:exact-ap-compression-bound}
 \|p_m^{\bar c_F^2}(\widetilde A_x)\|
 \leq C_m(\bar c_F^2).
\end{equation}
{Although $\widetilde A_x$ and $A_x^{\rm AP}$ share their nonzero
spectra, the ambient normal block can be nonnormal. The bound
\eqref{eq:exact-ap-compression-bound} is therefore asserted for the
compression; for the nonlinear AP map, convergence follows from
\Cref{thm:chebyshev-exact-ap}\textup{(ii)} and
\eqref{eq:polynomial-linear-rate}, with a prefactor and any local
R-linear factor $\widehat q\in(q_m,1)$.}
\end{remark}

\begin{corollary}[Retraction properties of Chebyshev-accelerated exact AP]
\label[corollary]{cor:chebyshev-exact-ap-geometry}
{Under \Cref{thm:chebyshev-exact-ap}, $\Psi_m$ induces a local
retraction on $\M$, and $D\Psi_m(x)=P_{T_x\M}$ at fixed points,
as for ordinary AP. If $r=2$, the retraction is second-order and,
uniformly for $x$ on compact subsets of its domain in $\M$ and
sufficiently small $\eta\in T_x\M$,}
\begin{equation}
\label{eq:chebyshev-ap-fourth-order-agreement}
 \|\Psi_m(x+\eta)-P_\M(x+\eta)\|=O(\|\eta\|^4).
\end{equation}
{These derivative and retraction properties, including
\eqref{eq:chebyshev-ap-fourth-order-agreement} when $r=2$, are
independent of the realization and root ordering. This does not require
the nonlinear limiting maps to coincide away from $\M$.}
\end{corollary}

\begin{proof}
{After shrinking the limiting-map neighborhood $V$,
\Cref{lem:exact-linearization} gives $B_x^{\rm AP}=0$ for every
$x\in\M\cap V$. For $p=p_m^{\bar c_F^2}$,
\eqref{eq:polynomial-coupling-block} therefore yields}
\[
{
B_{p,x}=B_x^{\rm AP}q_p(A_x^{\rm AP})=0,
\qquad x\in\M\cap V.
}
\]
{Thus \eqref{eq:polynomial-limiting-derivative} gives
$C_x=0$ and $D\Psi_m(x)=P_{T_x\M}$ for either realization.
The induced retraction follows from \Cref{thm:limiting-map}.
When $r=2$, the neighborhood identity $B_{p,x}=0$ permits application
of \Cref{thm:second-order}\textup{(ii)}, giving both the second-order
property and \eqref{eq:chebyshev-ap-fourth-order-agreement}.}
\end{proof}

{The factorized and affine-combination maps in
\eqref{eq:factored-cycle} and \eqref{eq:orbit-affine-transform} give
\Cref{alg:cheb-ap-cycle,alg:cheb-ap-orbit}, respectively, using
\eqref{eq:chebyshev-interval-polynomial} and \eqref{eq:chebyshev-roots}.
The degree, endpoint, and root ordering are fixed across restarts, so
each algorithm repeats a stationary map. In \Cref{thm:chebyshev-exact-ap},
$\rho=\bar c_F^2$.}
{Both algorithms use $m$ AP evaluations per cycle and share the spectral
factors $q_m$ per cycle and $q_m^{1/m}$ per AP evaluation from
\Cref{thm:chebyshev-exact-ap}.
This evaluation count covers the cycle updates and excludes any additional
work required by the stopping test. For $r=2$, their
limiting maps also satisfy \eqref{eq:polynomial-realization-agreement}.
\Cref{subsec:numerics-rates,subsec:numerics-geometry} compare their
spectral factors and retraction errors numerically.}

{The Boolean stopping test $\operatorname{Stop}(y)$ returns true when the
chosen stopping criterion is met, and false otherwise. It is evaluated
only at the initial point and at completed-cycle endpoints, with no checks
on the intermediate iterates $z_j$. For example, an approximate-feasibility
test is $\max_{i=1,2}\dist(y,\M_i)\leq\varepsilon$, where
$\varepsilon>0$ is a prescribed tolerance.}

\begin{algorithm}[tb]
\caption{Factorized Chebyshev realization of AP}
\label{alg:cheb-ap-cycle}
\begin{algorithmic}[1]
\Require Initial point $y$, degree $m\geq1$, endpoint $\rho\in(0,1)$,
and {a Boolean stopping test $\operatorname{Stop}(y)$}
\State $S_{\rm AP}\gets P_{\Mone}P_{\Mtwo}$
\For{$j=1,\ldots,m$}
  \State $\lambda_j\gets
  \frac{\rho}{2}\left(1+\cos\frac{(2j-1)\pi}{2m}\right)$
  \State $\omega_j\gets(1-\lambda_j)^{-1}$
\EndFor
\State Order the pairs $(\lambda_j,\omega_j)$ once
       \Comment{greedy Leja order in the experiments}
\While{$\operatorname{Stop}(y)$ is false}
  \State $z_0\gets y$
  \For{$j=1,\ldots,m$}
    \State $z_j\gets(1-\omega_j)z_{j-1}
      +\omega_j S_{\rm AP}(z_{j-1})$
  \EndFor
  \State $y\gets z_m$
\EndWhile
\Ensure $y$ {such that $\operatorname{Stop}(y)$ is true}
\Statex \Comment{$m$ base AP evaluations per cycle}
\end{algorithmic}
\end{algorithm}

\begin{algorithm}[tb]
\caption{{Affine-combination} Chebyshev realization of AP}
\label{alg:cheb-ap-orbit}
\begin{algorithmic}[1]
\Require Initial point $y$, degree $m\geq1$, endpoint $\rho\in(0,1)$,
and {a Boolean stopping test $\operatorname{Stop}(y)$}
\State $S_{\rm AP}\gets P_{\Mone}P_{\Mtwo}$
\State Compute $a_0,\ldots,a_m$ from
  $p_m^\rho(\lambda)=\sum_{j=0}^m a_j\lambda^j$
\While{$\operatorname{Stop}(y)$ is false}
  \State $z_0\gets y$
  \For{$j=1,\ldots,m$}
    \State $z_j\gets S_{\rm AP}(z_{j-1})$
  \EndFor
  \State $y\gets\sum_{j=0}^m a_jz_j$
\EndWhile
\Ensure $y$ {such that $\operatorname{Stop}(y)$ is true}
\Statex \Comment{$m$ base AP evaluations and $m+1$ {iterates} per cycle}
\end{algorithmic}
\end{algorithm}

\subsection{{Subspace comparison and angle scaling}}

{We compare the ambient-interval design with the sharper spectral
information available for subspaces.} For two linear subspaces $U,V\subset\E$, {with $0<c_F(U,V)<1$,} write
$c_F:=c_F(U,V)$ and $\theta_F:=\theta_F(U,V)$, and define the
self-adjoint compression
\[
\mathcal C_{U,V}:=P_UP_VP_U|_{U\ominus(U\cap V)}.
\]
{Let $e_k$ and $e_k^{(m)}$ denote the iteration errors in
$U\ominus(U\cap V)$ for ordinary AP and degree-$m$ Chebyshev cycles,
respectively. With $\rho=c_F^2$, the identity
$\|\mathcal C_{U,V}\|=\rho$ and
\eqref{eq:exact-ap-compression-bound} give}
\[
{
\begin{aligned}
e_{k+1}&=\mathcal C_{U,V}e_k,
&\|e_{k+1}\|&\leq\rho\|e_k\|,\\
e_{k+1}^{(m)}&=p_m^\rho(\mathcal C_{U,V})e_k^{(m)},
&\|e_{k+1}^{(m)}\|&\leq C_m(\rho)\|e_k^{(m)}\|.
\end{aligned}
}
\]
{The corresponding spectral factors per AP evaluation are
$\rho$ and $C_m(\rho)^{1/m}$.
The degree-one Chebyshev factor is $c_F^2/(2-c_F^2)$; by
\eqref{eq:asymptotic-chebyshev-factor}, the limit as $m\to\infty$
at fixed $c_F$ is}
\begin{equation}
\label{eq:subspace-chebyshev-factor}
 \frac{1-\sqrt{1-c_F^2}}{1+\sqrt{1-c_F^2}}
 =\frac{1-\sin\theta_F}{1+\sin\theta_F}.
\end{equation}

{For two-subspace iteration errors in the compression space, a positive
lower spectral bound may also be available.}  Let $\theta_{\max}$ denote the largest principal
angle represented in the compression, with
$\theta_{\max}=\pi/2$ by convention when an unmatched zero mode is retained.
Then the spectrum of $\mathcal C_{U,V}$ lies in
$[\cos^2\theta_{\max},\cos^2\theta_F]$.  Knyazev, Jujunashvili, and Argentati
\cite{KnyazevJujunashviliArgentati2010} placed this iteration in an
optimal-polynomial and conjugate-gradient framework.  Bello-Cruz
\cite[Section~6.4{, Proposition~6.8,} and Theorem~6.9]{BelloCruz2026} recently gave an explicit
Chebyshev semi-iteration on this full interval.  Equivalently, it applies
Chebyshev acceleration to $I-\mathcal C_{U,V}$ on
$[\sin^2\theta_F,\sin^2\theta_{\max}]$, {with limiting
spectral factor per AP evaluation as $m\to\infty$}
\begin{equation}
\label{eq:full-interval-chebyshev-factor}
 \frac{\sin\theta_{\max}-\sin\theta_F}
      {\sin\theta_{\max}+\sin\theta_F}.
\end{equation}
{Equation \eqref{eq:full-interval-chebyshev-factor} reduces to
\eqref{eq:subspace-chebyshev-factor} when $\theta_{\max}=\pi/2$,
and is smaller when $\theta_{\max}<\pi/2$. Our interval $[0,c_F^2]$
uses only the upper bound, allowing zero normal modes in the ambient
space and nonlinear realizations whose iterates need not stay in either
individual manifold.}

\begin{remark}[{Large-degree spectral factor and Friedrichs-angle scaling}]
\label[remark]{rem:ap-square-root}
{For the compact set $K$ in \Cref{thm:chebyshev-exact-ap},
write $\theta_F:=\arccos\bar c_F\in(0,\pi/2)$.
This is the minimum of the pointwise Friedrichs angles on $K$.
The ordinary AP factor in \eqref{eq:exact-ap-uniform-factor} satisfies}
\begin{equation}
\label{eq:ap-friedrichs-gap-scaling}
\rho_{\rm AP}=\bar c_F^2=\cos^2\theta_F,
\qquad
1-\rho_{\rm AP}=\sin^2\theta_F\sim\theta_F^2
\quad(\theta_F\downarrow0),
\end{equation}
{For comparison with Chebyshev acceleration, set}
\begin{equation}
\label{eq:friedrichs-chebyshev-factor}
r_F:=\frac{1-\sin\theta_F}{1+\sin\theta_F}.
\end{equation}
{At each fixed $\theta_F>0$,
\eqref{eq:asymptotic-chebyshev-factor} and the exact minimax value give}
\begin{equation}
\label{eq:ap-degree-angle-factor}
C_m(\bar c_F^2)=\frac{2r_F^m}{1+r_F^{2m}},
\qquad
C_m(\bar c_F^2)^{1/m}\leq2^{1/m}r_F,
\qquad
\lim_{m\to\infty}C_m(\bar c_F^2)^{1/m}=r_F.
\end{equation}
{Thus $r_F$ is the large-degree limit of the spectral factor per AP
evaluation for the interval design $[0,\bar c_F^2]$.
After taking this limit, letting
$\theta_F\downarrow0$ yields}
\begin{equation}
\label{eq:ap-log-factor-angle-scaling}
-\log r_F
=2\operatorname{artanh}(\sin\theta_F)
\sim2\theta_F.
\end{equation}
{Let $N$ count base AP evaluations and fix a target
$\varepsilon\in(0,1)$. The scalar envelope comparisons
$\rho_{\rm AP}^N\leq\varepsilon$ and $r_F^N\leq\varepsilon$
then give evaluation thresholds with respective asymptotic sizes}
\[
{
\frac{\log(1/\varepsilon)}{-\log\rho_{\rm AP}}
\sim\theta_F^{-2}\log(1/\varepsilon),
\qquad
\frac{\log(1/\varepsilon)}{-\log r_F}
\sim\frac{\log(1/\varepsilon)}{2\theta_F}.
}
\]
{These comparisons take $m\to\infty$ at fixed angle before
the small-angle limit. Keeping $m$ fixed gives a different scaling:
\eqref{eq:ap-degree-angle-factor} implies}
\[
{
\begin{aligned}
-\log\bigl(C_m(\bar c_F^2)^{1/m}\bigr)
&=\frac1m\log\cosh\!\bigl(2m\operatorname{artanh}(\sin\theta_F)\bigr)\\
&\sim2m\theta_F^2
\qquad(\theta_F\downarrow0,\ m\text{ fixed}).
\end{aligned}
}
\]
{The spectral comparison gives no uniform nonlinear work
bound, since the convergence neighborhood and R-linear prefactor
may depend on $m$.}

{For two linear subspaces with Friedrichs angle $\theta_F$, GAP with
the angle-based relaxation parameters of \cite[Theorem~4.2]{FaltGiselsson2023} has normal spectral
factor $r_F$ per iteration. Each GAP iteration uses one projection onto
each subspace, so its factor is measured at the same projection cost
as one base AP evaluation. For the interval design $[0,\bar c_F^2]$, the large-degree limit in
\eqref{eq:ap-degree-angle-factor} matches this GAP factor,
while $C_m(\bar c_F^2)^{1/m}>r_F$ for every finite $m$.
The GAP maps themselves are analyzed in \Cref{sec:gap}.}
\end{remark}

\section{Extensions to alternating-projection-type maps}
\label{sec:extensions}

{Projection approximations offer a way to reduce the cost of
exact AP, while relaxation introduces parameters for improving its
convergence rate. We study conditions under which these variants admit
local convergence, smooth limiting maps, and induced retractions within
FPMA, and under which polynomial acceleration improves their spectral
factors. For inexact AP, we first state an approximation condition that
preserves the derivative of the composed AP map at fixed points.
For relaxed and generalized AP, we identify the normal spectra and the
stable parameter ranges.}

{Throughout this section, use the geometric setting of
\Cref{sec:exact-ap}: $r\in\{1,2\}$, $0<\holderexp\leq1$, and
$\Mone$, $\Mtwo$, and $\M=\Mone\cap\Mtwo$ are
$C^{r+1,\holderexp}$ embedded submanifolds. Their intersection is clean
throughout a relative neighborhood $\mathcal O$ of the compact set
$K\subset\M$, as in \eqref{eq:clean-intersection}. The local metric
projections are single-valued and $C^{r,\holderexp}$ on a common ambient
open neighborhood of $\mathcal O$.}

\subsection{First-order projection surrogates and inexact alternating projections}
\label{sec:projection-surrogates}

A $C^1$ map $\phi_i$ {on an ambient neighborhood of $K$}
is called a \emph{first-order projection surrogate
for $\M_i$ along $\M$} if, at every $x\in\M$ near $K$,
\begin{equation}
\label{eq:projection-surrogate-jet}
  \phi_i(x)=x,
  \qquad
  D\phi_i(x)=P_{T_x\M_i}.
\end{equation}
Since $DP_{\M_i}(x)=P_{T_x\M_i}$, {these equalities mean that
$\phi_i$ and $P_{\M_i}$ agree in value and derivative at every $x\in\M$ near $K$. The derivatives are taken in the ambient space.}
Equivalently, for each such $x$,
\[
  \phi_i(y)-P_{\M_i}(y)=o(\|y-x\|)
  \qquad\text{as }y\to x.
\]
We do not require $\phi_i$ to take values in $\M_i$; when
$\phi_i(y)\in\M_i$ throughout its local domain, we call the surrogate
$\M_i$-valued.

{The first-order consistency condition in
\cite[Definition~4.1 and Lemma~4.1]{ChenHeHuang2026},}
\[
  \|\phi_i(y)-P_{\M_i}(y)\|
  =o\bigl(\dist(y,\M_i)\bigr),
\]
{implies \eqref{eq:projection-surrogate-jet} under $C^1$ regularity;
the converse need not hold. Thus \eqref{eq:projection-surrogate-jet}
permits a broader class of local approximations than that consistency condition.}

{For the convergence statements below, let each $\phi_i$ be
single-valued and $C^{r,\holderexp}$ on an ambient open neighborhood of
$\mathcal O$, and require \eqref{eq:projection-surrogate-jet} throughout
$\mathcal O$. Assume that $S=\phi_1\circ\phi_2$ is well defined and
$C^{r,\holderexp}$ on a common ambient open neighborhood of
$\mathcal O$. We call $S$ an inexact AP map (iAP).
The chain rule gives the derivative of exact AP, while
\Cref{lem:exact-linearization} identifies its blocks:}
\begin{equation}
\label{eq:iap-linearization}
DS(x)=P_{T_x\Mone}P_{T_x\Mtwo},
\qquad
 A_x^{\rm iAP}=A_x^{\rm AP},
 \qquad
 B_x^{\rm iAP}=0,\qquad x\in\mathcal O.
\end{equation}
{In particular, $S$ fixes $\mathcal O$. Its normal spectra and vanishing
coupling are independent of whether the component maps are
$\M_i$-valued. This allows direct application of the polynomial theorem.}

\begin{corollary}[{Polynomial transformation of inexact AP}]
\label[corollary]{cor:polynomial-iap}
{Under the preceding hypotheses, including
\eqref{eq:projection-surrogate-jet}, let $p$
be consistent and let $H_p$ be a first-order realization for $S$ satisfying
\Cref{ass:well-defined}. Suppose that}
\[
\rho_p^{\rm iAP}:=
\sup_{x\in K}\spr\bigl(p(A_x^{\rm iAP})\bigr)<1.
\]
{For $y$ in a sufficiently small open neighborhood $V$ of $K$, the
sequence $\{H_p^k(y)\}_{k\geq0}$ is well defined and converges.
Its limit defines a $C^r$ map $\Psi_p:V\to\M$. Every
$\widehat q\in(\rho_p^{\rm iAP},1)$ is an admissible local R-linear
factor in \eqref{eq:polynomial-linear-rate}, and the limiting map induces
a local retraction. If $p$ has degree $m\geq1$, $H_p$ can be evaluated
using at most $m$ evaluations of $S$ per application, and}
\[
\bigl(\rho^{\mathrm{iAP}}_p\bigr)^{1/m}
<
\sup_{x\in K} \operatorname{spr}\bigl(A^{\mathrm{iAP}}_x\bigr),
\]
{then it has a smaller spectral factor per base-map evaluation than
ordinary iAP, with the cost convention of \Cref{def:acc-stab}.
If the manifolds are $C^{3,\holderexp}$ and $H_p$ is
$C^{2,\holderexp}$, the induced retraction is second-order and has
$O(\|\eta\|^4)$ agreement with the metric-projection retraction
for sufficiently small $\eta\in T_x\M$, uniformly for $x$ on compact
subsets of $\M\cap V$.}
\end{corollary}

\begin{proof}
{By \eqref{eq:iap-linearization}, $S$ satisfies
\Cref{ass:geometry} after restricting the common domain. The local
regularity of $H_p$ and transformed stability permit application of
\Cref{thm:main-polynomial}, giving the limit and rate on a neighborhood
where every iterate is defined. The induced retraction follows from
\Cref{sec:limiting-map}, and the comparison per evaluation is
\eqref{eq:per-evaluation-acceleration}. For the higher-order conclusions,
\eqref{eq:polynomial-coupling-block} gives $B_{p,x}=0$ on
$\M\cap V$ after shrinking $V$, so \Cref{thm:second-order} applies.}
\end{proof}

\begin{corollary}[{Chebyshev acceleration of inexact AP}]
\label[corollary]{cor:chebyshev-iap}
{Under the same hypotheses, including
\eqref{eq:projection-surrogate-jet}, assume}
\[
 \bar c_F:=\sup_{x\in K}c_F(T_x\Mone,T_x\Mtwo)\in(0,1).
\]
{For each $m\geq1$, let $H_m$ be an affine-combination or factorized
realization of $p_m^{\bar c_F^2}$ as in
\eqref{eq:orbit-affine-transform} or \eqref{eq:factored-cycle},
satisfying \Cref{ass:well-defined}.
Its transformed normal blocks are uniformly spectrally stable, with exact
uniform spectral factor}
\begin{equation}
\label{eq:iap-chebyshev-factor}
 q_m^{\rm iAP}
 :=\sup_{x\in K}\spr\!\left(p_m^{\bar c_F^2}(A_x^{\rm iAP})\right)
 =C_m(\bar c_F^2),
\end{equation}
{and the factor per evaluation of $S$ satisfies}
\begin{equation}
\label{eq:iap-per-evaluation-factor}
 (q_m^{\rm iAP})^{1/m}
 \leq\frac{\bar c_F^2}{2-\bar c_F^2}
 <\bar c_F^2
 =\sup_{x\in K}\spr(A_x^{\rm iAP}).
\end{equation}
{Thus the convergence and retraction conclusions of
\Cref{cor:polynomial-iap} hold, with every
$\widehat q\in(q_m^{\rm iAP},1)$ admissible in
\eqref{eq:polynomial-linear-rate}. If the manifolds are $C^{3,\holderexp}$ and $H_m$ is $C^{2,\holderexp}$,
the induced retraction is second-order and has fourth-order agreement
with metric projection for tangent perturbations.}
\end{corollary}

\begin{proof}
{Equation \eqref{eq:iap-linearization} identifies the normal blocks
with those of exact AP. Their spectra therefore lie in
$[0,\bar c_F^2]$ by \eqref{eq:exact-ap-spectrum}, and the upper endpoint
is attained by the compact family. The endpoint argument of
\Cref{thm:chebyshev-exact-ap} gives
\eqref{eq:iap-chebyshev-factor}; \eqref{eq:finite-chebyshev-factor}
gives \eqref{eq:iap-per-evaluation-factor}. These two realizations use
$m$ evaluations of $S$ per application. Their transformed stability and
local well-definedness now allow \Cref{cor:polynomial-iap} to apply.}
\end{proof}

{Condition \eqref{eq:projection-surrogate-jet} is sufficient but not
necessary for vanishing coupling.
If each component map fixes $\M_i$ and its derivative preserves
$N_x\M_i$ as in \eqref{eq:iap-normal-space-invariance},
\Cref{prop:iap-zero-coupling} in \Cref{app:iap-zero-coupling} gives $B_x^{\rm iAP}=0$.
This condition does not identify the normal block with that of exact AP,
so the Friedrichs-angle Chebyshev design need not transfer. Once suitable
spectral information is available, \Cref{thm:main-polynomial} still
applies under its stability and regularity hypotheses. If genuine
tangential first-order error changes the coupling block, the
second-order criterion remains \eqref{eq:second-order-criterion}.}

\subsection{Relaxed and generalized alternating projections}\label{sec:gap}

{Generalized alternating projections (GAP) introduce
relaxation parameters at two levels to improve AP's convergence while
retaining one projection onto each set per iteration. Under the standing
geometric and local-projection hypotheses, let
$\alpha_1,\alpha_2,\beta\in\R$. The individual relaxed projections and
their composition are}
\begin{equation}
\label{eq:gap-inner-map}
R_{\alpha_i}^{(i)}(y):=(1-\alpha_i)y+\alpha_iP_{\M_i}(y),
\quad i=1,2,
\qquad
S_{\alpha_1,\alpha_2}
:=R_{\alpha_1}^{(1)}\circ R_{\alpha_2}^{(2)}.
\end{equation}
{We call $S_{\alpha_1,\alpha_2}$ the two-parameter
relaxed AP map: each metric projection is relaxed separately. It is
the inner map of GAP, with $\alpha_1$ and $\alpha_2$ as the inner
relaxation parameters. The outer parameter $\beta$ relaxes the whole
composition, giving the three-parameter GAP map}
\begin{equation}
\label{eq:gap-outer-map}
p_\beta(\lambda):=1-\beta+\beta\lambda,
\qquad
G_{\alpha_1,\alpha_2,\beta}
:=(1-\beta)\Id+\beta S_{\alpha_1,\alpha_2}.
\end{equation}
{Thus $\beta=1$ recovers the inner map, and
$\alpha_1=\alpha_2=\beta=1$ recovers ordinary AP. Setting
$\alpha_1=\alpha_2=1$ gives
$G_{1,1,\beta}=(1-\beta)\Id+\beta S_{\rm AP}$, the degree-one
relaxation of the full AP map described in \eqref{eq:factored-cycle}.
This differs from relaxing the two projections separately in
\eqref{eq:gap-inner-map}. The superscript $\mathrm{core}$ below
refers to the derivative blocks of the inner map.
Related relaxed projection methods are studied in}
\cite{CegielskiSuchocka2008}.

{For two linear subspaces with $0<\theta_F<\pi/2$,
angle-based GAP parameters attain the normal spectral factor
$(1-\sin\theta_F)/(1+\sin\theta_F)<\cos^2\theta_F$ at the same
projection cost as one AP step. The corresponding local convergence
analysis extends to smooth manifolds}
\cite{FaltGiselsson2017,FaltGiselsson2023}.
{Once the normal-block conditions are verified, the
convergence and retraction properties of the inner and full GAP maps
follow from \Cref{thm:limiting-map,thm:main-polynomial,thm:second-order}
under the stated smoothness assumptions. For conciseness, the spectral
verification and these FPMA applications are collected in
\Cref{app:gap-analysis}. We give below the derivative blocks and an
explicit Chebyshev acceleration result for the inner map with
$0<\alpha_1,\alpha_2\leq1$.}

\begin{proposition}[Fixed-point-manifold linearization of the two-parameter relaxed AP map]
\label[proposition]{prop:gap-linearization}
{For $x\in\mathcal O$, the map \eqref{eq:gap-inner-map} satisfies}
\begin{equation}
\label{eq:gap-map-derivative}
S_{\alpha_1,\alpha_2}(x)=x,
\qquad
DS_{\alpha_1,\alpha_2}(x)=L_{1,x}L_{2,x},
\qquad
L_{i,x}:=(1-\alpha_i)I+\alpha_iP_{T_x\M_i}.
\end{equation}
Both $T_x\M$ and $N_x\M$ are invariant. The restriction to $T_x\M$ is the
identity, the normal-to-tangential block satisfies $B_x^{\rm core}=0$, and
\begin{equation}
\label{eq:gap-core-linearization}
DS_{\alpha_1,\alpha_2}(x)
=P_{T_x\M}+A_x^{\rm core}P_{N_x\M},
\qquad
A_x^{\rm core}:=
L_{1,x}L_{2,x}|_{N_x\M}.
\end{equation}
\end{proposition}

\begin{proof}
{The chain rule gives \eqref{eq:gap-map-derivative}, and
$L_{i,x}v=v$ for $v\in T_x\M$. For $n\in N_x\M$, self-adjointness
of $P_{T_x\M_i}$ gives}
\[
{\langle P_{T_x\M_i}n,v\rangle
=\langle n,P_{T_x\M_i}v\rangle=\langle n,v\rangle=0,
\qquad v\in T_x\M.}
\]
{Each $L_{i,x}$ therefore preserves $N_x\M$, proving
\eqref{eq:gap-core-linearization} and $B_x^{\rm core}=0$.}
\end{proof}

{The vanishing coupling in \Cref{prop:gap-linearization}
leaves normal spectral stability as the remaining condition for
\Cref{thm:limiting-map}. The principal-angle analysis in
\Cref{app:gap-analysis} computes the normal spectra and shows that
$0<\alpha_1,\alpha_2<2$ ensures
$\rho_{\rm core}:=\sup_{x\in K}\spr(A_x^{\rm core})<1$
(\Cref{thm:gap-inner-stability}; \eqref{eq:gap-core-stability}).
Thus the inner iterates converge locally to a $C^r$ limiting map
$\Psi_{\rm core}$, with every $\widehat q\in(\rho_{\rm core},1)$
an admissible local R-linear factor. The limiting map induces a local
retraction. When $r=2$, this retraction is second-order and
$\Psi_{\rm core}$ has fourth-order agreement with metric projection for
tangent perturbations, uniformly on compact subsets of its domain in
$\M$ (\Cref{cor:gap-retraction}).}

{For the full GAP map \eqref{eq:gap-outer-map}, the corresponding
normal block is $p_\beta(A_x^{\rm core})=(1-\beta)I+\beta A_x^{\rm core}$.
For fixed inner parameters $0<\alpha_1,\alpha_2<2$ and positive outer
parameters $\beta>0$, \Cref{thm:gap-outer-stability} shows that this family is uniformly
spectrally stable exactly for $0<\beta<\beta_{\max}$, where
\eqref{eq:gap-outer-threshold} defines the threshold from the normal
spectra over $K$. Since $\beta_{\max}>1$, this range includes every
$0<\beta\leq1$. For every $\beta\in(0,\beta_{\max})$, set}
\[
{
\rho_{\rm out}:=\sup_{x\in K}
\spr\!\bigl((1-\beta)I+\beta A_x^{\rm core}\bigr)<1.
}
\]
{\Cref{cor:outer-gap-retraction} gives the same local convergence,
regularity, and retraction conclusions for the GAP iterates, with
admissible R-linear factors $\widehat q\in(\rho_{\rm out},1)$.
The second-order and fourth-order tangent-agreement properties again
hold when $r=2$.}

{The outer relaxation is a degree-one polynomial
realization of the inner map, with $p_\beta(1)=1$.
We now consider first-order realizations $H_p$ of arbitrary real
consistent polynomials for the same inner map. By
\eqref{eq:polynomial-linearization}, their normal blocks are
$p(A_x^{\rm core})$ and their coupling blocks remain zero.
For the following corollary, the inner parameters may again be arbitrary
real numbers: only the transformed blocks must be uniformly spectrally
stable. Thus the condition concerns $p(A_x^{\rm core})$, and can hold
even when $\rho_{\rm core}\geq1$.}

\begin{corollary}[Polynomially transformed two-parameter relaxed AP map]
\label[corollary]{cor:polynomial-gap}
{Let $H_p$ be a first-order realization of a real consistent polynomial
$p$ for $S_{\alpha_1,\alpha_2}$ satisfying \Cref{ass:well-defined}.
If}
\[
\rho_p^{\rm core}:=
\sup_{x\in K}\spr(p(A_x^{\rm core}))<1,
\]
{then, for $y$ in a sufficiently small open neighborhood $V$ of $K$,
the sequence $\{H_p^k(y)\}_{k\geq0}$ is well defined and converges.
Its limit defines a $C^r$ map $\Psi_p:V\to\M$. Every
$\widehat q\in(\rho_p^{\rm core},1)$ is an admissible local R-linear
factor in \eqref{eq:polynomial-linear-rate}, and the limiting map
induces a local retraction. If $r=2$, this retraction is second-order,
and its limiting map has $O(\|\eta\|^4)$ agreement with $P_\M$
for sufficiently small $\eta\in T_x\M$, uniformly for $x$ on compact
subsets of $\M\cap V$.}
\end{corollary}

\begin{proof}
{Equation \eqref{eq:gap-core-linearization} gives the
fixed-point-manifold structure, and \eqref{eq:polynomial-linearization}
gives normal block $p(A_x^{\rm core})$ and zero coupling for $H_p$.
The transformed stability condition permits \Cref{thm:main-polynomial}
to apply. The retraction follows from \Cref{sec:limiting-map}; when
$r=2$, the vanishing coupling on $\M\cap V$ after shrinking $V$
also gives both higher-order conclusions by \Cref{thm:second-order}.}
\end{proof}

{An explicit Chebyshev design requires a real spectral
interval. For the following theorem, fix
$0<\alpha_1,\alpha_2\leq1$ and assume
$\bar c_F:=\sup_{x\in K}c_F(T_x\Mone,T_x\Mtwo)\in(0,1)$.
Write $\delta:=(1-\alpha_1)(1-\alpha_2)$ and
$\bar\tau:=2-\alpha_1-\alpha_2+\alpha_1\alpha_2\bar c_F^2$.
The theorem establishes the following closed form for the uniform normal
spectral radius defined above:}
\begin{equation}
\label{eq:underrelaxed-gap-rho}
 \rho_{\rm core}
 =\frac{\bar\tau+
 \sqrt{\bar\tau^2-4\delta}}{2}.
\end{equation}
{For every $m\geq1$, let $H_m$ be an affine-combination or
factorized realization of $p_m^{\rho_{\rm core}}$ as in
\eqref{eq:orbit-affine-transform} or \eqref{eq:factored-cycle},
satisfying \Cref{ass:well-defined}.}

\begin{theorem}[{Chebyshev acceleration of two-parameter relaxed AP}]
\label{thm:chebyshev-gap}
{With this setup, $0<\rho_{\rm core}<1$ and}
\begin{equation}
\label{eq:underrelaxed-gap-spectrum}
 \spec(A_x^{\rm core})\subset[0,\rho_{\rm core}]{\quad(x\in K)},
 \qquad
 \sup_{x\in K}\spr(A_x^{\rm core})=\rho_{\rm core}.
\end{equation}
{The exact uniform spectral factor of $H_m$ is}
\begin{equation}
\label{eq:chebyshev-gap-factor}
 q_m^{\rm core}:=
 \sup_{x\in K}\spr\!\bigl(p_m^{\rho_{\rm core}}(A_x^{\rm core})\bigr)
 =C_m(\rho_{\rm core}),
\end{equation}
{and its factor per evaluation of the inner map satisfies}
\begin{equation}
\label{eq:chebyshev-gap-per-evaluation-factor}
 (q_m^{\rm core})^{1/m}
 \leq\frac{\rho_{\rm core}}{2-\rho_{\rm core}}
 <\rho_{\rm core}.
\end{equation}
{For each fixed $m$, the convergence and local retraction
conclusions of \Cref{cor:polynomial-gap} hold with
$\rho_p^{\rm core}=q_m^{\rm core}$; in particular, the iterates converge
locally to a $C^r$ limiting map and every
$\widehat q\in(q_m^{\rm core},1)$ is an admissible R-linear factor.
If the manifolds are $C^{3,\holderexp}$ and $H_m$ is $C^{2,\holderexp}$,
the second-order and fourth-order tangent-agreement conclusions also
hold, with the local and compact-uniform scope stated in that corollary.}
\end{theorem}

\begin{proof}
{First identify the spectral interval and its attained
upper endpoint. Put $u_i=1-\alpha_i\in[0,1)$.
By \eqref{eq:gap-characteristic-polynomial}, each principal block has
characteristic polynomial $z^2-\tau(c)z+\delta$, where
$c$ is the corresponding principal-angle cosine and}
\[
{
\tau(c)=u_1+u_2+\alpha_1\alpha_2c^2\geq2\sqrt{\delta},
\qquad \delta=u_1u_2.
}
\]
{Both roots are real and nonnegative. The larger root
$\lambda_+(c)$ is nondecreasing in $c^2$, and
$\lambda_+(0)=\max\{u_1,u_2\}\geq\delta$.
It therefore dominates the remaining scalar eigenvalues in
\eqref{eq:gap-transverse-eigenvalues} and
\eqref{eq:gap-complement-eigenvalue}.
Compactness and $\bar c_F>0$ give a principal block with
$c=\bar c_F$, proving \eqref{eq:underrelaxed-gap-rho} and
\eqref{eq:underrelaxed-gap-spectrum}, including attainment of the
upper endpoint. This endpoint is positive, and is below one by
\Cref{thm:gap-inner-stability}.}

{Next, spectral mapping and the Chebyshev interval bound,
together with the attained endpoint, give}
\[
{
\begin{aligned}
q_m^{\rm core}
&\leq \max_{\lambda\in[0,\rho_{\rm core}]}
 |p_m^{\rho_{\rm core}}(\lambda)|=C_m(\rho_{\rm core}),\\
q_m^{\rm core}
&\geq |p_m^{\rho_{\rm core}}(\rho_{\rm core})|
 =C_m(\rho_{\rm core}).
\end{aligned}
}
\]
{This proves \eqref{eq:chebyshev-gap-factor}.
Each realization uses $m$ inner-map evaluations, so
\eqref{eq:finite-chebyshev-factor} yields the strict comparison
\eqref{eq:chebyshev-gap-per-evaluation-factor}.
In particular, $q_m^{\rm core}<1$.}

{Finally, the local well-definedness and regularity of
$H_m$, together with these stable normal blocks and the zero coupling
in \eqref{eq:polynomial-coupling-block}, permit application of
\Cref{cor:polynomial-gap}. This gives local convergence, the rate, and
the induced retraction. Under the additional $C^{3,\holderexp}$ and
$C^{2,\holderexp}$ assumptions, apply that corollary with $r=2$ to
obtain second-order retractions and fourth-order agreement for small
tangent perturbations, uniformly on compact subsets.}
\end{proof}

{The strict improvement in
\eqref{eq:chebyshev-gap-per-evaluation-factor} is relative to the chosen
inner map $S_{\alpha_1,\alpha_2}$, equivalently GAP with $\beta=1$.
The choice $\alpha_1=\alpha_2=1$ recovers the exact AP result.
For other inner parameters, the comparison does not guarantee an
improvement over ordinary AP, nor does it establish improvement over
optimally tuned three-parameter GAP. For general inner parameters the normal spectra
may contain negative or complex eigenvalues, and outer relaxation
changes them through $p_\beta$ in \eqref{eq:gap-outer-map}.
\Cref{thm:main-polynomial} still gives local convergence when the
polynomially transformed normal blocks are uniformly spectrally stable.
For a degree-$m$ polynomial realized using at most $m$ evaluations of a
stable full GAP map, strict acceleration additionally requires the
comparison in \eqref{eq:per-evaluation-acceleration}.
Designing such polynomials for general inner and outer parameters,
and comparing them with optimally tuned GAP, is left for future work.}

\section{Numerical experiments}
\label{sec:numerics}

{We assess Chebyshev acceleration through three questions.
Does the response to a small normal perturbation agree with the predicted
spectral factor? Do the accelerated limiting maps preserve the retraction
properties of AP? Can these methods reduce the computational cost of a
nonlinear feasibility problem? We answer the first two questions using
controlled manifold intersections with known geometry, and the third
using finite unit-norm tight frame construction and completion.}

{We count one application of the base AP map, comprising one
projection onto each set, as one base AP evaluation. A degree-$m$
Chebyshev cycle costs $m$ such evaluations, so comparisons use evaluation
counts rather than cycle counts. In the frame tests, each evaluation
requires one thin singular value decomposition (SVD); SVD count is
therefore the primary cost measure, with elapsed time as a secondary
measure. Common warm-up steps and rejected proposals are included.
Additional protocol and implementation details are collected in
\Cref{app:numerics}.}

\subsection{{Spectral factors and Friedrichs-angle scaling}}
\label{subsec:numerics-rates}

{We first test the spectral prediction of
\Cref{thm:chebyshev-exact-ap} and compare the work required as the
intersection angle decreases. For a prescribed Friedrichs cosine
$c_F\in(0,1)$, we use a sphere and a hyperplane,}
\begin{equation}
 \Mtwo=\{y\in\R^d:\|y\|=1\},
 \qquad
 \Mone=\{y\in\R^d:a^Ty=c_F\},
 \qquad \|a\|=1.
 \label{eq:numerical-sphere-affine}
\end{equation}
{Their intersection is a $(d-2)$-sphere, and the normal
derivative block of AP has spectral radius $\rho=c_F^2$.
For a tested map $H$, its normal spectral factor means
$\spr(DH(x)|_{N_x\M})$ at a fixed point $x$.
In this example it equals $\rho$ for AP and $C_m(\rho)$ for either
degree-$m$ Chebyshev realization, by
\eqref{eq:exact-ap-chebyshev-factor}.
By \Cref{thm:limiting-map}\textup{(iii)}, every
$\widehat q\in(\rho,1)$ is an admissible local R-linear convergence factor
for AP. For each fixed $m$, \Cref{thm:chebyshev-exact-ap}\textup{(ii)}
gives the same conclusion for $\widehat q\in(C_m(\rho),1)$, with one
iteration corresponding to a complete Chebyshev cycle.
The corresponding spectral factors per base AP evaluation are $\rho$
and $C_m(\rho)^{1/m}$.}

{To measure this prediction without forming the derivative,
choose a unit eigenvector $v$ of the AP normal block associated with
$\rho$, perturb $x$ to $x+hv$, and apply $H$ once. We record}
\begin{equation}
\label{eq:factor-finite-difference}
 \|P_{N_x\M}(H(x+hv)-x)\|/h.
\end{equation}
{This quantity is the size of the resulting normal displacement
divided by the input displacement. Since $H(x)=x$, this finite-difference
quantity tends to $\|P_{N_x\M}DH(x)[v]\|$ as $h\downarrow0$.
For AP and both polynomial realizations, the chosen direction attains
the normal spectral factor. Thus the measurement can be compared with
$\rho$ or $C_m(\rho)$; it is a directional derivative estimate rather
than a ratio of successive nonlinear iteration errors.}

{We set $d=100$, $h=10^{-7}$,
$c_F\in\{0.5,0.8,0.95,0.99,0.999\}$, and $m\in\{1,2,4,8\}$.
Both the affine-combination construction \eqref{eq:orbit-affine-transform}
and the factorized construction \eqref{eq:factored-cycle} are tested,
the latter with Leja-ordered roots, over 20 random coordinate
realizations. Figure~\ref{fig:rate-validation}(a) displays the AP and
factorized results. Across both Chebyshev constructions, the largest
relative discrepancy for theoretical factors at least $10^{-8}$ is
$3.66\times10^{-3}$; for the factorized construction alone it is
$1.30\times10^{-4}$. The threshold excludes factors too small for a
meaningful relative-error comparison at the chosen finite-difference
step. These measurements support the predicted spectral factors and the
shared first-order model of the two realizations.}

{We next compare AP with factorized Chebyshev cycles when
the Friedrichs angle is small. Set $d=50$, $c_F=\cos\theta_F$,
$\theta_F\in\{0.4,0.2,0.1,0.05,0.025\}$, and
$m=\lceil1/\theta_F\rceil$. Both methods start from $x+10^{-8}v$
and reduce the distance to the known intersection by a factor $10^{-4}$,
subject to an absolute tolerance floor of $10^{-13}$. All runs over 20 coordinate
realizations at each angle reach the target. As the angle decreases from $0.4$ to
$0.025$, the median base AP evaluation counts rise from $57$ to $14733$
for AP and from $18$ to $280$ for Chebyshev acceleration.
The log--log fitted slopes are $-2.004$ and $-0.992$, respectively,
close to the predicted values $-2$ and $-1$ in
\Cref{rem:ap-square-root}. Hence this test shows a slower
increase in work for the selected Chebyshev degrees as the angle
shrinks. It probes the large-degree spectral prediction with
$m\sim1/\theta_F$ at finitely many angles.}

\begin{figure}[!htbp]
\centering
\begin{minipage}[t]{0.49\textwidth}
  \centering
  \includegraphics[width=\linewidth]{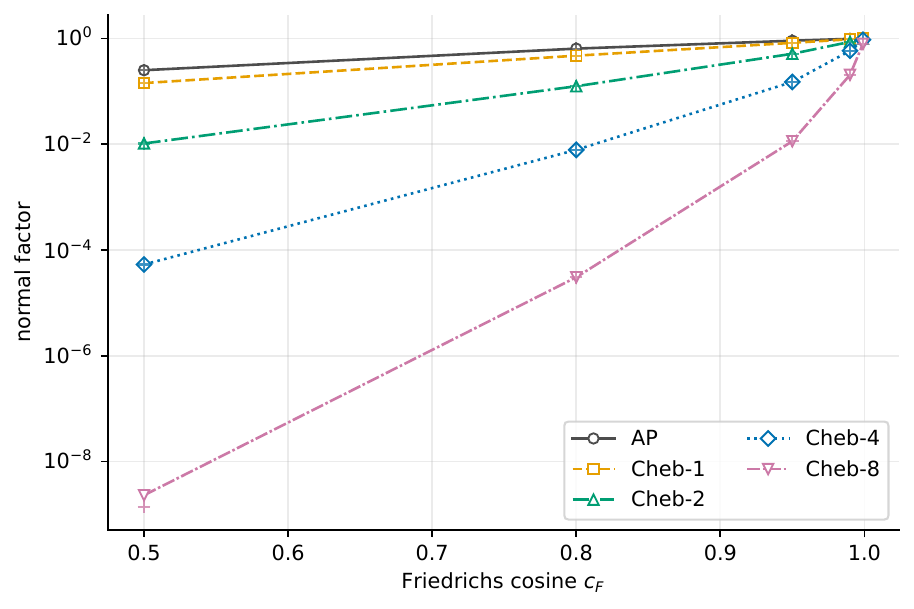}
  \small {(a) Estimated and exact normal spectral factors.}
\end{minipage}\hfill
\begin{minipage}[t]{0.49\textwidth}
  \centering
  \includegraphics[width=\linewidth]{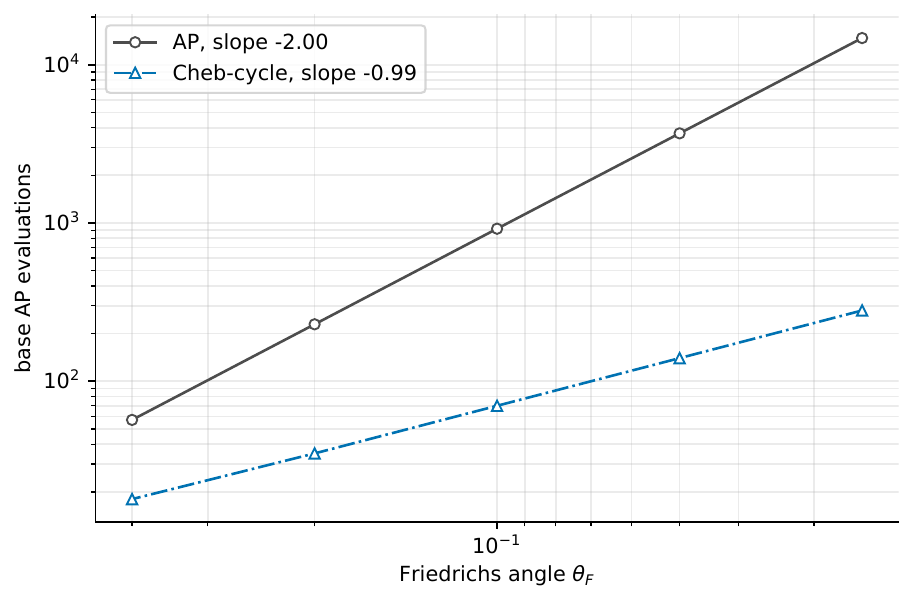}
  \small (b) Base AP evaluations versus Friedrichs angle.
\end{minipage}
\caption{{Spectral-factor and work comparisons for
\eqref{eq:numerical-sphere-affine}. In (a), curves give median
finite-difference measurements \eqref{eq:factor-finite-difference} over
20 coordinate realizations, and crosses give the exact normal spectral
factors. Cheb-$m$ denotes one degree-$m$ factorized cycle, costing $m$
base AP evaluations. In (b), points give median base AP evaluation counts
to reduce the distance to the intersection by $10^{-4}$, with
$m=\lceil1/\theta_F\rceil$ for Chebyshev acceleration. Slopes are fitted
over the five tested angles.}}
\label{fig:rate-validation}
\end{figure}


\subsection{{Preservation of retraction properties}}
\label{subsec:numerics-geometry}

{The next test asks whether acceleration preserves the
geometric accuracy of the limiting map. We compare AP with both
Chebyshev realizations of degrees two and four on two curved surfaces
in $\R^3$,}
\begin{align}
 \Mone&=\{(u,v,\tfrac12\kappa u^2):u,v\in\R\},\notag\\
 \Mtwo&=\{(u,v,\tfrac12\kappa u^2+(s+\gamma u)v):u,v\in\R\},
 \qquad s=(c_F^{-2}-1)^{1/2}{.}
 \label{eq:numerical-parabolic-surfaces}
\end{align}
{The local intersection is the parabola
$\M=\{(u,0,\tfrac12\kappa u^2):u\in\R\}$. We take
$c_F=0.8$, $\kappa=1$, and $\gamma=0.8$; the mixed term $\gamma uv$
reduces the symmetry without changing the Friedrichs cosine at the
origin. Starting from a tangent perturbation $y_t=x+t\eta$, where
$x=0$ and $\eta$ is a unit tangent vector, we approximate the limiting map
$\Psi$ of each method using a limit tolerance of $10^{-14}$ and measure}
\begin{equation}
 e_2(t)=\|P_{T_x\M}(\Psi(y_t)-x-t\eta)\|,
 \qquad
 e_4(t)=\|\Psi(y_t)-P_\M(y_t)\|{.}
 \label{eq:numerical-geometry-errors}
\end{equation}
{The error $e_2$ measures the tangential departure from the
initial displacement $t\eta$, while $e_4$ measures the difference from
metric projection onto the intersection. A second-order retraction
satisfies $e_2(t)=o(t^2)$. In this setting,
\eqref{eq:base-tangent-jet} gives the sharper prediction
$e_2(t)=O(t^3)$, and \Cref{thm:fourth-order-agreement} gives
$e_4(t)=O(t^4)$. We therefore compare fitted log--log slopes with three
and four, respectively. The fits use eight step lengths between
$0.003557$ and $0.2$, with 20 random choices of tangent sign.}

{Figure~\ref{fig:geometry-validation} gives fitted orders
$2.992$--$3.001$ for $e_2$ and $3.963$--$4.057$ for $e_4$,
consistent with the cubic and fourth-order predictions for AP and both
accelerated realizations. The differences between the two accelerated
limiting maps have fitted orders $3.944$ and $3.948$ at degrees two and
four, consistent with \eqref{eq:polynomial-realization-agreement}.}

\begin{figure}[!htbp]
\centering
\includegraphics[width=0.92\textwidth]{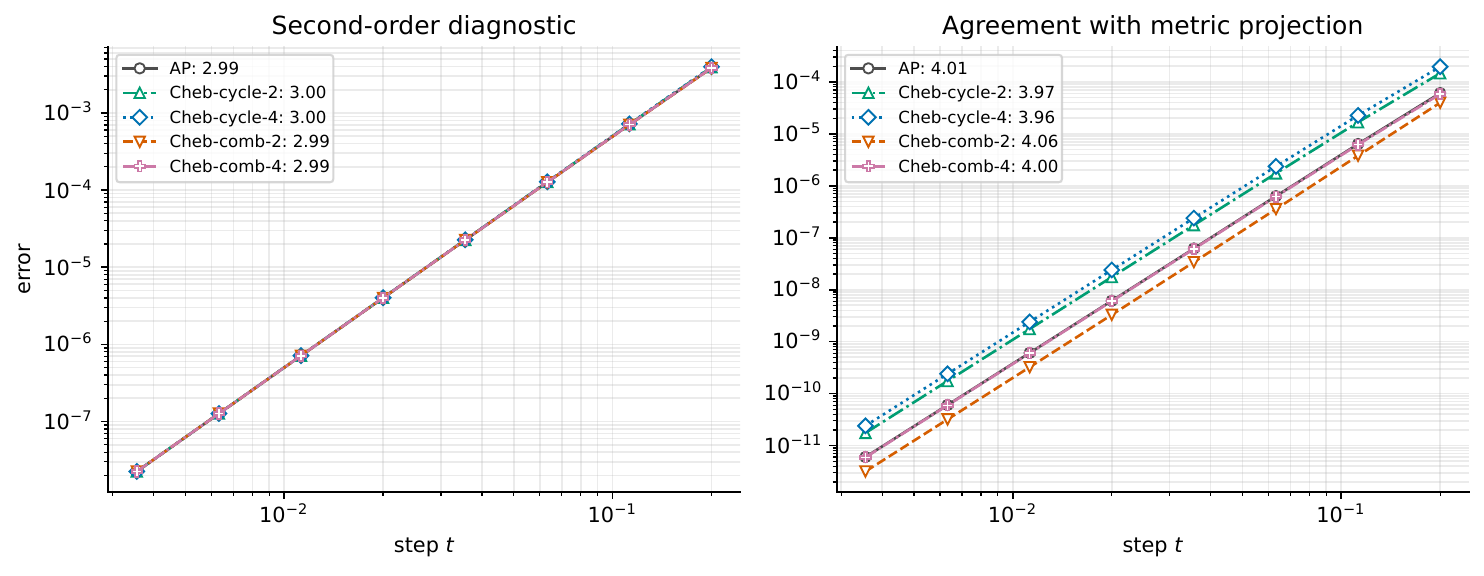}
\caption{Retraction errors
\eqref{eq:numerical-geometry-errors} for AP and Chebyshev acceleration.
Left: tangential error $e_2$; right: metric-projection error $e_4$.
Legends report fitted orders.}
\label{fig:geometry-validation}
\end{figure}


\subsection{{Tight-frame construction and completion}}
\label{subsec:numerics-frames}

{We compare computational costs for finite unit-norm
tight frame (FUNTF) construction and completion at a common feasibility
tolerance. A FUNTF consists of $N$ unit vectors in $\R^d$ whose outer
products sum to $(N/d)I_d$, providing redundant spanning systems for
signal representations and sensing}
\cite{BenedettoFickus2003}.
{Construction seeks all $N$ vectors; completion keeps
$p$ prescribed vectors and seeks the remaining $n:=N-p$ vectors.
Writing the free vectors as $X=[x_1,\ldots,x_n]\in\R^{d\times n}$,
both tasks take the feasibility form}
\begin{equation}
 X\in\mathcal U\cap\mathcal T_C,
 \qquad
 \mathcal U=\{X:\|x_j\|=1,\ j=1,\ldots,n\},
 \qquad
 \mathcal T_C=\{X:XX^T=C\}.
 \label{eq:funtf-feasibility}
\end{equation}
{For construction, $p=0$, $n=N$, and $C=(N/d)I_d$.
For completion, $C=(N/d)I_d-X_{\rm fix}X_{\rm fix}^T$, where
$X_{\rm fix}\in\R^{d\times p}$ contains the prescribed vectors;
$[X_{\rm fix}\ X]$ is a FUNTF precisely when $X$ is feasible.
\Cref{tab:frame-results} lists five problem sizes as $(d\times N;p)$.
Each prescribed block comes from a precomputed feasible frame, and all
methods start from the same normalized Gaussian free vectors for each
instance.}

{AP is a standard method for this problem}
\cite{Liu2021}.
{Projection onto $\mathcal U$ normalizes the columns.
Projection onto $\mathcal T_C$ is a Procrustes step using one thin SVD,}
\begin{equation}
\label{eq:frame-tightness-projection}
 P_{\mathcal T_C}(Y)=C^{1/2}UV^T,
 \qquad U\Sigma V^T=\operatorname{svd}(C^{1/2}Y).
\end{equation}
{Thus one application of $S_C=P_{\mathcal T_C}P_{\mathcal U}$
costs one SVD. We measure both constraints using the normalized
feasibility residual}
\begin{equation}
 r(X)=\left(
 \frac{\|XX^T-C\|_F^2}{\max\{1,\|C\|_F\}^2}
 +\frac1n\sum_{j=1}^n(\|x_j\|-1)^2
 \right)^{1/2}.
 \label{eq:frame-residual}
\end{equation}
{Here $\|\cdot\|_F$ denotes the Frobenius norm.
The two terms measure tightness and unit length,
respectively. All methods stop at $r(X)\leq10^{-10}$ within a budget
of $5\times10^4$ AP evaluations.}

{We compare AP (\Cref{sec:exact-ap}) with the factorized
map \eqref{eq:factored-cycle} at degrees one (R-AP), four (Cheb-4),
and eight (Cheb-8). Comparison with AP measures
the overall gain; comparison with R-AP tests the benefit of higher
degree. All methods begin with ten common AP steps, included in the
cost. For the accelerated methods, up to five recent positive residual
ratios $r_k/r_{k-1}$ give a common heuristic interval estimate
$\widehat\rho$ (\Cref{alg:funtf-cheb-ap}). R-AP relaxes the full
AP map $S_C$ as $(1-\omega)\Id+\omega S_C$, with the minimax
parameter $\omega=2/(2-\widehat\rho)$ from
\eqref{eq:optimal-degree-one-relaxation} for the estimated interval.
Cheb-4 and Cheb-8 factorize $p_m^{\widehat\rho}$ from
\eqref{eq:chebyshev-interval-polynomial} in Leja order and repeat the
cycles. The estimate is not a certified spectral bound. A nonfinite
residual or a cycle residual above $1.2$ times the previous value
rejects the Chebyshev proposal and triggers AP fallback; all attempted
evaluations count.}

{We test 20 paired instances per size: 100 instances
and 400 method runs. The table reports median SVD counts and times,
with each timing measured as the median of three repetitions.
Figure~\ref{fig:frame-speedups} reports median paired SVD speedups
$\nu_{\rm AP}/\nu_{\rm method}$ with 95\% percentile bootstrap
intervals from 5000 resamples. Values above one indicate lower cost
at the same tolerance.}

\begin{algorithm}[!tbp]
\caption{Safeguarded Chebyshev-AP used for FUNTF construction and completion}
\label{alg:funtf-cheb-ap}
\begin{algorithmic}[1]
\Require Target $C\succeq0$, initial $X$ with nonzero columns, degree
$m\in\{4,8\}$, warm-up $w=10$, tolerance $\varepsilon$, evaluation budget $B$
\State $S_C\gets P_{\mathcal T_C}P_{\mathcal U}$,\quad
       $\nu\gets0$,\quad $r_0\gets r(X)$
\For{$k=1,\ldots,w$}
  \If{$r_{k-1}\leq\varepsilon$}
    \State \Return $X$
  \EndIf
  \State $X\gets S_C(X)$,\quad $\nu\gets\nu+1$,\quad $r_k\gets r(X)$
\EndFor
\State $q_{\rm med}\gets
  \operatorname{median}\{r_k/r_{k-1}:\text{up to five recent valid ratios}\}$
\State $\widehat\rho\gets
  \min\{0.9999,\max\{10^{-6},1.02q_{\rm med}\}\}$
\State Compute the Leja-ordered $(\lambda_j,\omega_j)_{j=1}^m$
       from $\widehat\rho$ as in \Cref{alg:cheb-ap-cycle}
\While{$r(X)>\varepsilon$ and $\nu+m\leq B$}
  \State $X_{\rm old}\gets X$,\quad $r_{\rm old}\gets r(X)$,\quad $Y_0\gets X$
  \For{$j=1,\ldots,m$}
    \State $Y_j\gets(1-\omega_j)Y_{j-1}+\omega_jS_C(Y_{j-1})$
  \EndFor
  \State $Y\gets Y_m$,\quad $\nu\gets\nu+m$
  \If{$r(Y)$ is nonfinite or $r(Y)>1.2r_{\rm old}$}
    \State $X\gets X_{\rm old}$
    \For{$j=1,\ldots,m$ while $\nu<B$}
      \State $X\gets S_C(X)$,\quad $\nu\gets\nu+1$
    \EndFor
  \Else
    \State $X\gets Y$
  \EndIf
\EndWhile
\Ensure $X$ and the counted number $\nu$ of base AP evaluations
\end{algorithmic}
\end{algorithm}

\begin{table}[!htbp]
\centering
\caption{FUNTF feasibility results over 20 paired instances. Each entry
is ``median SVD count / median time in milliseconds.'' All 400 reported runs
converged. Boldface marks the smallest median SVD count and time in each
row; ties in SVD count are retained.}
\label{tab:frame-results}
\small
\setlength{\tabcolsep}{5pt}
\begin{tabular}{lcccc}
\toprule
problem $(d\times N;p)$ & AP & R-AP & Cheb-4 & Cheb-8\\
\midrule
FUNTF $(16\times32;0)$
 & $54.5/5.47$ & $37.5/3.94$ & $\mathbf{30/3.55}$ & $34/4.04$\\
FUNTF $(32\times64;0)$
 & $45/10.79$ & $31/7.68$ & $\mathbf{26/6.72}$ & $\mathbf{26}/6.89$\\
FUNTF $(64\times128;0)$
 & $38/34.55$ & $28/25.87$ & $\mathbf{26/23.84}$ & $\mathbf{26}/24.16$\\
completion $(32\times64;16)$
 & $62/12.73$ & $41/8.45$ & $\mathbf{34/7.45}$ & $\mathbf{34}/7.57$\\
completion $(64\times128;32)$
 & $52/34.93$ & $36/24.37$ & $\mathbf{30/20.37}$ & $34/22.98$\\
\bottomrule
\end{tabular}
\end{table}

{All 400 runs attain residuals below $10^{-10}$.
Cheb-4's per-problem median speedups over AP are $1.46$--$1.82$
in SVD count and $1.45$--$1.71$ in time; pooled medians over the
100 paired instances are $1.73$ and $1.61$, respectively. Each
per-problem SVD bootstrap interval lies above one. Cheb-4 uses fewer
SVDs than R-AP on all 100 instances and less time on 93. Cheb-8 does
not uniformly improve on Cheb-4, so a higher degree need not reduce
the cost of a finite run.}

{The 100 interval estimates range from $0.562$ to
$0.714$, with no safeguard rejections. Timings describe this
implementation; SVD counts more directly measure projection work.}

\begin{figure}[!htbp]
\centering
\includegraphics[width=0.82\textwidth]{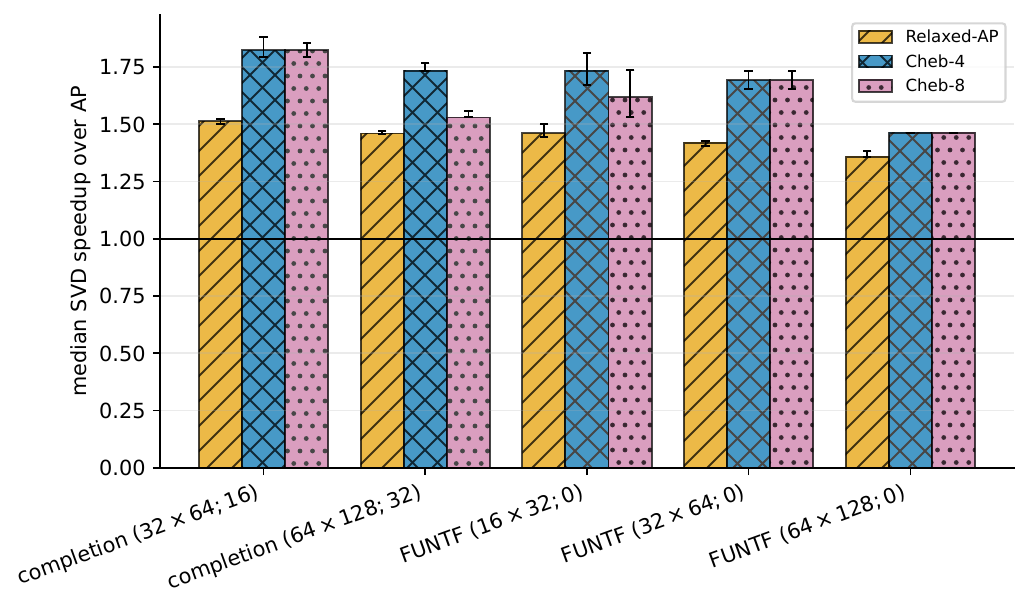}
\caption{{Computational savings relative to AP on the
five FUNTF problems. Points give the median paired SVD-count ratio
$\nu_{\rm AP}/\nu_{\rm method}$ over 20 instances of each size;
values above one indicate fewer SVDs at the same feasibility tolerance.
Error bars give 95\% percentile bootstrap intervals from 5000 resamples.}}
\label{fig:frame-speedups}
\end{figure}

{Taken together, the controlled tests support the predicted
normal spectral factors and preservation of the limiting maps'
geometric accuracy, while the frame tests demonstrate lower projection
cost for Cheb-4 on the specified construction and completion instances.
The comparison with R-AP isolates an additional benefit of the
higher-degree polynomial, and the Cheb-8 results highlight the need to
choose degree by actual work. The finite-angle measurements and the
heuristic frame designs provide numerical evidence within these tests;
they do not extend the local theory to a uniform nonlinear complexity
or a general reliability guarantee.}


\section{Conclusion}
\label{sec:conclusion}

The FPMA framework separates normal convergence from limiting-map behavior
through the tangent--normal linearization of a map that fixes a manifold
pointwise.  Under uniform spectral stability, polynomial acceleration changes
the normal block from $A_x$ to $p(A_x)$ while preserving the normal derivative
$B_x(I-A_x)^{-1}$ of the limiting map.  For exact AP on a clean manifold
intersection, the uniform normal spectral radius is
$\rho_{\rm AP}=\bar c_F^2$, and Chebyshev acceleration gives a smaller
{normal spectral factor} per AP evaluation for every finite degree.  In
the {large-degree Chebyshev limit, the normal spectral
factors predict an improvement in the Friedrichs-angle dependence}
from $O(\theta_F^{-2})$ to $O(\theta_F^{-1})$.  Under the stated
smoothness conditions, the accelerated limiting maps also induce local
second-order retractions and agree with metric projection to fourth order for
tangent perturbations.  The framework further applies to {inexact maps that match metric projections in value and
derivative at the intersection} and to relaxed and generalized AP under
appropriate stability conditions.  The numerical experiments support the
predicted {normal spectral} factors, angle scaling, and retraction behavior, and show
reduced SVD counts on the tested FUNTF problems.

\section*{Acknowledgements}

The research questions, analytical framework, and methods presented in this
work were conceived and developed by the authors. The authors used ChatGPT-5.6 Sol (OpenAI), a generative artificial intelligence (AI) system, to assist with selected technical parts of the mathematical development, including portions of proofs involving recurrence relations, inductive arguments, and related algebraic derivations, as well as with drafting selected portions of the manuscript. The authors subsequently checked and revised all
AI-assisted material and independently verified all mathematical arguments,
computations, and references. The authors take full responsibility for the
accuracy and content of the final manuscript.

\appendix
\input{appendix_limiting_map_proofs}

\input{appendix_inexact_zero_coupling}
\input{appendix_gap_analysis}
\input{appendix_numerics}
\input{appendix_chebyshev_formulas}

\begingroup
\hbadness=1400
\small
\printbibliography
\endgroup

\end{document}

%% file: limiting_map_proof.tex
{To prove \Cref{thm:limiting-map}, we first pass to a
contracting iterate $H=F^N$. The next three lemmas construct the limiting
map of $H$,
establish its joint $C^r$ dependence on the initial point, and recover
the convergence and rate of the original $F$-iteration.}
{We use the norm conventions of \Cref{sec:fixed-manifold-framework}.}

\begin{lemma}[{A contracting iterate and its limiting map}]
\label{lem:contracting-iterate}
\label[lemma]{app:phase-map}
\label[lemma]{app:phase-invariance}
{Suppose \Cref{ass:geometry} holds with $S=F$, and
the normal blocks $A_x$ satisfy \eqref{eq:normal-spectral-stability}. For every $\widehat q\in(\rho_*,1)$, there exist an integer
$N\geq1$, a number $0<\beta<\widehat q^{\,N}$, and an open neighborhood
$\mathcal V_0$ of $K$ such that $H:=F^N$ is defined on $\mathcal V_0$
and $H(\mathcal V_0)\subset\mathcal V_0$.
The iterates $H^k$ converge uniformly on $\mathcal V_0$ to a continuous
map $\Psi_H:\mathcal V_0\to\M$ fixing $\M\cap\mathcal V_0$ pointwise.
There is a constant $C_H>0$ such that}
\begin{equation}\label{eq:app-h-rate}
{
\|H^k(y)-\Psi_H(y)\|
\leq C_H\beta^k\dist(y,\M),
\qquad y\in\mathcal V_0,\quad k\geq0.
}
\end{equation}
{The maps $F^j$, $0\leq j<N$, are defined both on
$\mathcal V_0$ and on $\Psi_H(\mathcal V_0)$.}
\end{lemma}

\begin{proof}[{Proof sketch}]
{Use \eqref{eq:normal-power-bound} to choose $H=F^N$ with
$\sup_{x\in K}\|A_x^N\|<\beta$.
While the iterates remain in a tubular chart \eqref{eq:tubular-chart},
their coordinates $(u_k,v_k)$ satisfy}
\[
{
\|v_k\|\leq\beta^k\|v_0\|,
\qquad
\|u_{k+1}-u_k\|=O(\beta^k\|v_0\|).
}
\]
{A margin from the chart boundary that bounds the total
base displacement gives an $H$-invariant neighborhood, so these bounds
hold for all $k$.
Their geometric tails define a continuous limiting map and give
\eqref{eq:app-h-rate}. A finite chart cover makes the estimates uniform.
The full proof is given in \Cref{app:proof-contracting-iterate}.}
\end{proof}

\begin{lemma}[{Smoothness of the limiting map}]
\label{lem:limiting-map-smoothness}
{Under the assumptions of \Cref{lem:contracting-iterate},
use its construction with a sufficiently small common normal radius.
Then the ambient limiting map satisfies
$\Psi_H\in C^r(\mathcal V_0,\E)$ and, for $1\leq j\leq r$,}
\begin{equation}\label{eq:app-derivative-convergence}
{
D^jH^k\longrightarrow D^j\Psi_H
\quad\text{locally uniformly on }\mathcal V_0.
}
\end{equation}
\end{lemma}

\begin{proof}[{Proof sketch}]
{Differentiate the coordinate iterates in the full
initial pair $y=(u_0,v_0)$.
The variational recurrences give decaying normal derivatives and
summable increments of the base derivatives.
For $r=2$, the H\"older bounds and decaying normal variations also make
the second-order source terms summable.
On every compact coordinate ball $Q$ inside the final coordinate domain,}
\[
{
\sum_{k\geq0}\sup_Q\|D_y^j(u_{k+1}-u_k)\|<\infty,
\qquad
\sup_Q\|D_y^jv_k\|\longrightarrow0,
\quad 1\leq j\leq r.
}
\]
{The fundamental theorem of calculus identifies the
derivative limits with those of $(u_\infty,0)$.
Composition with the tubular charts gives the joint $C^r$ ambient map
and \eqref{eq:app-derivative-convergence}.
The full proof is given in \Cref{app:proof-limiting-map-smoothness}.}
\end{proof}

\begin{lemma}[{Recovery of the original iteration}]
\label{lem:original-iteration}
{Under the assumptions of \Cref{lem:contracting-iterate},
use its construction for a fixed $\widehat q\in(\rho_*,1)$.
All $F$-iterates are defined on $\mathcal V_0$ and converge to
$\Psi_F=\Psi_H$. For some $C_{\widehat q}>0$,}
\begin{equation}\label{eq:app-original-rate}
{
\|F^k(y)-\Psi_F(y)\|
\leq C_{\widehat q}\widehat q^{\,k}\dist(y,\M),
\qquad y\in\mathcal V_0,\quad k\geq0.
}
\end{equation}
{The limiting maps obtained from different choices of
$N$ and $\widehat q$ agree on the intersections of their domains.}
\end{lemma}

{The full proof is given in
\Cref{app:proof-original-iteration}.}

\begin{proof}[{Proof of \Cref{thm:limiting-map}}]
\begingroup
Applying \Cref{lem:contracting-iterate,lem:limiting-map-smoothness,lem:original-iteration}
for one choice of $\widehat q\in(\rho_*,1)$ gives an open neighborhood
$V$ of $K$ on which the full sequence converges and the limiting map
\eqref{eq:limiting-map-definition} is $C^r$, with $\Psi_F=\Psi_H$ fixing
$\M\cap V$ pointwise.
This proves \textup{(i)} and the regularity assertion in \textup{(ii)}.
At $x\in\M\cap V$, \eqref{eq:app-derivative-convergence} and the
fixed-point derivative blocks in \eqref{eq:tangent-normal-block} give
\begin{equation}\label{eq:app-phase-derivative}
\begin{aligned}
D\Psi_F(x)
&=\lim_{k\to\infty}D(F^{Nk})(x)\\
&=\lim_{k\to\infty}
\begin{pmatrix}
I&B_x\displaystyle\sum_{s=0}^{Nk-1}A_x^s\\
0&A_x^{Nk}
\end{pmatrix}
=\begin{pmatrix}I&B_x(I-A_x)^{-1}\\0&0\end{pmatrix}.
\end{aligned}
\end{equation}
Here the normal blocks remain stable on the relative neighborhood used
in \Cref{lem:contracting-iterate}; hence the matrix limit is valid at
every such fixed point. This proves \eqref{eq:limiting-map-derivative}
and completes \textup{(ii)}.

For any $\widehat q\in(\rho_*,1)$,
\Cref{lem:contracting-iterate,lem:original-iteration} give
\eqref{eq:app-original-rate} on a possibly different neighborhood. The last assertion of
\Cref{lem:original-iteration} identifies its limiting map with $\Psi_F$
on the intersection with $V$. Restricting $V$ to this intersection proves
\eqref{eq:limiting-map-rate}, hence \textup{(iii)}, with its stated
dependence on $\widehat q$.
\endgroup
\end{proof}

%% file: appendix_limiting_map_proofs.tex
\section{{Proofs of the limiting-map lemmas}}
\label[appendix]{app:limiting-map-proofs}

{This appendix gives the complete proofs of
\Cref{lem:contracting-iterate,lem:limiting-map-smoothness,lem:original-iteration}.}
{We use the norm conventions of \Cref{sec:fixed-manifold-framework}.
In the normal coordinates of \eqref{eq:tubular-chart}, $B_\delta(0)$ denotes
the open Euclidean ball of radius $\delta$ in $\R^c$, and
$\overline{B_\delta(0)}$ its closure.}

\subsection{Proof of Lemma~\ref{lem:contracting-iterate}}
\label{app:proof-contracting-iterate}

\begin{proof}
{Fix $\widehat q\in(\rho_*,1)$. By
\eqref{eq:normal-power-bound}, choose $N\geq1$ and $\beta$ such that}
\begin{equation}\label{eq:app-contracting-iterate-choice}
{\sup_{x\in K}\|A_x^N\|<\beta<\widehat q^{\,N}.}
\end{equation}
{Set $H=F^N$. Because $F$ fixes $K$, the finitely many
maps $F,\ldots,F^N$ have a common open domain near $K$.
The map $H$ is $C^{r,\holderexp}$, fixes $\M$ pointwise on its domain,
and has normal block $A_x^N$ by \eqref{eq:tangent-normal-block}.
By continuity, choose an open neighborhood
$U\subset\E$ of $K$ in that common domain such that
$\|A_x^N\|<\beta$ for every $x\in\M\cap U$.}

{Choose a local tubular chart $\Theta_\ell$ from
\eqref{eq:tubular-chart} with base patch in $\M\cap U$.
Since $H$ fixes this patch, its coordinate representation near $v=0$ is}
\begin{equation}\label{eq:app-coordinate-model}
{
(\Theta_\ell^{-1}\circ H\circ\Theta_\ell)(u,v)
=(u+h(u,v),G(u,v)).
}
\end{equation}
{Here $h(u,v)$ is the change in the coordinates of the
base point, and $G(u,v)$ gives the coefficients of the updated normal
vector in the frame at $\chi_\ell(u+h(u,v))$.
In particular, $\|G(u,v)\|=\dist(H(\Theta_\ell(u,v)),\M)$.
We omit the chart index in these component maps and, in coordinate
expressions, also denote the map in \eqref{eq:app-coordinate-model} by $H$.
Because the ambient map fixes $\chi_\ell(u)$, the map in
\eqref{eq:app-coordinate-model} fixes $(u,0)$, so}
\(h(u,0)=0\), \(G(u,0)=0\),
\(D_uh(u,0)=0\), and \(D_uG(u,0)=0{.}\)
{Write}
\(D(u):=D_vh(u,0)\) and \(A_H(u):=D_vG(u,0){.}\)
{The matrix $A_H(u)$ represents $A_{\chi_\ell(u)}^N$
in the orthonormal normal frame of \eqref{eq:tubular-chart}; hence
$\|A_H(u)\|=\|A_{\chi_\ell(u)}^N\|<\beta$.
On a compact inner coordinate domain, $C^{1,\holderexp}$ regularity
gives local constants $L_0,L_1$ such that, for sufficiently small $v$,}
\begin{align}
\|h(u,v)-D(u)v\|&\leq L_0\|v\|^{1+\holderexp}, \label{eq:app-h0}\\
\|G(u,v)-A_H(u)v\|&\leq L_0\|v\|^{1+\holderexp}, \label{eq:app-g0}\\
\|D_uh(u,v)\|+\|D_vh(u,v)-D(u)\|+\|D_uG(u,v)\|
&\leq L_1\|v\|^{\holderexp}. \label{eq:app-d1}
\end{align}
{Compactness of $K$ gives finitely many such charts,
with outer base domains $\mathcal O_\ell^+$ and inner coordinate balls
$\mathcal O_\ell\Subset\mathcal O_\ell^+$ whose images cover $K$.
Choose a common $\delta_0>0$ so that each compact input set
$\Theta_\ell(\overline{\mathcal O_\ell}\times\overline{B_{\delta_0}(0)})$
lies in $U$ and its $H$-image in the corresponding chart image.
The charts, their inverses and coordinate updates have uniform
$C^{r,\holderexp}$ bounds on these fixed compact sets.
Thus \eqref{eq:app-h0}--\eqref{eq:app-d1} hold with common $L_0,L_1$ for
$u\in\mathcal O_\ell$, $\|v\|\leq\delta_0$, and $\|D(u)\|\leq D_0$
for some $D_0>0$; these bounds remain valid under radius reductions.}

{Using \eqref{eq:app-contracting-iterate-choice} and the identity
$D_vG(u,0)=A_H(u)$, continuity allows us to reduce $\delta_0$ so that}
\begin{equation}\label{eq:app-normal-contraction}
\|G(u,v)\|\leq\beta\|v\|,
\qquad
\|D_vG(u,v)\|\leq\beta
\end{equation}
throughout those charts.

{For $r=2$, the coordinate map in
\eqref{eq:app-coordinate-model} is $C^{2,\holderexp}$.
Differentiating $h(u,0)=G(u,0)=0$ twice with respect to $u$ gives
$D_{uu}^2h(u,0)=D_{uu}^2G(u,0)=0$.
H\"older continuity of these second derivatives and uniform bounds for
the mixed and pure $v$ second derivatives on the fixed compact domains
give a common $L_2$ such that}
\begin{align}
\|D_{uu}^2h(u,v)\|&\leq L_2\|v\|^{\holderexp}, &
\|D_{uv}^2h(u,v)\|+\|D_{vv}^2h(u,v)\|&\leq L_2,
\label{eq:app-d2h}\\
\|D_{uu}^2G(u,v)\|&\leq L_2\|v\|^{\holderexp}, &
\|D_{uv}^2G(u,v)\|+\|D_{vv}^2G(u,v)\|
&\leq L_2. \label{eq:app-d2g}
\end{align}
{The mixed derivative satisfies
$D_{uv}^2h(u,0)=D_uD(u)$ and may be nonzero.}

{We first estimate the iterates in local coordinates using
\eqref{eq:app-coordinate-model}. By \eqref{eq:app-h0},}
\(\|h(u,v)\|\leq L_h\|v\|{,}\)
{where $L_h:=D_0+L_0\delta_0^{\holderexp}$.
For a finite sequence of iterates with coordinates
$(u_k,v_k):=\Theta_\ell^{-1}(H^k(y))$ remaining in
$\mathcal O_\ell\times B_{\delta_0}(0)$, the update is
$u_{k+1}=u_k+h(u_k,v_k)$ and $v_{k+1}=G(u_k,v_k)$.
By \eqref{eq:tubular-projection}, $\|v_k\|=\dist(H^k(y),\M)$,
and \eqref{eq:app-normal-contraction} gives}
\begin{equation}\label{eq:app-orbit-normal}
\|v_k\|\leq\beta^k\|v_0\|.
\end{equation}
{The bound for $h$ then yields}
\begin{equation}\label{eq:app-orbit-drift}
\|u_{k+1}-u_k\|
\leq L_h\|v_k\|
\leq L_h\beta^k\|v_0\|.
\end{equation}
{For $j>k$ in this segment, summing
\eqref{eq:app-orbit-drift} gives
$\|u_j-u_k\|\leq L_h\sum_{i=k}^{j-1}\beta^i\|v_0\|
\leq L_h\beta^k\|v_0\|/(1-\beta)$.
Thus, if the iterates remain in the coordinate cylinder for all $k$,
then $u_k\to u_\infty$, $v_k\to0$, and}
\begin{equation}\label{eq:app-coordinate-tail}
\|u_\infty-u_k\|\leq\frac{L_h}{1-\beta}\beta^k\|v_0\|,
\qquad \|v_k\|\leq\beta^k\|v_0\|{.}
\end{equation}

{The preceding sum bounds the total base displacement
by $a\|v_0\|$, where $a:=L_h/(1-\beta)$.
We now use this bound to choose a domain that keeps all iterates
in the same chart. Choose $0<\delta\leq\min\{1,\delta_0\}$.
If $\dim\M=0$, take
$W_\ell:=\mathcal O_\ell\times B_\delta(0)$; the base coordinate is constant
and \eqref{eq:app-normal-contraction} gives invariance.
If $\dim\M>0$, let
$d_\ell(u):=\dist(u,\partial\mathcal O_\ell)$ and reserve the
corresponding margin to the boundary by taking}
\begin{equation}\label{eq:app-wedge-domain}
W_\ell:=\left\{(u,v)\in\mathcal O_\ell\times B_\delta(0):
d_\ell(u)>a\|v\|\right\}{.}
\end{equation}
{For $(u,v)\in W_\ell$, we have
$\|h(u,v)\|\leq L_h\|v\|\leq a\|v\|<d_\ell(u)$,
so $u+h(u,v)\in\mathcal O_\ell$. The $1$-Lipschitz property of $d_\ell$ yields}
\[
\begin{aligned}
d_\ell(u+h(u,v))
&\geq d_\ell(u)-L_h\|v\|\\
&>a\|v\|-L_h\|v\|
=a\beta\|v\|\geq a\|G(u,v)\|{.}
\end{aligned}
\]
{Together with $\|G(u,v)\|\leq\beta\|v\|<\delta$,
this proves invariance in the positive-dimensional case as well.
In either case, $W_\ell$ is open and contains
$\mathcal O_\ell\times\{0\}$. Thus $\mathcal V_0$ is an open neighborhood
of $K$, with}
\begin{equation}\label{eq:app-forward-domain}
\mathcal V_0:=\bigcup_\ell\Theta_\ell(W_\ell),\qquad
{H(\mathcal V_0)\subset\mathcal V_0.}
\end{equation}
{Reducing $\delta$ preserves these properties and the fixed coefficient bounds.}

{For $y=\Theta_\ell(u_0,v_0)$ with $(u_0,v_0)\in W_\ell$,
invariance now ensures that \eqref{eq:app-orbit-normal}--\eqref{eq:app-coordinate-tail}
hold for all $k\geq0$. For $\dim\M>0$, their total-displacement bound gives}
\[
{d_\ell(u_\infty)\geq d_\ell(u_0)
-\sum_{k=0}^\infty\|u_{k+1}-u_k\|
\geq d_\ell(u_0)-a\|v_0\|>0.}
\]
{Since $u_k\in\mathcal O_\ell$, the limit lies in
$\mathcal O_\ell$; for $\dim\M=0$, it equals $u_0$.
Define the ambient limiting map by}
\begin{equation}\label{eq:app-coordinate-limit}
{\Psi_H(\Theta_\ell(u_0,v_0)):=\Theta_\ell(u_\infty,0)=\chi_\ell(u_\infty).}
\end{equation}
{These definitions agree on overlaps because they are
limits of the same ambient sequence $H^k(y)$.}
{By \eqref{eq:tubular-chart} and \eqref{eq:tubular-projection},
$\|v_0\|=\dist(\Theta_\ell(u_0,v_0),\M)$.
Both $(u_k,v_k)$ and $(u_\infty,0)$ lie in the convex coordinate cylinder
$\mathcal O_\ell\times B_{\delta_0}(0)$.
The fixed derivative bounds for $\Theta_\ell$ and
\eqref{eq:app-coordinate-tail} therefore give a uniform constant $C_H>0$ such that}
\[
\|H^k(y)-\Psi_H(y)\|
\leq C_H\beta^k\dist(y,\M),
\qquad y\in \mathcal V_0,\quad k\geq0{.}
\]

{Since $\dist(y,\M)<\delta$ on $\mathcal V_0$,
\eqref{eq:app-h-rate} gives uniform convergence of the continuous
maps $H^k$, so $\Psi_H$ is continuous. It fixes $\M\cap\mathcal V_0$
because $H$ fixes $\M$ pointwise.
Finally, $\mathcal V_0\cup\Psi_H(\mathcal V_0)$ is contained in the union
of the fixed input sets, hence in $U$. Thus $F^j$, $0\leq j<N$, are defined on both
$\mathcal V_0$ and $\Psi_H(\mathcal V_0)$.}
\end{proof}

\subsection{Proof of Lemma~\ref{lem:limiting-map-smoothness}}
\label{app:proof-limiting-map-smoothness}

\begin{proof}
{We estimate the derivatives of the coordinate iterates
with respect to all initial coordinates. Their locally uniform convergence
will give the required joint regularity.
Use \eqref{eq:app-coordinate-model} and the fixed coefficient bounds from
\Cref{lem:contracting-iterate}. Each reduction of the normal radius uses
\eqref{eq:app-wedge-domain} when $\dim\M>0$, preserving the invariance
in \eqref{eq:app-forward-domain} and these coefficient bounds.}

\paragraph{{First derivatives.}}
{For an initial coordinate pair $y=(u_0,v_0)\in W_\ell$,
set $U_k:=D_yu_k$ and $V_k:=D_yv_k$.
The initial derivatives are coordinate projections, so
$\|U_0\|,\|V_0\|\leq1$.
The derivative of \eqref{eq:app-coordinate-model} has the block form}
\begin{equation}\label{eq:app-variational-block}
DH(u_k,v_k)=
\begin{pmatrix}
I+E_k&R_k\\
Q_k&A_{H,k}
\end{pmatrix}.
\end{equation}
{Here $E_k=D_uh$, $R_k=D_vh$, $Q_k=D_uG$, and
$A_{H,k}=D_vG$, all evaluated at $(u_k,v_k)$.
Since $\|v_0\|<\delta\leq\delta_0$, \eqref{eq:app-orbit-normal} gives
$\|v_k\|\leq\delta\beta^k$.
Using \eqref{eq:app-d1}, the bound $\|D(u)\|\leq D_0$ from
\Cref{lem:contracting-iterate}, and \eqref{eq:app-normal-contraction}, we obtain}
\begin{equation}\label{eq:app-first-blocks}
{
\begin{aligned}
\|E_k\|+\|Q_k\|
&\leq L_1\|v_k\|^{\holderexp}
\leq L_1\delta^{\holderexp}\beta^{\holderexp k},\\
\|R_k\|
&\leq\|D(u_k)\|+\|D_vh(u_k,v_k)-D(u_k)\|\\
&\leq D_0+L_1\|v_k\|^{\holderexp}
\leq D_0+L_1\delta_0^{\holderexp}=: \overline R,\\
\|A_{H,k}\|
&=\|D_vG(u_k,v_k)\|\leq\beta.
\end{aligned}
}
\end{equation}
{The fixed radius $\delta_0$ makes $\overline R$
independent of subsequent reductions of $\delta$.}
{The chain rule and \eqref{eq:app-variational-block} give}
\begin{align}
U_{k+1}&=(I+E_k)U_k+R_kV_k, \label{eq:app-U}\\
V_{k+1}&=Q_kU_k+A_{H,k}V_k. \label{eq:app-V}
\end{align}
{We seek a uniform bound for $U_k$ and geometric decay
of $V_k$, which together will make the increments of $U_k$ summable.
Choose $\sigma\in(\beta^{\holderexp},1)$.
Since $\beta\leq\beta^{\holderexp}<\sigma$, the convolution estimate is}
\begin{equation}\label{eq:app-geometric-convolution}
{
\sum_{j=0}^k\beta^{k-j}\beta^{\holderexp j}
\leq\sum_{j=0}^k\beta^{k-j}\sigma^j
=\frac{\sigma^{k+1}-\beta^{k+1}}{\sigma-\beta}
\leq\frac{\sigma^{k+1}}{\sigma-\beta}.
}
\end{equation}
{For positive constants $M_U,M_V$ to be chosen,
suppose that $\|U_j\|\leq M_U$ and $\|V_j\|\leq M_V\sigma^j$
for $0\leq j\leq k$.
Iterating \eqref{eq:app-V} and using
\eqref{eq:app-geometric-convolution} yields}
\[
{
\begin{aligned}
\|V_{k+1}\|
&\leq\beta^{k+1}
+L_1\delta^{\holderexp}M_U
\sum_{j=0}^k\beta^{k-j}\beta^{\holderexp j}\\
&\leq\left(1+\frac{L_1\delta^{\holderexp}M_U}
{\sigma-\beta}\right)\sigma^{k+1}.
\end{aligned}
}
\]
{By \eqref{eq:app-first-blocks}, the products of
$1+\|E_j\|$ are bounded by
$\exp(L_1\delta^{\holderexp}/(1-\beta^{\holderexp}))$.
Iterating \eqref{eq:app-U} and summing the terms $R_jV_j$ therefore gives}
\[
{
\|U_{k+1}\|
\leq
\exp\left(\frac{L_1\delta^{\holderexp}}
{1-\beta^{\holderexp}}\right)
\left(1+\frac{\overline R M_V}{1-\sigma}\right).
}
\]
{These two estimates determine the choices needed to
close the induction. Take $M_V:=2$, choose
$M_U>2(1+\overline R M_V/(1-\sigma))$, and then reduce $\delta$ so that}
\begin{equation}\label{eq:app-first-smallness}
{
\frac{L_1\delta^{\holderexp}M_U}{\sigma-\beta}\leq1,
\qquad
\exp\left(\frac{L_1\delta^{\holderexp}}
{1-\beta^{\holderexp}}\right)\leq2.
}
\end{equation}
{The preceding estimates now give
$\|V_{k+1}\|\leq M_V\sigma^{k+1}$ and $\|U_{k+1}\|<M_U$.
The initial bounds thus imply, by induction,}
\begin{equation}\label{eq:app-first-induction}
{
\|U_k\|\leq M_U,\qquad
\|V_k\|\leq M_V\sigma^k,\qquad k\geq0.
}
\end{equation}
{Finally, \eqref{eq:app-U}, \eqref{eq:app-first-blocks},
and \eqref{eq:app-first-induction} give}
\[
{
\|U_{k+1}-U_k\|
\leq(L_1\delta^{\holderexp}M_U+\overline R M_V)\sigma^k.
}
\]
{Hence $U_k$ converges uniformly and $V_k\to0$
uniformly in the initial coordinates.}

\paragraph{{Second derivatives when $r=2$.}}
{Set $X_k:=D_y^2u_k$ and $Y_k:=D_y^2v_k$;
here $X_0=Y_0=0$.
Differentiating \eqref{eq:app-U}--\eqref{eq:app-V} once more gives}
\begin{align}
X_{k+1}
&=(I+E_k)X_k+R_kY_k+S_k^u, \label{eq:app-X}\\
Y_{k+1}
&=Q_kX_k+A_{H,k}Y_k+S_k^v, \label{eq:app-Y}
\end{align}
{where $H_u(u,v):=u+h(u,v)$ and $H_v(u,v):=G(u,v)$
are the components of \eqref{eq:app-coordinate-model}.
For initial-coordinate directions $\xi,\zeta$, the second-order chain
rule produces the source terms}
\[
{
S_k^a[\xi,\zeta]
=D^2H_a(u_k,v_k)
[(U_k\xi,V_k\xi),(U_k\zeta,V_k\zeta)],
\qquad a\in\{u,v\}.
}
\]
{Since $D^2H_u=D^2h$, the bounds
\eqref{eq:app-d2h}--\eqref{eq:app-d2g} control the $uu$, mixed, and $vv$ terms.
Using \eqref{eq:app-orbit-normal}, \eqref{eq:app-first-induction},
$\delta\leq1$, and $\beta^{\holderexp}<\sigma$ gives}
\[
{
\begin{aligned}
\|S_k^a\|
&\leq L_2\bigl(\|v_k\|^{\holderexp}\|U_k\|^2
+2\|U_k\|\|V_k\|+\|V_k\|^2\bigr)\\
&\leq L_2(M_U+M_V)^2\sigma^k.
\end{aligned}
}
\]
{Thus the fixed values
$K_u=K_v:=L_2(M_U+M_V)^2$ give}
\begin{align}
\|S_k^u\|&{\leq K_u\sigma^k,}
\label{eq:app-source-u}\\
\|S_k^v\|&{\leq K_v\sigma^k.}
\label{eq:app-source-v}
\end{align}
{These values remain valid under further reductions of
$\delta$, since $L_2,M_U,M_V$ are already fixed.
For positive $M_X,M_Y$ to be chosen, suppose
$\|X_j\|\leq M_X$ and $\|Y_j\|\leq M_Y\sigma^j$
for $0\leq j\leq k$.
Iterating \eqref{eq:app-Y} and applying
\eqref{eq:app-geometric-convolution} gives}
\[
{
\begin{aligned}
\|Y_{k+1}\|
&\leq L_1\delta^{\holderexp}M_X
\sum_{j=0}^k\beta^{k-j}\beta^{\holderexp j}
+K_v\sum_{j=0}^k\beta^{k-j}\sigma^j\\
&\leq
\frac{L_1\delta^{\holderexp}M_X+K_v}{\sigma-\beta}
\sigma^{k+1}.
\end{aligned}
}
\]
{For \eqref{eq:app-X}, the product bound in
\eqref{eq:app-first-smallness} and the summable terms
$R_jY_j+S_j^u$ give}
\[
{
\|X_{k+1}\|
\leq\frac{2(\overline R M_Y+K_u)}{1-\sigma}.
}
\]
{To close this induction, first choose
$M_Y>2K_v/(\sigma-\beta)$ and then
$M_X>2(\overline R M_Y+K_u)/(1-\sigma)$.
Reduce $\delta$ further so that
$L_1\delta^{\holderexp}M_X/(\sigma-\beta)\leq M_Y/2$.
The preceding estimates give $\|Y_{k+1}\|\leq M_Y\sigma^{k+1}$
and $\|X_{k+1}\|<M_X$, so the zero initial values imply}
\begin{equation}\label{eq:app-second-bootstrap}
{
\sup_{k\geq0}\|X_k\|\leq M_X,\qquad
\|Y_k\|\leq M_Y\sigma^k.
}
\end{equation}
{Now \eqref{eq:app-X}, \eqref{eq:app-first-blocks},
\eqref{eq:app-source-u}, and \eqref{eq:app-second-bootstrap} yield}
\[
{
\|X_{k+1}-X_k\|
\leq(L_1\delta^{\holderexp}M_X+\overline R M_Y+K_u)\sigma^k.
}
\]
{Thus $X_k$ converges uniformly and $Y_k\to0$
uniformly in the initial coordinates.}

\paragraph{{Smoothness of the limiting map.}}
{It remains to identify the derivative limits above with
the derivatives of $\Psi_H$.
Regard $(u_k,v_k)$ as maps of the full initial coordinates $y=(u_0,v_0)$.
Around any point of the final $W_\ell$, choose a compact convex ball
$\mathcal Q\Subset W_\ell$ containing that point in its interior.
By \Cref{lem:contracting-iterate} and the estimates above, these maps
and their derivatives through order $r$ converge uniformly on $\mathcal Q$.
Apply the fundamental theorem of calculus to the coordinate iterates
and, when $r=2$, to their first derivatives.
This identifies the derivative limits with the derivatives of
$(u_\infty,0)$, so this map is jointly $C^r$ in $(u_0,v_0)$.
By definition of the coordinate iterates and
\eqref{eq:app-coordinate-limit}, on $\Theta_\ell(W_\ell)$ we have}
\[
{
H^k=\Theta_\ell\circ(u_k,v_k)\circ\Theta_\ell^{-1},
\qquad
\Psi_H=\Theta_\ell\circ(u_\infty,0)\circ\Theta_\ell^{-1}.
}
\]
{These formulas represent the same ambient maps on chart
overlaps. The chain rule and the fixed chart bounds therefore give
$\Psi_H\in C^r(\mathcal V_0,\E)$ and the locally uniform derivative
convergence in \eqref{eq:app-derivative-convergence}.}
\end{proof}

\subsection{Proof of Lemma~\ref{lem:original-iteration}}
\label{app:proof-original-iteration}

\begin{proof}
{We first recover the intermediate $F$-steps between
successive $H=F^N$ iterations. Write $k=Nn+j$, $0\leq j<N$.
By \eqref{eq:app-forward-domain}, $H^n(y)\in\mathcal V_0$ for
$y\in\mathcal V_0$. The domain assertion in
\Cref{lem:contracting-iterate} therefore gives
$F^k(y)=F^j(H^n(y))$, so all $F$-iterates are defined.
Since $F^j$ fixes $\Psi_H(y)\in\M$, continuity gives}
\[
{
F^{Nn+j}(y)=F^j(H^n(y))\longrightarrow
F^j(\Psi_H(y))=\Psi_H(y).
}
\]
{There are only finitely many residue classes $j$;
hence the full sequence converges to $\Psi_F(y)=\Psi_H(y)$.}

{We next estimate the errors at these intermediate steps.
The finitely many maps $F^j\circ\Theta_\ell$ have continuous derivatives
on the fixed compact coordinate domains from the proof of
\Cref{lem:contracting-iterate}. Thus a common $L\geq1$ satisfies}
\begin{equation}\label{eq:app-finite-iterate-derivative-bound}
{
\|D(F^j\circ\Theta_\ell)(u,v)\|\leq L,
\qquad (u,v)\in\overline{\mathcal O_\ell}\times
\overline{B_{\delta_0}(0)},\quad 0\leq j<N.
}
\end{equation}
{Choose a chart with $y=\Theta_\ell(u_0,v_0)$ and
$(u_0,v_0)\in W_\ell$. By \eqref{eq:app-coordinate-limit}, write
$H^n(y)=\Theta_\ell(u_n,v_n)$ and
$\Psi_H(y)=\Theta_\ell(u_\infty,0)$.
The segment joining these coordinate points lies in the convex cylinder
$\mathcal O_\ell\times B_\delta(0)$.
}
\[
{
\begin{aligned}
\|F^{Nn+j}(y)-\Psi_H(y)\|
&=\left\|\int_0^1 D(F^j\circ\Theta_\ell)
\bigl(u_\infty+t(u_n-u_\infty),t v_n\bigr)
\binom{u_n-u_\infty}{v_n}\,dt\right\|\\
&\overset{\eqref{eq:app-finite-iterate-derivative-bound}}{\leq}
L\bigl(\|u_n-u_\infty\|+\|v_n\|\bigr)\\
&\overset{\eqref{eq:app-coordinate-tail}}{\leq}
L\left(1+\frac{L_h}{1-\beta}\right)\beta^n\|v_0\|\\
&\overset{\eqref{eq:tubular-chart},\,\eqref{eq:tubular-projection}}{=}
L\left(1+\frac{L_h}{1-\beta}\right)
\beta^n\dist(y,\M).
\end{aligned}
}
\]
{Finally, the choice $\beta<\widehat q^{\,N}$ in
\eqref{eq:app-contracting-iterate-choice} converts this bound into
a rate per $F$-step:}
\[
{
\beta^n\leq\widehat q^{\,Nn}
=\widehat q^{-j}\widehat q^{\,k}
\leq\widehat q^{-(N-1)}\widehat q^{\,k}.
}
\]
{This proves \eqref{eq:app-original-rate} with
$C_{\widehat q}=L(1+L_h/(1-\beta))\widehat q^{-(N-1)}$.
Different choices of $N$ and $\widehat q$ give the limit of the same
full sequence $F^k(y)$, so their limiting maps agree wherever their
domains overlap.}
\end{proof}

%% file: appendix_inexact_zero_coupling.tex
\section{A sufficient condition for vanishing coupling in projection-type compositions}
\label[appendix]{app:iap-zero-coupling}

We give a sufficient condition for the normal-to-tangential block to vanish.
It does not require the inexact maps to have the same first-order jet as the
metric projections: instead, at points of $\M$, their derivatives preserve
the individual normal spaces.

\begin{proposition}[Vanishing coupling for projection-type compositions]
\label{prop:iap-zero-coupling}
Suppose that each $\phi_i$ fixes $\M_i$ and that, for $x\in\M$,
\begin{equation}\label{eq:iap-normal-space-invariance}
  D\phi_i(x)N_x\M_i\subseteq N_x\M_i,
  \qquad i=1,2.
\end{equation}
Then $D\phi_i(x)$ preserves $N_x\M$, and the composition
$S=\phi_1\circ\phi_2$ satisfies
\begin{equation}\label{eq:iap-zero-coupling}
B_x^{\rm iAP}=0.
\end{equation}
\end{proposition}

\begin{proof}
Since $\phi_i|_{\M_i}=\Id$, its derivative is the identity on
$T_x\M_i$. Let $n\in N_x\M$ and write $n=t_i+n_i$ for its orthogonal
decomposition relative to $T_x\M_i\oplus N_x\M_i$. Because
$\M\subseteq\M_i$, we have $T_x\M\subseteq T_x\M_i$ and hence
$N_x\M_i\subseteq N_x\M$. Thus $n_i\in N_x\M$, and the assumed
invariance {in \eqref{eq:iap-normal-space-invariance}} gives $D\phi_i(x)n_i\in N_x\M_i\subseteq N_x\M$.
Also, $t_i=P_{T_x\M_i}n$ is orthogonal to $T_x\M$, so
$t_i\in T_x\M_i\cap N_x\M$. Both terms in
$D\phi_i(x)n=t_i+D\phi_i(x)n_i$ therefore belong to $N_x\M$.
{By the block definition \eqref{eq:normal-tangent-blocks},
the composition's invariance of $N_x\M$ gives \eqref{eq:iap-zero-coupling}.}
\end{proof}

%% file: appendix_gap_analysis.tex
\section{Principal-angle analysis of relaxed and generalized alternating projections}\label[appendix]{app:gap-analysis}

This appendix gives the spectral ingredients used in \Cref{sec:gap} in the
normal-block notation of the main text. Relaxed alternating projections have
an established literature; see, for example, \cite{CegielskiSuchocka2008}.
The principal-angle spectral analysis of generalized alternating projections
(GAP) for two subspaces appears in \cite{FaltGiselsson2017}, and the local
manifold analysis in \cite{FaltGiselsson2023}. We cite the principal-angle
spectrum and prove the uniform normal stability and exact positive
outer-relaxation interval needed here.

Throughout, we retain the standing hypotheses of \Cref{sec:gap}: $\Mone$,
$\Mtwo$, and their intersection $\M$ are $C^{r+1,\holderexp}$, the intersection
is clean on a relative neighborhood of compact $K\subset\M$, and
both local metric projections are single-valued with the stated regularity on
a common ambient neighborhood.

\subsection{Principal-angle spectrum and inner stability}

Fix $x\in\M$ and abbreviate
\[
U=T_x\Mone,\qquad V=T_x\Mtwo,\qquad W=U\cap V=T_x\M.
\]
We use the principal angles of the reduced spaces from
\Cref{def:friedrichs-principal}.

\begin{lemma}[Principal-angle spectrum of the two-parameter relaxed AP map]
\label[lemma]{prop:gap-block}
On a principal two-plane associated with a nontrivial principal angle in
$W^\perp$, choose an orthonormal basis such that
\[
P_U=\begin{pmatrix}1&0\\0&0\end{pmatrix},
\qquad
P_V=\begin{pmatrix}c^2&cs\\cs&s^2\end{pmatrix},
\qquad
c=\cos\theta,\quad s=\sin\theta,\quad
0<\theta\leq\frac{\pi}{2}.
\]
For the projection order
$R_{\alpha_1}^{(1)}\circ R_{\alpha_2}^{(2)}$, the corresponding normal block is
\begin{equation}\label{eq:gap-principal-block}
A_\theta(\alpha_1,\alpha_2)
=
\begin{pmatrix}
1-\alpha_2s^2&\alpha_2cs\\
(1-\alpha_1)\alpha_2cs&(1-\alpha_1)(1-\alpha_2c^2)
\end{pmatrix}.
\end{equation}
Its characteristic polynomial is
\begin{equation}\label{eq:gap-characteristic-polynomial}
z^2-\tau_\theta z+\delta,
\qquad
\tau_\theta=2-\alpha_1-\alpha_2+\alpha_1\alpha_2c^2,
\qquad
\delta=(1-\alpha_1)(1-\alpha_2),
\end{equation}
and hence
\begin{equation}\label{eq:gap-principal-eigenvalues}
\lambda_{\theta,\pm}
=\frac{\tau_\theta\pm\sqrt{\tau_\theta^2-4\delta}}{2},
\end{equation}
where the square root may be complex. The remaining normal components have
eigenvalues
\begin{equation}\label{eq:gap-transverse-eigenvalues}
1-\alpha_2\quad\text{on }U\cap V^\perp,
\qquad
1-\alpha_1\quad\text{on }U^\perp\cap V,
\end{equation}
and
\begin{equation}\label{eq:gap-complement-eigenvalue}
(1-\alpha_1)(1-\alpha_2)
\quad\text{on }(U+V)^\perp.
\end{equation}
The common component $W$ is the tangential identity block.
\end{lemma}

{The eigenvalues in \eqref{eq:gap-principal-eigenvalues},
\eqref{eq:gap-transverse-eigenvalues}, and
\eqref{eq:gap-complement-eigenvalue} form} the principal-angle spectrum of the two-parameter relaxed AP map in
our projection order and tangent--normal convention; see
\cite{FaltGiselsson2017,FaltGiselsson2023}. Reversing the two relaxed
projection operators changes the {matrix block
\eqref{eq:gap-principal-block}}, but not its
characteristic polynomial {\eqref{eq:gap-characteristic-polynomial}}. The spectral conclusions below are therefore
unchanged.

\begin{lemma}[Uniform normal stability of the two-parameter relaxed AP map]
\label[lemma]{thm:gap-inner-stability}
Suppose
\begin{equation}\label{eq:gap-inner-parameter-range}
0<\alpha_1<2,\qquad 0<\alpha_2<2.
\end{equation}
Then every eigenvalue of the normal block lies strictly inside the unit disk.
Consequently,
\begin{equation}\label{eq:gap-core-stability}
\spr(A_x^{\rm core})<1,
\qquad
\rho_{\rm core}:=\sup_{x\in K}\spr(A_x^{\rm core})<1.
\end{equation}
\end{lemma}

\begin{proof}
For a principal block, the classical Jury--Schur criterion for the real monic
quadratic $z^2-\tau_\theta z+\delta$ reduces to
\begin{equation}\label{eq:gap-jury-schur}
1-\tau_\theta+\delta>0,\qquad
1+\tau_\theta+\delta>0,\qquad
1-\delta>0.
\end{equation}
{By \eqref{eq:gap-characteristic-polynomial}},
\begin{align*}
1-\tau_\theta+\delta
&=\alpha_1\alpha_2s^2>0,\\
1+\tau_\theta+\delta
&=4-2\alpha_1-2\alpha_2+\alpha_1\alpha_2(1+c^2)\\
&=4\bigl((1-\alpha_1/2)(1-\alpha_2/2)
+c^2(\alpha_1/2)(\alpha_2/2)\bigr)>0.
\end{align*}
Moreover, $1-\alpha_i\in(-1,1)$, so
$\delta=(1-\alpha_1)(1-\alpha_2)<1$. Thus {the inequalities in
\eqref{eq:gap-jury-schur} hold, and} every principal-block eigenvalue
lies in the open unit disk. The remaining eigenvalues listed in
{\eqref{eq:gap-transverse-eigenvalues} and
\eqref{eq:gap-complement-eigenvalue}} have modulus strictly below one as well.

Finally, $x\mapsto A_x^{\rm core}$ and the finite-dimensional spectral radius
are continuous. Compactness of $K$ makes the strict spectral-radius bound
uniform on $K$.
\end{proof}

\begin{remark}
The range $0<\alpha_1,\alpha_2<2$ {in
\eqref{eq:gap-inner-parameter-range}} gives a sufficient uniform spectral-stability
region; we do not claim it is necessary or parameter-optimal. If either inner
relaxation parameter exceeds $1$, the principal eigenvalues {in \eqref{eq:gap-principal-eigenvalues}} may be complex.
A real-interval Chebyshev design then requires additional spectral information.
\end{remark}

\begin{corollary}[Retraction induced by the two-parameter relaxed AP map]
\label[corollary]{cor:gap-retraction}
Under the standing geometric hypotheses and the hypotheses of
\Cref{thm:gap-inner-stability}, there is a neighborhood $V$ of $K$ on which
\[
y_{k+1}=S_{\alpha_1,\alpha_2}(y_k)
\]
converges to a $C^r$ limiting map
\[
\Psi_{\rm core}(y):=\lim_{k\to\infty}S_{\alpha_1,\alpha_2}^k(y).
\]
For every $\widehat q\in(\rho_{\rm core},1)$, after shrinking $V$ if
necessary,
\[
\|S_{\alpha_1,\alpha_2}^k(y)-\Psi_{\rm core}(y)\|
\leq C_{\widehat q}\widehat q^{\,k}\dist(y,\M).
\]
The map
\[
\Retr_x^{\rm core}(\eta):=\Psi_{\rm core}(x+\eta)
\]
is a local retraction. {Since \eqref{eq:gap-core-linearization}
gives $B_x^{\rm core}=0$, when $r=2$ the retraction is second-order and,
for every compact}
$K_0\subset\M\cap V$, there are $C,\varepsilon>0$ such that
\[
\|\Psi_{\rm core}(x+\eta)-P_\M(x+\eta)\|
\leq C\|\eta\|^4
\]
whenever $x\in K_0$, $\eta\in T_x\M$, and $\|\eta\|<\varepsilon$.
Thus the limiting map has fourth-order agreement with the projection
retraction for tangent perturbations.
\end{corollary}

\begin{proof}
Apply {\Cref{thm:main-polynomial,thm:limiting-map}} to
$S=S_{\alpha_1,\alpha_2}$ and the identity polynomial
$p(\lambda)=\lambda$, using
\Cref{prop:gap-linearization,thm:gap-inner-stability}. {The vanishing coupling obtained from
\eqref{eq:gap-core-linearization}} then permits application of
{\Cref{thm:second-order}}.
\end{proof}

\subsection{Outer relaxation}

{The full GAP map and its affine polynomial
$p_\beta(\lambda)=1-\beta+\beta\lambda$ are defined in
\eqref{eq:gap-outer-map}. Since $p_\beta(1)=1$, outer relaxation
preserves the fixed-point manifold.}
Three-parameter GAP and its
spectral rate optimization for two subspaces were studied in
\cite{FaltGiselsson2017}. Here the inner parameters are fixed, and we need the
exact positive outer-relaxation interval for the compact normal spectral
family arising along $K$.

Let
\begin{equation}\label{eq:gap-core-spectrum}
\Sigma_{\rm core}:=\bigcup_{x\in K}\spec(A_x^{\rm core})
\end{equation}
and define
\begin{equation}\label{eq:gap-outer-threshold}
\beta_{\max}:=
\inf_{\lambda\in\Sigma_{\rm core}}
\frac{2(1-\operatorname{Re}\lambda)}{|1-\lambda|^2}.
\end{equation}
We use the convention $\inf\varnothing=+\infty$; if the normal spaces have
dimension zero on $K$, every stability conclusion below is vacuous.

\begin{theorem}[Exact Schur-stability interval for positive outer relaxation]
\label{thm:gap-outer-stability}
Under the hypotheses of \Cref{thm:gap-inner-stability},
$\Sigma_{\rm core}$ {from \eqref{eq:gap-core-spectrum}} is a compact subset of the open unit disk and
\[
\beta_{\max}>1.
\]
For $\beta>0$, the transformed normal-block family
$p_\beta(A_x^{\rm core})$ is uniformly spectrally stable if and only if
\begin{equation}\label{eq:gap-outer-stability-interval}
0<\beta<\beta_{\max}.
\end{equation}
Here exactness refers to positive scalar outer relaxation for fixed inner
parameters $\alpha_1,\alpha_2$ and the compact normal-block family over $K$;
it does not assert joint parameter optimality, a fastest convergence rate, or
an operator-norm optimum.
\end{theorem}

\begin{proof}
If $\Sigma_{\rm core}=\varnothing$, the normal bundle has rank zero over $K$
and the stability statement is vacuous. Otherwise, continuity of the bundle
endomorphism over compact $K$ makes $\Sigma_{\rm core}$ compact. By
{\eqref{eq:gap-core-stability}}, it is contained in a closed disk
$|\lambda|\leq\rho_{\rm core}<1$.

For fixed $\lambda\neq1$ and $\beta>0$,
\[
|1-\beta+\beta\lambda|^2
=1+2\beta(\operatorname{Re}\lambda-1)
+\beta^2|1-\lambda|^2.
\]
Hence
\begin{equation}\label{eq:gap-scalar-outer-stability}
|1-\beta+\beta\lambda|<1
\quad\Longleftrightarrow\quad
0<\beta<\frac{2(1-\operatorname{Re}\lambda)}{|1-\lambda|^2}.
\end{equation}
The threshold function is continuous on the compact set
$\Sigma_{\rm core}$, and
\[
\frac{2(1-\operatorname{Re}\lambda)}{|1-\lambda|^2}-1
=\frac{1-|\lambda|^2}{|1-\lambda|^2}>0.
\]
Thus its minimum {in \eqref{eq:gap-outer-threshold}} is attained and $\beta_{\max}>1$. The spectral mapping theorem gives
\[
\spec(p_\beta(A_x^{\rm core}))
=p_\beta(\spec(A_x^{\rm core})),
\]
so the scalar equivalence {\eqref{eq:gap-scalar-outer-stability}} gives uniform spectral stability exactly for
$0<\beta<\beta_{\max}$. At $\beta=\beta_{\max}$, an eigenvalue attaining the
minimum maps to the unit circle, and for $\beta>\beta_{\max}$ the same
eigenvalue maps outside it.
\end{proof}

In particular, $\beta_{\max}>1$ implies that every convex-combination outer
relaxation $0<\beta\leq1$ is uniformly spectrally stable in the normal block.

\begin{corollary}[Retraction induced by the generalized alternating projection map]
\label[corollary]{cor:outer-gap-retraction}
Under the standing geometric and local-projection hypotheses of this appendix,
if
\[
0<\alpha_1,\alpha_2<2,
\qquad
0<\beta<\beta_{\max},
\]
then there is a neighborhood $V$ of $K$ on which
$G_{\alpha_1,\alpha_2,\beta}^k(y)$ converges to a $C^r$ limiting map
$\Psi_{\rm out}(y)$. With
\[
\rho_{\rm out}:=
\sup_{x\in K}\spr(p_\beta(A_x^{\rm core}))<1,
\]
every $\widehat q\in(\rho_{\rm out},1)$ is, after shrinking $V$ if necessary,
an admissible local R-linear convergence factor. The limiting map induces a
local retraction. Since {\eqref{eq:gap-core-linearization}} gives
$B_x^{\rm core}=0$ and a consistent polynomial transformation preserves the
identity $B_x=0$, {when $r=2$ the induced retraction is
second-order and, for every compact} $K_0\subset\M\cap V$, there are $C,\varepsilon>0$ such that
\[
\|\Psi_{\rm out}(x+\eta)-P_\M(x+\eta)\|
\leq C\|\eta\|^4
\]
whenever $x\in K_0$, $\eta\in T_x\M$, and $\|\eta\|<\varepsilon$. Thus the
full GAP map has fourth-order agreement with the projection retraction for
tangent perturbations.
\end{corollary}

\begin{proof}
By {the stability criterion
\eqref{eq:gap-outer-stability-interval} in} \Cref{thm:gap-outer-stability},
\[
\sup_{x\in K}\spr(p_\beta(A_x^{\rm core}))<1.
\]
Apply {\Cref{thm:main-polynomial,thm:limiting-map}} with base map
$S_{\alpha_1,\alpha_2}$ and polynomial $p_\beta$. The higher-order conclusions
follow from {\eqref{eq:gap-core-linearization} and}
\Cref{thm:fourth-order-agreement}.
\end{proof}

%% file: appendix_numerics.tex
\section{Numerical protocol and supplementary results}
\label[appendix]{app:numerics}

\subsection{Controlled experiments}

{This appendix records the sampling and solver details
for the comparisons in \Cref{sec:numerics}. For
\eqref{eq:numerical-sphere-affine},}
let $a=d^{-1/2}(1,\ldots,1)^T$ and generate $x$ by normalizing a Gaussian vector
in $a^\perp$, scaling it by $\sqrt{1-c_F^2}$, and adding $c_Fa$.
{The finite-difference measurement is
\eqref{eq:factor-finite-difference}.}
We use $d=100$, $h=10^{-7}$,
$c_F\in\{0.5,0.8,0.95,0.99,0.999\}$,
$m\in\{1,2,4,8\}$, and 20 seeds. The factorized realization uses
Leja-ordered roots; the {affine-combination} realization evaluates the same polynomial
in the power basis.

To probe the Friedrichs-angle scaling as $\theta_F\to0$ predicted by the
{limiting normal spectral factor}, set $d=50$,
$\theta_F\in\{0.4,0.2,0.1,0.05,0.025\}$, and
$m=\lceil1/\theta_F\rceil$. Starting from $x+10^{-8}v$, each method reduces
the distance to the known intersection by $10^{-4}$, subject to an absolute
floor $10^{-13}$ and an evaluation budget of $10^5$ base AP evaluations.
All runs for the 20 seeds reach the target. {In the listed
decreasing order of the five angles, the median
AP/Chebyshev base AP evaluation counts are}
\[
 (57,229,920,3683,14733)/(18,35,70,140,280),
\]
and least-squares fits of the log evaluation count against log $\theta_F$
give slopes
$-2.0035$ and $-0.9919$.

For the retraction-property diagnostics on
\eqref{eq:numerical-parabolic-surfaces}, the local projections solve their
stationarity equations to residuals smaller than the errors used in the fits. The limit
tolerance is $10^{-14}$, and the eight tangent steps are
\[
0.2,0.112468,0.063246,0.035566,0.02,0.011247,0.006325,0.003557.
\]
The tangent sign is selected independently for 20 seeds, and all steps enter
the fits. The fitted $(e_2,e_4)$ orders are AP $(2.993,4.009)$; factorized
realizations of degree 2 and 4, $(2.998,3.969)$ and $(3.001,3.963)$; and
{affine-combination} realizations of degree 2 and 4, $(2.992,4.057)$ and
$(2.993,4.004)$. Differences between the two realizations have fitted orders
$3.944$ and $3.948$. All 25 predeclared factor, base AP evaluation, and
retraction-property diagnostics fall within their prescribed ranges.

\subsection{FUNTF protocol}

{The projections and stopping residual used for
\eqref{eq:funtf-feasibility} are defined in
\eqref{eq:frame-tightness-projection} and \eqref{eq:frame-residual}.}
{Construction uses $(d,N)=(16,32)$, $(32,64)$, and
$(64,128)$. Completion uses $(d,N,p)=(32,64,16)$ and $(64,128,32)$.} For completion, a reference FUNTF is
first generated to residual $10^{-12}$, its first $p$ columns are fixed, and
the remaining columns are initialized as normalized Gaussian vectors. All
methods share each seed's initial matrix.

After ten base AP evaluations included in the reported cost, the common
heuristic estimator is
\begin{equation}
 \widehat\rho=\min\{0.9999,\max\{10^{-6},
 1.02\operatorname{median}(r_k/r_{k-1})\}\},
 \label{eq:rho-estimator}
\end{equation}
where $r_k:=r(X_k)$ is the residual {defined in
\eqref{eq:frame-residual}} after the $k$th counted base AP evaluation,
and the median uses at most five recent positive finite ratios. This
residual-ratio estimate is not a certified spectral bound; away from the local
asymptotic regime it may under- or overestimate the true uniform normal
spectral radius.
{R-AP relaxes the full AP map $S_C$ as
$(1-\omega)\Id+\omega S_C$, with
$\omega=2/(2-\widehat\rho)$ chosen by
\eqref{eq:optimal-degree-one-relaxation} for the interval estimated
by \eqref{eq:rho-estimator}}; the degree-4 and degree-8 periodic schemes use
Leja-ordered roots and restart. A nonfinite residual or a residual greater
than $1.2$ times the previous residual rejects the proposed periodic scheme.
The solver then falls back to AP, and all attempted base AP evaluations are
counted.
The evaluation budget is $5\times10^4$.

There are 20 seeds for each of five problems and four methods (400 runs).
Timing is the median of three deterministic repetitions; counts and times
are aggregated by medians. Paired speedup intervals are the 2.5\% and 97.5\%
percentiles of 5000 bootstrap resamples (seed 20260912). Each reported base
AP evaluation for these frame methods performs one thin SVD, so the base AP
evaluation and SVD counts coincide. SVD count is the primary computational
metric. Timings depend on the implementation, and the fixed method order
makes them secondary descriptive evidence.

\subsection{Computational environment}

The frame experiments use Python 3.10.5 on 64-bit Windows 10 (Intel64 Family
6 Model 140, eight logical CPUs), NumPy 1.26.0 with OpenBLAS64 0.3.23.dev,
SciPy 1.15.2, pandas 2.2.3, Matplotlib 3.10.3, and PyYAML 6.0.2. Runs use one
process and one thread for both OpenBLAS and OpenMP. Exact versions and
machine metadata are stored in \path{experiments/requirements-lock.txt} and
each result directory.

%% file: appendix_chebyshev_formulas.tex
\section{{Explicit formulas for Chebyshev polynomials}}
\label[appendix]{app:chebyshev-formulas}

{For each nonnegative integer $m$, $T_m$ denotes the Chebyshev
polynomial of the first kind \cite{MasonHandscomb2003,Saad2003}.
The initial values and recurrence are}
\begin{equation}
\label{eq:chebyshev-recurrence}
{
T_0(t)=1,\qquad T_1(t)=t,\qquad
T_{m+1}(t)=2tT_m(t)-T_{m-1}(t),\quad m\geq1.
}
\end{equation}
{Thus $T_2(t)=2t^2-1$, $T_3(t)=4t^3-3t$, and
$T_4(t)=8t^4-8t^2+1$. Equivalent formulas for real $t$ are}
\begin{equation}
\label{eq:chebyshev-real-formulas}
{
T_m(t)=
\begin{cases}
\cos(m\arccos t),& |t|\leq1,\\[2pt]
\dfrac{(t+\sqrt{t^2-1})^m+(t-\sqrt{t^2-1})^m}{2},& |t|>1.
\end{cases}
}
\end{equation}
{In the second line, the square root is the nonnegative real root;
the formula also applies for $t<-1$. In particular,
$|T_m(t)|=\cosh(m\operatorname{arcosh}|t|)>1$ when $m\geq1$ and $|t|>1$.
An explicit power expansion for $m\geq1$ is}
\begin{equation}
\label{eq:chebyshev-power-expansion}
{
T_m(t)=\frac{m}{2}
\sum_{j=0}^{\lfloor m/2\rfloor}
(-1)^j\frac{(m-j-1)!}{j!(m-2j)!}(2t)^{m-2j}.
}
\end{equation}
{For $m=0$, use $T_0=1$ instead of
\eqref{eq:chebyshev-power-expansion}.}